%% file: preprint.tex
\documentclass[twoside]{article}

\usepackage[preprint]{aistats2026}
\usepackage[round]{natbib}

\input{preamble}
\input{math_commands}

\begin{document}

%

%
\runningauthor{Abdurakhmon Sadiev, Yury Demidovich, Igor Sokolov, Grigory Malinovsky, Sarit Khirirat,   Peter Richtárik}

\twocolumn[

\aistatstitle{Improved Convergence in Parameter-Agnostic Error Feedback through Momentum}

\aistatsauthor{Abdurakhmon Sadiev${^{*}}$ \And Yury Demidovich \And Igor Sokolov  }

\aistatsauthor{ Grigory Malinovsky \And Sarit Khirirat${^{*}}$ \And   Peter Richtárik}

\aistatsaddress{ \\
King Abdullah University of Science and Technology\\ Center
of Excellence for Generative AI\\ Thuwal, Saudi Arabia
} 
]


\footnotetext[1]{{*} Corresponding authors: \texttt{abdurakhmon.sadiev}, \texttt{sarit.khirirat@kaust.edu.sa}}

\begin{abstract}
Communication compression is essential for scalable distributed training of modern machine learning models, but it often degrades convergence due to the noise it introduces. 
Error Feedback (EF) mechanisms are widely adopted to mitigate this issue of distributed compression algorithms. 
Despite their popularity and training efficiency, existing distributed EF algorithms often require prior knowledge of problem parameters (e.g., smoothness constants) to fine-tune stepsizes. 
This limits their practical applicability especially in large-scale  neural network training. 
In this paper, we study normalized error feedback algorithms that combine EF with normalized updates,  various momentum variants, and parameter-agnostic, time-varying stepsizes, thus eliminating the need for problem-dependent tuning.
We analyze the convergence of these algorithms for minimizing smooth functions, and establish parameter-agnostic complexity bounds that are close to the best-known bounds with carefully-tuned problem-dependent stepsizes. 
Specifically, we show that normalized EF21 achieve the convergence rate of near $\cO(1/T^{1/4})$ for Polyak's heavy-ball momentum,  $\cO(1/T^{2/7})$ for Iterative Gradient Transport (IGT), and  $\cO(1/T^{1/3})$ for STORM and Hessian-corrected momentum. 
Our results hold with decreasing stepsizes and small mini-batches.
Finally, our empirical experiments confirm our theoretical insights. 
\end{abstract}


\section{Introduction}

Distributed optimization has become essential for efficiently training modern machine learning models on large-scale datasets. 
This shift is driven by the increasing size of models—such as deep neural networks with billions of parameters—and the increasing volume of training data~\citep{Kolesnikov2019BigT,LMFewShot}. In distributed settings, multiple clients collaborate in parallel, sharing local information (e.g., stochastic gradients) with a central server to jointly minimize the average of their objective functions, thereby keeping training computationally feasible. However, a major challenge in this paradigm is the communication bottleneck, which becomes particularly severe with large models. For instance, transmitting the VGG-16 model~\citep{simonyan2014very}, which contains 138.34 million parameters, requires over 500 MB of data per exchange—placing a significant strain on network resources during distributed stochastic gradient descent (SGD), which aims to find an $\epsilon$-approximate critical point $x^t$ such that $\Exp{\norm{\nabla f(x^t)}} \leq \epsilon$.

A popular strategy to reduce communication overhead is  compression, where clients apply compression operators to their local gradients before sending them to the server. Contractive but potentially biased compressors—including many widely used compressors  from sparsification~\citep{alistarh2018convergence,StitchSparseSGD} to quantization~\citep{wen2017terngrad,alistarh2017qsgd,beznosikov2023biased,horvoth2022natural}—have been shown to have favorable theoretical properties~\citep{StitchSparseSGD,karimireddy2019error,EF21,fatkhullin2025ef21}, and to outperform randomized (unbiased) compressors in practice for distributed gradient-based algorithms~\citep{lin2017deep,sun2019sparse,vogels2019powersgd}. However, naively aggregating these biased compressed gradients from the clients in general does not converge~\citep{khirirat2020compressed} or even  diverges~\citep{karimireddy2019error,beznosikov2023biased}.

To improve the convergence stability of distributed compression methods, error feedback (EF) mechanisms have been proposed. A widely studied EF variant, \algname{EF14}~\citep{Seide20141bitSG}, has been extensively explored in both centralized~\citep{StitchSparseSGD,karimireddy2019error} and distributed optimization settings~\citep{alistarh2018convergence,tang2019doublesqueeze,tang2019texttt,gorbunov2020linearly,khirirat2020compressed,qian2021error}. However, many of these works rely on the assumption that the norms of the stochastic gradients are uniformly bounded, which restricts the class of objective functions the EF algorithms can solve.
Under this assumption, distributed EF algorithms have been shown to minimize smooth nonconvex functions with a convergence rate of $\cO(1/T^{1/3})$ in the gradient norm, where $T$ denotes the total number of iterations~\citep{koloskova2019decentralized}.

To further improve theoretical convergence guarantees, many novel EF variants have been developed, including \algname{EControl}~\citep{gao2024econtrol} and \algname{EF21}~\citep{EF21}. Notably, \algname{EF21} achieves a convergence rate of $\cO(1/\sqrt{T})$ in the gradient norm for deterministic optimization problems—matching the rate of classical gradient descent. This method has also been extended to broader settings. 
One extension is \algname{EF21-P}~\citep{EF21-P}, which adapts \algname{EF21} for federated optimization, where model parameters (rather than gradients) are exchanged between clients and the server. 
Another extension is \algname{EF21-SGDM}~\citep{fatkhullin2023momentum,khirirat2024errorfeedbackl0l1smoothnessnormalization}, which incorporates Polyak momentum into the original \algname{EF21} algorithms to handle stochastic optimization, where clients compute local stochastic gradients. \algname{EF21-SGDM} achieves the convergence rate of $\cO(1/{T}^{1/4})$. 
To further accelerate the convergence, \algname{EF21-MVR}~\citep{fatkhullin2023momentum} integrates the momentum-with-variance-reduction (MVR) update from STORM~\citep{cutkosky2019momentum}, and enjoys the improved convergence rate of $\cO(1/T^{1/3})$. 


The smoothness assumption of objective functions plays a key role in fine-tuning stepsizes  to ensure significant convergence improvements in distributed EF algorithms. However, stepsize rules that depend on smoothness are often impractical to implement. For example, estimating smoothness constants is typically infeasible, especially in deep neural network training. This motivates the need for \emph{parameter-agnostic} stepsize rules that do not depend on the smoothness constants. Crucially, it is important to design these stepsize rules so that they still achieve \emph{near-optimal} convergence, implying the performance of distributed EF algorithms with parameter-agnostic stepsizes  almost matches that of EF algorithms with optimally tuned, problem-dependent stepsizes.


One  approach for incorporating parameter-agnostic stepsizes into gradient-based algorithms—while still achieving near-optimal convergence—is through \emph{normalization}. For example, normalized stochastic momentum methods~\citep{cutkosky2020momentum} have been shown by~\citet{hubler2024parameter} to attain a near-optimal convergence rate of $\tilde{\cO}(1/{T}^{1/4})$ using parameter-agnostic, decreasing stepsize rules. However, to the best of our knowledge, this normalization approach has been so far limited to centralized algorithms.

\section{Contributions}

We summarize our  key contributions as follows:

\renewcommand\labelitemi{\ding{117}}
\begin{itemize}
    \item \textbf{Distributed EF21 algorithms with parameter-agnostic stepsizes and five momentum variants.} 
    In~\Cref{sec:ef21_momentum}, 
    we propose distributed \algname{EF21} algorithms that exploit normalization and momentum for solving stochastic optimization. Specifically, our algorithms 
    employ normalization, which enables parameter-agnostic stepsizes without requiring the knowledge of problem parameters, such as the smoothness constant $L$ or the suboptimality gap $f(x^0) - f^{\inf}$.
    Furthermore, our algorithms leverage {\color{black} five variants of momentum updates widely adopted in centralized stochastic gradient algorithms: (1) Polyak momentum, (2) IGT momentum~\citep{cutkosky2020momentum}, (3) MVR momentum~\citep{cutkosky2019momentum}, (4) two second-order momentum schemes by~\citet{salehkaleybar2022momentum,tran2022better}.

    } 
    
    \item \textbf{Near-optimal convergence for non-convex, smooth functions.} 
    In~\Cref{sec:convergence}, we prove that our proposed algorithms, using parameter-agnostic and decreasing stepsizes, achieve near-optimal convergence rates for minimizing non-convex smooth functions.
    Our results match—up to logarithmic factors—the convergence guarantees of momentum-based \algname{EF21} algorithms from prior work, which typically rely on problem-dependent stepsizes.
     They also align with the rates achieved by centralized stochastic methods using analogous momentum variants. 
     Furthermore, our algorithms only require a batch size of $\cO\br{1}$, in contrast to existing EF algorithms such as \citet{fatkhullin2023momentum}, which often rely on problem-dependent stepsizes and larger mini-batch sizes.
    A summary of theoretical comparisons between our results and existing analyses in provided in~\Cref{table:complexities-ef}.
    

    \begin{table*}
	\centering
	\caption{A comparison of distributed error feedback methods using contractive compressors for stochastic optimization under data heterogeneity. In the table, \textbf{PA} indicates whether the methods use parameter-agnostic stepsizes,  \textbf{SO} denotes the use of second-order information of the functions, and \textbf{Complexity} refers to the number of iterations $T$ equired to ensure that the output $x^T$ of the method to satisfy $\Exp{\norm{\nabla f(x^T)}} \leq \varepsilon$ for some $\varepsilon>0$. } 
	\label{table:complexities-ef}
	\begin{threeparttable}
    \resizebox{0.6\textwidth}{!}{%
		\begin{tabular}{ccccc}
			&  \\			
			{\bf Method}  & \bf Work & \bf PA&  \bf SO& \bf \shortstack{Complexity}\\			
			\hline
			\algname{EF14} &\citet{Seide20141bitSG}& \xmark & \xmark 
			& $\cO\br{\varepsilon^{-4}}$
			\\ 
			\hline
			\algname{Choco-SGD}&\citet{koloskova2020decentralized} & \xmark& \xmark
			& $\cO\br{\varepsilon^{-4}}$\\ 
			
			\hline
			\algname{EF21-SGD}&\citet{fatkhullin2025ef21} & \xmark & \xmark
			& $\cO\br{\varepsilon^{-4}}$\\ 

			\hline
			\algname{EF21-SGDM}&\citet{fatkhullin2023momentum}
			& \xmark & \xmark & $\cO\br{\varepsilon^{-4}}$\\ 

            \hline
			\algname{EF21-MVR}&\citet{fatkhullin2023momentum}
			& \xmark & \xmark& $\cO\br{\varepsilon^{-3}}$\\ 
            
			\hline
			\algname{$\norm{\text{EF21-SGDM}}$}&\citet{khirirat2024errorfeedbackl0l1smoothnessnormalization}
			& \xmark & \xmark& $\cO\br{\varepsilon^{-4}}$\\ 
            
			\hline
			\rowcolor{LightCyan}
			\algname{$\norm{\text{EF21-SGDM}}$} &\textbf{This work} & \cmark & \xmark
			& $\Tilde{\cO}\br{\varepsilon^{-4}}$ \\ 
			\hline
			\rowcolor{LightCyan}
			\algname{$\norm{\text{EF21-IGT}}$} &\textbf{This work} & \cmark & \xmark
			& $\Tilde{\cO}\br{\varepsilon^{-\nicefrac{7}{2}}}$ \\ 

            \hline
			\rowcolor{LightCyan}
			\algname{$\norm{\text{EF21-RHM}}$} &\textbf{This work} & \cmark & \cmark
			& $\Tilde{\cO}\br{\varepsilon^{-3}}$ \\ 

            \hline
			\rowcolor{LightCyan}
			\algname{$\norm{\text{EF21-HM}}$} &\textbf{This work} & \cmark & \cmark
			& $\Tilde{\cO}\br{\varepsilon^{-3}}$ \\ 

            \hline
			\rowcolor{LightCyan}
			\algname{$\norm{\text{EF21-MVR}}$} &\textbf{This work} & \cmark & \xmark
			& $\Tilde{\cO}\br{\varepsilon^{-3}}$ \\ 
		\end{tabular}
        }
	\end{threeparttable}
\end{table*}

    \item \textbf{Numerical evaluation.} 
    In~\Cref{sec:exp}, we benchmark our proposed algorithms for solving the image classification task with the CIFAR-10 dataset using the ResNet-18 model. Among five momentum variants, \algname{$|| \text{EF21-HM} ||$}, which employs second-order momentum, achieves the fastest per-epoch convergence. However, this performance comes at the cost of higher computations per epoch. Notably, \algname{$|| \text{EF21-IGT} ||$} outperforms other error feedback algorithms in terms of the solution accuracy against the wall-clock time. 
\end{itemize}

\section{Related Work}

\paragraph{Error feedback.}
Error feedback (EF) mechanisms have been widely adopted to enhance the solution accuracy of gradient-based algorithms that employ communication compression. The first of these EF mechanisms, \algname{EF14}, was introduced by~\citet{Seide20141bitSG} and later analyzed in both centralized~\citep{StitchSparseSGD,karimireddy2019error} and distributed settings~\citep{alistarh2018convergence,tang2019doublesqueeze,tang2019texttt,gorbunov2020linearly,khirirat2020compressed,qian2021error}. 
Another novel EF mechanism, \algname{EF21}, was proposed by~\citet{EF21} and offers an improved convergence rate of $\cO(1/\sqrt{T})$ in the gradient norm, compared to the $\cO(1/T^{1/3})$ rate of earlier EF algorithms like~\citet{koloskova2019decentralized}. 
More importantly, \algname{EF21} achieves this rate without requiring restrictive conditions, such as uniformly bounded gradient norms or bounded data heterogeneity.
Furthermore, \algname{EF21} has been extended to broader problem settings. For stochastic optimization, it can be adapted by using large mini-batches~\citep{fatkhullin2025ef21} or by employing Polyak momentum, resulting in \algname{EF21-SGDM} and \algname{EF21-MVR}~\citep{fatkhullin2023momentum}. In the context of federated optimization, where model parameters are exchanged instead of gradients, its variant called \algname{EF21-P}~\citep{EF21-P} has been proposed. 
More recently, \algname{EControl}~\citep{gao2024econtrol} was introduced to provide even stronger convergence guarantees for distributed stochastic optimization, thus advancing the development of provably efficient, distributed EF methods.

\paragraph{Normalization.}
Normalization has been commonly used to stabilize the training of randomized gradient-based algorithms for solving problems under relaxed smoothness conditions or in the presence of heavy-tailed noise.  
For minimizing relaxed smooth functions, normalization enables  stochastic momentum methods~\citep{zhao2021convergence,hubler2024parameter} in the centralized setting and \algname{EF21-SGDM}~\citep{fatkhullin2023momentum,khirirat2024errorfeedbackl0l1smoothnessnormalization} in the distributed setting to converge. 
In the heavy-tailed noise setting, normalization has been shown to ensure the convergence of stochastic methods in high-probability guarantees, including stochastic momentum methods as demonstrated by~\cite{cutkosky2021high,hubler2024gradient}.
Furthermore, normalization allows for the use of parameter-agnostic stepsize rules, as shown by~\citet{fatkhullin2023momentum}. 
However, these results are limited to the centralized setting.

\paragraph{Stochastic momentum algorithms.}
Stochastic momentum algorithms are widely used and studied for minimizing smooth objective functions. These algorithms are inspired by Polyak’s heavy-ball momentum~\citep{polyak1964speeding}, which achieves accelerated linear convergence compared to classical gradient descent when applied to twice continuously differentiable, strongly convex, and smooth functions~\citep{ghadimi2015global}. Several works~\citep{yan2018unified,yu2019linear,liu2020improved,cutkosky2020momentum,hubler2024gradient} have shown that stochastic momentum algorithms can achieve a convergence rate of $\cO(1/T^{1/4})$, matching that of classical stochastic gradient descent (SGD).
To further accelerate convergence, recent research has proposed various momentum variants that modify the gradient estimators used in momentum updates. One notable approach is extrapolated momentum, introduced by~\citet{cutkosky2020momentum}, which achieves a convergence rate of $\cO(1/T^{2/7})$. This has been further improved by algorithms such as STORM~\citep{cutkosky2019momentum}, MARS~\citep{yuan2024mars}, and two second-order momentum algorithms~\citep{salehkaleybar2022momentum,tran2022better}, all of which attain a rate of $\cO(1/T^{1/3})$. This rate is known to be optimal for minimizing smooth nonconvex functions under mild conditions~\citep{arjevani2023lower}.
However, these algorithms have primarily been developed and analyzed in the context of centralized optimization. As a result, their  applicability to the distributed optimization setting for broader training applications remains limited.

\paragraph{Parameter-agnostic algorithms.}
Other approaches for parameter-agnostic stepsizes, in addition to normalization, include adaptive stochastic methods, such as backtracking line search~\citep{armijo1966minimization}, AdaGrad~\citep{duchi2011adaptive}, AdaGrad-Norm~\citep{streeter2010less}, and ADAM~\citep{kingma2017adammethodstochasticoptimization}.

\section{Preliminaries}

\subsection{Notations}

We denote the expectation of a random variable $u$ by $\Exp{u}$. For any vectors $x, y \in \R^d$, $\inp{x}{y}$ refers to their inner product, and $\norm{x} = \sqrt{\inp{x}{x}}$ denotes the Euclidean norm of the vector $x$. For a matrix $A \in \R^{m \times n}$, $\norm{A}$ denotes its spectral norm, i.e., the largest singular value of $A$. The notation $\cO(h(x))$ implies that a function $f(x)$ satisfies $f(x) \leq c \cdot h(x)$ for some constant $c > 0$, while $\tilde{\cO}(h(x))$ hides both constant and logarithmic factors. 
Finally, we use $\underset{x\in\R^d}{\inf} f(x)$ to denote the infimum of a function $f \colon \R^d \to \R$, and $\underset{x\in\R^d}{\min} f(x)$ to denote its minimum, when it exists.


\subsection{Problem Formulation}
Consider a distributed stochastic optimization problem: 
\begin{eqnarray}\label{eqn:problem}
 \underset{x\in\R^d}{\min} \quad f(x):=\frac{1}{n}\sum_{i=1}^n f_i(x), 
\end{eqnarray}
where $f_i(x)= \ExpSub{\xi_i\sim \cD_i}{f_i(x;\xi_i)}$, and  $f_i(x;\xi_i)$ is a possibly nonconvex function  parameterized by the vector $x\in\R^d$ on the random variable $\xi_i$ drawn from the data distribution $\cD_i$ known to client $i$. 
Problem~(\ref{eqn:problem}) often appears in supervised machine learning applications~\citep{friedman2009elements}.

\subsection{Assumptions}

To facilitate our  analysis, we impose standard assumptions on compression operators and objective functions.

First, we assume an $\alpha$-contractive  compression operator, which covers many popular biased compressors of interest, including TopK~\citep{alistarh2018convergence} and RandK sparsifiers~\citep{beznosikov2023biased}.
\begin{assumption}[Contractive compression]\label{assum:contract_comp}
	A biased but possibly randomized compressor $\cC:\R^d\rightarrow\R^d$ is $\alpha$-contractive with its sample $\xi_i \sim \cD_i$ if there exists $\alpha \in (0,1]$ such that 
\begin{eqnarray*}
	\Exp{\sqnorm{\cC(v)-v}}  \leq (1-\alpha)\sqnorm{v}, \quad \forall v \in \R^d.
\end{eqnarray*}		
\end{assumption}

Second, we introduce commonly used conditions on the objective function, including the existence of a finite infimum, and Lipschitz continuity of the gradient and Hessian of component and stochastic functions.
\begin{assumption}\label{assum:inf}
The function $f:\R^d\rightarrow\R$ is bounded from below, i.e., $f^{\inf}= \underset{x\in\R^d}{\inf} f(x) > -\infty$.
\end{assumption}

\begin{assumption}\label{assum:L_stoc_comp_fcn}
The stochastic component function $f_i(x;\xi_i)$ has the $L_{\text{ms},i}$-Lipschitz continuous gradient, if there exists $L_{\text{ms},i} >0$ such that for all $x,y\in\R^d$, 
\begin{eqnarray*}
    \ExpSub{\xi_i}{\sqnorm{\nabla f_i(x;\xi_i) - \nabla f_i(y;\xi_i)}} \leq L_{\text{ms},i}^2 \sqnorm{x-y}.
\end{eqnarray*}
\end{assumption}

\begin{assumption}\label{assum:L_comp_fcn}
The component function $f_i:\R^d\rightarrow\R$ has the $L_i$-Lipschitz continuous gradient, if there exists $L_i >0$ such that for all $x,y\in\R^d$, 
\begin{eqnarray*}
    \norm{\nabla f_i(x) - \nabla f_i(y)} \leq L_i \norm{x-y}.
\end{eqnarray*}
\end{assumption}

\begin{assumption}\label{assum:L_fcn}
The function $f:\R^d\rightarrow\R$ has the $L$-Lipschitz continuous gradient, if there exists $L >0$ such that for all $ x,y\in\R^d$, 
\begin{eqnarray*}
    \norm{\nabla f(x) - \nabla f(y)} \leq L \norm{x-y}.
\end{eqnarray*}
\end{assumption}

\begin{assumption}\label{assum:L_Hessian_comp_fcn}
The component function $f_i:\R^d\rightarrow\R$ has the $L_{h,i}$-Lipschitz continuous Hessian, if there exists $L_{h,i} >0$ such that for all $x,y\in\R^d$, 
\begin{eqnarray*}
    \norm{\nabla^2 f_i(x) - \nabla^2 f_i(y)} \leq L_{h,i} \norm{x-y}.
\end{eqnarray*}
\end{assumption}

\begin{assumption}\label{assum:L_Hessian_fcn}
The  function $f:\R^d\rightarrow\R$ has the $L_{h}$-Lipschitz continuous Hessian, if there exists $L_{h} >0$ such that for all $x,y\in\R^d$, 
\begin{eqnarray*}
    \norm{\nabla^2 f(x) - \nabla^2 f(y)} \leq L_{h} \norm{x-y}.
\end{eqnarray*}
\end{assumption}

Note that \Cref{assum:L_stoc_comp_fcn} implies \Cref{assum:L_comp_fcn} with $L_i = L_{\text{ms},i}$, while  \Cref{assum:L_comp_fcn} implies \Cref{assum:L_fcn} with $L = \frac{1}{n}\sum_{i=1}^n L_i$. Furthermore, \Cref{assum:L_Hessian_comp_fcn} implies \Cref{assum:L_Hessian_fcn} with $L_h = \frac{1}{n}\sum_{i=1}^n L_{h,i}$.

Third, we assume that each client can query its stochastic oracle to obtain its local noisy gradient and Hessian that satisfy the unbiased and variance-bounded condition, which is commonly used for analyzing stochastic gradient methods~\citep{cutkosky2019momentum,cutkosky2020momentum,tran2022better,hubler2024parameter,hubler2024gradient}. 
\begin{assumption}\label{assum:stoc_g_h}
	The local stochastic gradient $\nabla f_i(x;\xi_i)$ at client $i$ is an unbiased estimator of  $\nabla f_i(x)$ with bounded variance if it satisfies: for all $x\in\R^d$, 
	\begin{eqnarray*}
		\Exp{\nabla f_i(x;\xi_i)} & =&  \nabla f_i(x), \quad \text{and}  \\ 
        \Exp{ \sqnorm{\nabla f_i(x;\xi_i) - \nabla f(x)} } &\leq& \sigma^2_g.
	\end{eqnarray*}	
	Furthermore, 
	the local stochastic Hessian $\nabla^2 f_i(x;\xi)$ at client $i$ is an unbiased estimator of  $\nabla^2 f_i(x)$ with bounded variance if it satisfies: for all $x\in\R^d$,
	\begin{eqnarray*}
		\Exp{\nabla^2 f_i(x;\xi_i)} & =&  \nabla^2 f_i(x), \quad \text{and}  \\
     \Exp{ \sqnorm{\nabla^2 f_i(x;\xi_i) - \nabla^2 f(x)} } &\leq& \sigma^2_h.
	\end{eqnarray*}	
\end{assumption}




\section{EF21 Methods with Five Momentum Updates}
\label{sec:ef21_momentum}

To solve Problem~(\ref{eqn:problem}), we consider EF21 algorithms (\Cref{alg:EF_momentum}) that leverage normalized descent updates and momentum variants. 

At each iteration $t=0,1,\ldots,T-1$, key updating rules for~\Cref{alg:EF_momentum} are described as follows: 
\begin{itemize}
    \item Each client $i=1,2,\ldots,n$ computes stochastic gradients $\nabla f_i(\cdot;\xi_i^{t+1})$, and updates its local momentum vector $v_i^{t+1}$ depending on the choice of momentum. 
    \item Each client $i=1,2,\ldots,n$ transmits the compressed vector $c_i^{t+1} = \mathcal{C}_i^{t+1}(v_i^{t+1}-g_i^t)$, and updates its local memory vector $g_i^{t+1} = g_i^t + c_i^{t+1}$.
    \item The server receives $c_i^{t+1}$ from every client, updates $g^{t+1} = g^t + \frac{1}{n}\sum_{i=1}^n c_i^{t+1}$, and computes $x^{t+1} = x^t -\gamma_t \cdot \nicefrac{g^t}{\norm{g^t}}$.
\end{itemize}

Depending on the choice for updating the momentum,~\Cref{alg:EF_momentum} is referred to as:
\begin{itemize}
    \item \algname{$\|\text{EF21-SGDM}\|$}~\citep{khirirat2024errorfeedbackl0l1smoothnessnormalization} when using the Polyak momentum, 
    \item \algname{$\|\text{EF21-IGT}\|$} when using the Implicit Gradient Transport (IGT) momentum~\citep{cutkosky2020momentum}, 
    \item \algname{$\|\text{EF21-RHM}\|$} when using the second-order momentum variant by~\citet{salehkaleybar2022momentum}, \item \algname{$\|\text{EF21-HM}\|$} when using the second-order momentum variant by~\citet{tran2022better}, and 
    \item \algname{$\|\text{EF21-MVR}\|$} when using the momentum variant from STORM~\citep{cutkosky2019momentum}.
\end{itemize}
Here, \algname{$\|\text{EF21-SGDM}\|$} and  \algname{$\|\text{EF21-MVR}\|$} are the normalized versions of \algname{EF21-SGDM} and \algname{EF21-MVR}, respectively, as analyzed by~\citet{fatkhullin2023momentum}.
Moreover, by setting the compression operator $\cC_i^t$ to the identity and choosing $n=1$, 
our algorithms reduce to the corresponding normalized stochastic momentum methods in the centralized setting: 
(1) \algname{$\|\text{EF21-SGDM}\|$} becomes normalized stochastic momentum methods~\citep{cutkosky2020momentum,hubler2024parameter}, (2) \algname{$\|\text{EF21-IGT}\|$} becomes normalized SGD with IGT momentum~\citep{cutkosky2020momentum}, (3) \algname{$\|\text{EF21-RHM}\|$} becomes normalized second-order momentum methods analyzed by~\citet{salehkaleybar2022momentum}, (4) \algname{$\|\text{EF21-HM}\|$} becomes second-order momentum methods proposed by~\citet{tran2022better}, and (5) \algname{$\|\text{EF21-MVR}\|$} becomes STORM~\citep{cutkosky2019momentum}. 
Finally, notice that~\Cref{alg:EF_momentum} with $v_i^{t+1} = \nabla f_i(x^{t+1})$ recovers \algname{$\|\text{EF21}\|$} analyzed by~\citet{khirirat2024errorfeedbackl0l1smoothnessnormalization}.

\begin{algorithm}[t]
\caption{\algname{EF21} with Normalized Updates and Momentum Variants}
\label{alg:EF_momentum}
\begin{algorithmic}[1]
\STATE \textbf{Input:} Initial point $x^0, g_i^0,v_i^0 \in \mathbb{R}^d$, stepsize $\gamma_t > 0$, momentum parameter $0<\eta_t \leq 1$, compressor $\mathcal{C}_i^t$,  number of iterations $T$
\FOR{$t = 0, 1, \ldots, T-1$}
    \STATE Master computes $x^{t+1} = x^t - \gamma_t \frac{g^t}{\norm{g^t}}$, and broadcasts $x^{t+1}$ to all the clients 
    \FOR{Every client $i=1,2,\ldots,n$}
        \vspace{0.3em}
        \STATE \algname{$\|\text{EF21-SGDM}\|$}:
        \STATE  \quad Set $v_i^{t+1}=(1-\eta_t)v_i^t + \eta_t \nabla f_i \left( x^{t+1}  ;\xi_i^{t+1} \right)$

        \vspace{0.3em}
        \STATE \algname{$\|\text{EF21-IGT}\|$}:
        \STATE  \quad Set $v_i^{t+1}=(1-\eta_t)v_i^t + \eta_t \nabla f_i \left( y^t;\xi_i^{t+1} \right)$, where $y^t = x^{t+1} + \frac{1-\eta_t}{\eta_t} (x^{t+1}-x^t)$

        \vspace{0.3em}
        \STATE \algname{$\|\text{EF21-RHM}\|$}:
        \STATE  \quad Choose $\hat x^{t+1} = q_t x^{t+1} + (1-q_t) x^t$, where $q_t \sim \cU(0,1)$ 
        \STATE  \quad Set $v_i^{t+1}=(1-\eta_t)\tilde v_i^t + \eta_t \nabla f_i(x^{t+1};\xi_i^{t+1})$, where $\tilde v_i^t = v_i^t +  \nabla^2 f_i(\hat x^{t+1};\xi_i^{t+1}) (x^{t+1} - x^t)$

        \vspace{0.3em}
        \STATE \algname{$\|\text{EF21-HM}\|$}:
        \STATE  \quad Set $v_i^{t+1}=(1-\eta_t)\tilde v_i^t + \eta_t \nabla f_i(x^{t+1};\xi_i^{t+1})$, where $\tilde v_i^t = v_i^t + \nabla^2 f_i( x^{t+1};\xi_i^{t+1}) (x^{t+1} - x^t)$

        \vspace{0.3em}
        \STATE \algname{$\|\text{EF21-MVR}\|$}:
        \STATE  \quad Set $v_i^{t+1}=(1-\eta_t)\tilde v_i^t + \eta_t \nabla f_i(x^{t+1};\xi_i^{t+1})$, where $\tilde v^t_i =  v_i^t + \nabla f_i(x^{t+1};\xi_i^{t+1}) - \nabla f_i(x^t;\xi_i^{t+1})$

        \vspace{0.3cm}

        \STATE Transmit $c_i^{t+1} = \cC_i^{t+1}(v_i^{t+1}-g_i^t)$
        \STATE Update $g_i^{t+1}=g_i^t + c_i^{t+1}$ 
    \ENDFOR 
    \STATE Master receives $c_1^{t+1},\ldots,c_n^{t+1}$, and  computes $g^{t+1} = g^t + \frac{1}{n}\sum_{i=1}^n c_i^{t+1}$
\ENDFOR
\STATE \textbf{Output:} $x^T$
\end{algorithmic}
\end{algorithm}




\section{Convergence Theorems}
\label{sec:convergence}

To this end, we present the convergence of \Cref{alg:EF_momentum} that use five momentum variants: Polyak momentum, IGT momentum, MVR momentum in STORM, and two second-order momentum schemes.

\paragraph{$\|\text{EF21-SGDM}\|$.}
We begin by providing the convergence rate results for \algname{$\|\text{EF21-SGDM}\|$}.


\begin{theorem}\label{thm:sgdm}
Consider Problem~(\ref{eqn:problem}), where Assumptions~\ref{assum:contract_comp},~\ref{assum:inf},~\ref{assum:L_comp_fcn},~\ref{assum:L_fcn}, and~\ref{assum:stoc_g_h} hold. 
Let tuning parameters satisfy 
\begin{eqnarray*}
\eta_t = \left( \frac{2}{t+2} \right)^{1/2} \quad \text{and} \quad \gamma_t = \gamma_0 \left( \frac{2}{t+2} \right)^{3/4}
\end{eqnarray*}
with $\gamma_0 >0$. 
Then, the iterates $\{x^t\}$ governed by \algname{EF21-SGDM} 
satisfy
\begin{align*}
& \Exp{\norm{\nabla f(\tilde{x}^T)}}  \\
&= \widetilde{\cO}\left(\frac{\nicefrac{V_0}{\gamma_0} + \gamma_0\left(L + \nicefrac{\bar L}{\alpha^2}\right) + \sigma_g\left(\nicefrac{1}{\sqrt{n}} + \nicefrac{1}{\alpha^2}\right)}{T^{\nicefrac{1}{4}}}\right),
\end{align*}    
where $\tilde{x}^T$ is randomly chosen from $\{x^0,x^1,\ldots,x^{T-1}\}$ with probability $\nicefrac{\gamma_t}{\sum_{t=0}^{T-1}\gamma_t}$ for $t=0,1,\ldots,T-1$. Here, $\bar L = \frac{1}{n}\sum_{i=1}^n L_i$.
\end{theorem}

From~\Cref{thm:sgdm}, \algname{$\|\text{EF21-SGDM}\|$} achieves the $\tilde{\cO}(1/T^{1/4})$ convergence in the gradient norm.
This result holds without requiring the stepsizes to depend on problem-dependent parameters (e.g. smoothness constants), but still yields the convergence almost matching the $\cO(1/T^{1/4})$ convergence bound for \algname{$\|\text{EF21-SGDM}\|$} by~\citet{khirirat2024errorfeedbackl0l1smoothnessnormalization} and for \algname{EF21-SGDM} analyzed by~\citet{fatkhullin2023momentum}.
Furthermore, unlike Theorem 3 in \citet{fatkhullin2023momentum}, our result does not rely on initializing the algorithm with a sufficiently large mini-batch size. 
Finally, by setting the compression operator  $\mathcal{C}^t(\cdot)$ be the identity operator and by letting $n=1$, our result recovers the same 
$\tilde{\cO}(1/T^{1/4})$ convergence rate as that of normalized stochastic momentum methods in the centralized setting by~\citet{hubler2024parameter}.




\paragraph{$\|\text{EF21-IGT}\|$.}
To further improve the convergence performance of \algname{$\|\text{EF21-SGDM}\|$}, we replace Polyak momentum with IGT momentum~\citep{cutkosky2020momentum} in~\Cref{alg:EF_momentum}.
The resulting algorithm, referred to as \algname{$\|\text{EF21-IGT}\|$}, admits the following convergence guarantee: 
\begin{theorem}\label{thm:IGT}
Consider Problem~(\ref{eqn:problem}), where Assumptions~\ref{assum:contract_comp},~\ref{assum:inf},~\ref{assum:L_comp_fcn},~\ref{assum:L_fcn},~\ref{assum:L_Hessian_comp_fcn},~\ref{assum:L_Hessian_fcn} and~\ref{assum:stoc_g_h} hold. 
Let tuning parameters satisfy 
\begin{eqnarray*}
\eta_t = \left( \frac{2}{t+2} \right)^{4/7} \quad \text{and} \quad \gamma_t = \gamma_0 \left( \frac{2}{t+2} \right)^{5/7} 
\end{eqnarray*}
with $\gamma_0 > 0$. 
Then, the iterates $\{x^t\}$ governed by \algname{EF21-IGT} 
satisfy
\begin{align*}
&\Exp{\norm{\nabla f(\tilde{x}^T)}} \\
& = \widetilde{\cO}\left(\frac{\nicefrac{V_0}{\gamma_0} + \sigma_g\left(\nicefrac{1}{\sqrt{n}} + \nicefrac{1}{\alpha^2}\right) + \gamma_0B_1 + \gamma_0^2B_2}{T^{\nicefrac{2}{7}}}\right),
\end{align*} 
where  $\tilde{x}^T$ is randomly chosen from $\{x^0,x^1,\ldots,x^{T-1}\}$ with probability $\nicefrac{\gamma_t}{\sum_{t=0}^{T-1}\gamma_t}$ for $t=0,1,\ldots,T-1$.
Here, $B_1 = L+\nicefrac{\bar L}{\alpha^2}$ and $B_2 = L_h +\nicefrac{\bar L_h}{\alpha^2}$. Here, $\bar L =\frac{1}{n}\sum_{i=1}^n L_i$ and $\bar L_h = \frac{1}{n}\sum_{i=1}^n L_{h,i}$.
\end{theorem}

From~\Cref{thm:IGT}, \algname{$\|\text{EF21-IGT}\|$}, which uses decreasing, parameter-agnostic stepsize rules, attains the $\tilde{\cO}(1/T^{2/7})$ convergence in the gradient norm. 
When $\mathcal{C}^t(\cdot) = I$ and $n=1$, this result recovers the same $\tilde{\cO}(1/T^{2/7})$ rate as stochastic methods with IGT momentum in the centralized setting--specifically, Theorem 3 of~\citet{cutkosky2020momentum}, Theorem 6 of~\citet{cutkosky2021high} in the bounded variance case,  and  Theorem 2 of~\citet{sun2023momentum}.
Unlike these prior results, however, our method achieves this rate without requiring stepsizes to know smoothness constants.

\paragraph{$\|\text{EF21-RHM}\|$.}
Another momentum variant we incorporate into \algname{EF21} methods using momentum (\Cref{alg:EF_momentum}) to enhance their convergence performance is second-order momentum proposed by~\citet{salehkaleybar2022momentum}. 
We refer to the resulting algorithm as~\algname{$\|\text{EF21-RHM}\|$}, and its convergence guarantee is presented below:  
\begin{theorem}\label{thm:RHM}
Consider Problem~(\ref{eqn:problem}), where Assumptions~\ref{assum:contract_comp},~\ref{assum:inf},~\ref{assum:L_comp_fcn},~\ref{assum:L_fcn}
and~\ref{assum:stoc_g_h} hold. 
Let tuning parameters satisfy 
\begin{eqnarray*}
\eta_t = \left( \frac{2}{t+2} \right)^{2/3}\quad \text{and} \quad \gamma_t = \gamma_0 \left( \frac{2}{t+2} \right)^{2/3}
\end{eqnarray*}
with $\gamma_0 >0$. 
Then, the iterates $\{x^t\}$ governed by \algname{EF21-RHM} 
satisfy
\begin{align*}
&   \Exp{\norm{\nabla f(\tilde{x}^T)}} \\
&= \widetilde{\cO}\left(\frac{\nicefrac{V_0}{\gamma_0}+ \left(\sigma_g + \gamma_0\sigma_h\right)\left(\nicefrac{1}{\sqrt{n}}+\nicefrac{1}{\alpha^2}\right) + \gamma_0 C }{T^{\nicefrac{1}{3}}}\right),
\end{align*}  
where $\tilde{x}^T$ is randomly chosen from $\{x^0,x^1,\ldots,x^{T-1}\}$ with probability $\nicefrac{\gamma_t}{\sum_{t=0}^{T-1}\gamma_t}$ for $t=0,1,\ldots,T-1$. Here, $C= L +\nicefrac{\bar L}{\alpha^2}$ and $\bar L = \frac{1}{n}\sum_{i=1}^n L_i$.
\end{theorem}

From~\Cref{thm:RHM}, \algname{$\|\text{EF21-RHM}\|$} achieves the $\tilde{\cO}(1/T^{1/3})$ convergence in the gradient norm, thus improving upon the rates of \algname{$\|\text{EF21-SGDM}\|$} and \algname{$\|\text{EF21-IGT}\|$}.
When $\mathcal{C}^t(\cdot) = I$ and $n=1$, our result nearly matches the ${\cO}(1/T^{1/3})$ convergence for stochastic second-order momentum methods in the centralized setting  by~\citet{salehkaleybar2022momentum}, Theorem 2 of \citet{sadiev2025second} in the bounded variance case.

\paragraph{$\|\text{EF21-HM}\|$.}
In addition to the second-order momentum by~\citet{salehkaleybar2022momentum}, we also incorporate the second-order momentum proposed by~\citet{tran2022better} into~\Cref{alg:EF_momentum}. 
The resulting algorithm, which we call \algname{$\|\text{EF21-HM}\|$}, achieves the following convergence rate:  
\begin{theorem}\label{thm:HM}
Consider Problem~(\ref{eqn:problem}), where Assumptions~\ref{assum:contract_comp},~\ref{assum:inf},
~\ref{assum:L_fcn},
~\ref{assum:L_Hessian_comp_fcn}
and~\ref{assum:stoc_g_h} hold. 
Let tuning parameters satisfy 
\begin{eqnarray*}
\eta_t = \left( \frac{2}{t+2} \right)^{2/3}  \quad \text{and} \quad \gamma_t = \gamma_0 \left( \frac{2}{t+2} \right)^{2/3}
\end{eqnarray*}
with $\gamma_0 >0$. 
Then, the iterates $\{x^t\}$ governed by \algname{EF21-HM} 
satisfy
\begin{align*}
&    \Exp{\norm{\nabla f(\tilde{x}^T)}} \\
& = \widetilde{\cO}\left(\frac{\nicefrac{V_0}{\gamma_0} + D_\sigma\left(\nicefrac{1}{\sqrt{n}} + \nicefrac{1}{\alpha^2}\right)+ \gamma_0D_1 +  \gamma^2_0D_2 }{T^{\nicefrac{1}{3}}}\right),
\end{align*}    
where $\tilde{x}^T$ is randomly chosen from $\{x^0,x^1,\ldots,x^{T-1}\}$ with probability $\nicefrac{\gamma_t}{\sum_{t=0}^{T-1}\gamma_t}$ for $t=0,1,\ldots,T-1$. Here, $D_\sigma=\sigma_g + \gamma_0\sigma_h$, $D_1 =L + \nicefrac{\bar L}{\alpha^2}$,  $D_2 = L_h +\nicefrac{\bar L_h}{\alpha^2}$, $\bar L = \frac{1}{n}\sum_{i=1}^n L_i$, and $\bar L_h = \frac{1}{n}\sum_{i=1}^n L_{h,i}$.
\end{theorem}

From~\Cref{thm:HM}, \algname{$\|\text{EF21-HM}\|$} achieves the same $\tilde{\cO}(1/T^{1/3})$ convergence in the gradient norm as \algname{$\|\text{EF21-RHM}\|$}. 
When $\mathcal{C}^t(\cdot) = I$ and $n=1$, our result closely matches the ${\cO}(1/T^{1/3})$ convergence for stochastic second-order momentum methods in the centralized setting  by~\citet{tran2022better}.

\paragraph{$\|\text{EF21-MVR}\|$.}
Finally, we consider \Cref{alg:EF_momentum} with the MVR momentum from STORM~\citep{cutkosky2019momentum}, resulting in \algname{$\|\text{EF21-MVR}\|$}. 
The next theorem presents its parameter-agnostic convergence guarantee.
\begin{theorem}\label{thm:MVR}
Consider Problem~(\ref{eqn:problem}), where Assumptions~\ref{assum:contract_comp},~\ref{assum:inf},
~\ref{assum:L_stoc_comp_fcn},
~\ref{assum:L_fcn}
and~\ref{assum:stoc_g_h} hold. 
Let tuning parameters satisfy 
\begin{eqnarray*}
\eta_t = \left( \frac{2}{t+2} \right)^{2/3} \quad \text{and} \quad \gamma_t = \gamma_0 \left( \frac{2}{t+2} \right)^{2/3} 
\end{eqnarray*}
with $\gamma_0 > 0$. 
Then, the iterates $\{x^t\}$ governed by \algname{EF21-MVR} 
satisfy
\begin{align*}
&\Exp{\norm{\nabla f(\tilde{x}^T)}}  \\
&=\widetilde{\cO}\left(\frac{\nicefrac{V_0}{\gamma_0} + \sigma_g\left(\nicefrac{1}{\sqrt{n}} +\nicefrac{1}{\alpha^2}\right) +\gamma_0E_1}{T^{\nicefrac{1}{3}}}\right),
\end{align*}   
where  $\tilde{x}^T$ is randomly chosen from $\{x^0,x^1,\ldots,x^{T-1}\}$ with probability $\nicefrac{\gamma_t}{\sum_{t=0}^{T-1}\gamma_t}$ for $t=0,1,\ldots,T-1$. Here, $E_1=L + {\bar L}_{\text{ms}} + \nicefrac{{\bar L}_{\text{ms}}}{\alpha^2}$ and $\bar L_{\text{ms}} = \frac{1}{n}\sum_{i=1}^n L_{\text{ms},i}$.
\end{theorem}

From~\Cref{thm:MVR}, \algname{$\|\text{EF21-MVR}\|$} achieves the same $\tilde{\cO}(1/T^{1/3})$ convergence in the gradient norm as \algname{$\|\text{EF21-RHM}\|$} and \algname{$\|\text{EF21-HM}\|$}. 
Unlike \algname{$\|\text{EF21-RHM}\|$} and \algname{$\|\text{EF21-HM}\|$}, which require Hessian-vector evaluations per iteration, 
\algname{$\|\text{EF21-MVR}\|$} needs two stochastic gradient computations. 
Our result nearly matches the ${\cO}(1/T^{1/3})$ convergence for \algname{EF21-STORM/MVR} by~\citet{fatkhullin2023momentum}, but with notable improvements. 
\algname{$\|\text{EF21-MVR}\|$} does not require a large mini-batch size at initialization, and operates with parameter-agnostic stepsizes without the need to know problem-specific quantities like smoothness constants. 
Furthermore, when $\mathcal{C}^t(\cdot) = I$ and $n=1$, our result closely matches the ${\cO}(1/T^{1/3})$ convergence for STORM in the centralized setting  by~\citet{cutkosky2019momentum}, but without assuming additional restrictive conditions like bounded norms of stochastic gradients.




\begin{table*}[htbp]
\centering
\label{main_table:accuracy_results}
\begin{threeparttable}
\resizebox{0.65\textwidth}{!}{%
\begin{tabular}{@{}lccccc@{}}
\toprule
\textbf{Method} & \textbf{\makecell{Best Val.\\Accuracy (\%)}} & \textbf{\makecell{Corr. Test\\Accuracy (\%)}} & \textbf{\makecell{Epoch of\\Best}} & \textbf{\makecell{GPU Time\\to Best}} & \textbf{\makecell{Wall Time\\to Best}} \\ \midrule
\algname{EF21-SGDM}          & $74.66$          & $73.53$          & $76$ & 0h 35m 59s & 0h 35m 59s \\
\algname{EControl}           & $76.38$          & $74.49$          & $79$ & 0h 38m 06s & 0h 38m 06s \\
\algname{EF21-SGD}           & $77.50$          & $75.90$          & $69$ & 0h 30m 35s & 0h 30m 36s \\
\midrule
\algname{$\norm{\text{EF21-MVR}}$}  & $79.30$          & $78.77$          & $72$ & 0h 44m 33s & 0h 44m 33s \\
\algname{$\|\text{EF21-RHM}\|$} & $81.48$          & $80.54$          & $66$ & 1h 35m 10s & 1h 35m 10s \\
\algname{$\norm{\text{EF21-SGDM}}$} & $82.66$          & $81.70$          & $79$ & 0h 37m 54s & 0h 37m 55s \\
\textbf{\algname{$\norm{\text{EF21-IGT}}$}}  & $\mathbf{83.48}$          & $\mathbf{81.77}$          & $\mathbf{70}$ & \textbf{0h 39m 03s} & \textbf{0h 39m 03s} \\
\textbf{\algname{$\|\text{EF21-HM}\|$}}  & $\mathbf{84.32}$ & $\mathbf{83.22}$ & $\mathbf{75}$ & \textbf{1h 37m 55s} & \textbf{1h 37m 56s} \\ \bottomrule
\end{tabular}}
\caption{Best performance metrics achieved by each method when training ResNet-18 on CIFAR-10, sorted by validation accuracy. The top two results are highlighted in \textbf{bold}.}

\end{threeparttable}
\end{table*}

\begin{figure*}[t]
    \centering
    \begin{subfigure}{0.32\textwidth}
        \includegraphics[width=\linewidth]{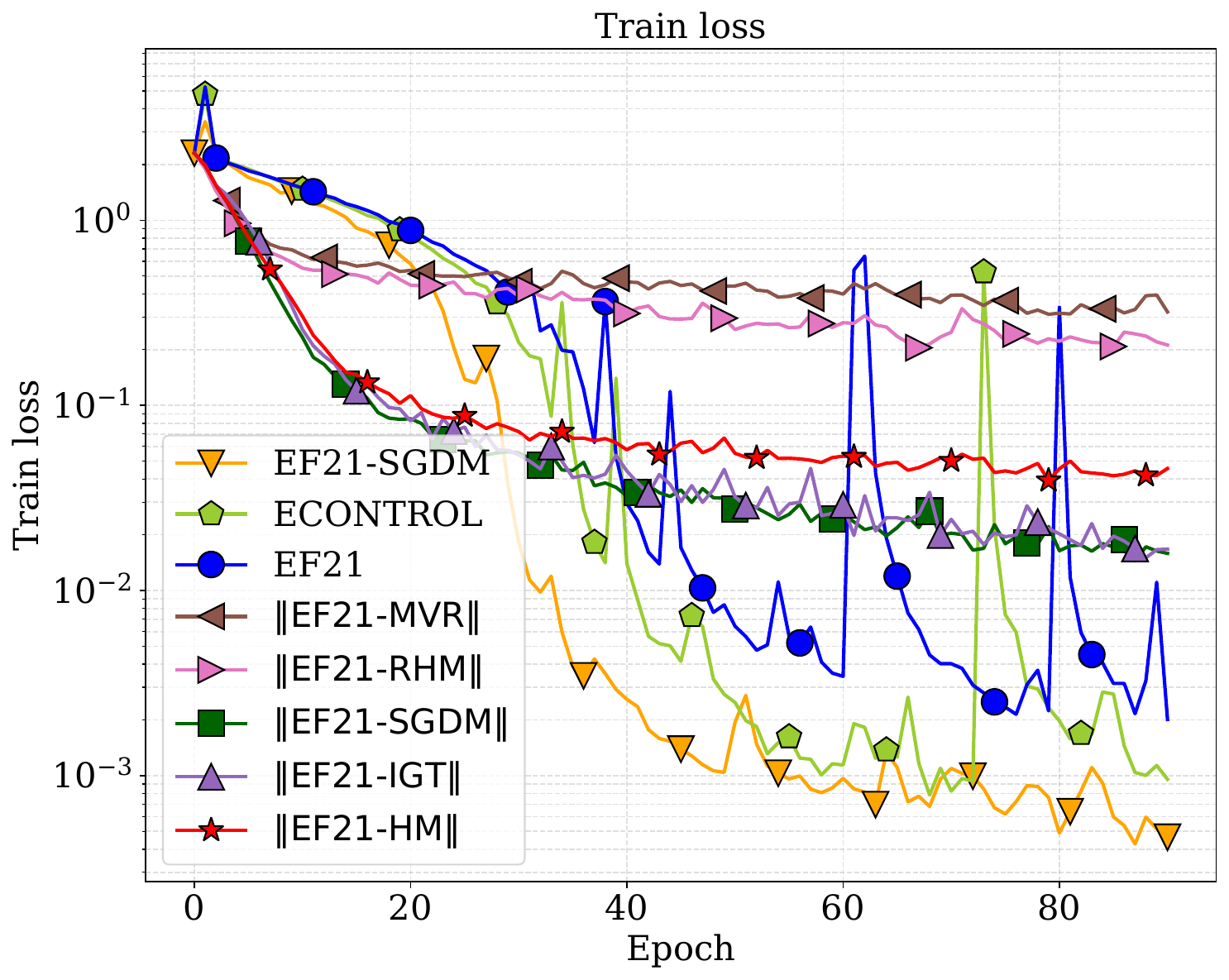}
        \caption{Training Loss}
        \label{main_fig:train_loss_epoch}
    \end{subfigure}
    \hfill 
    \begin{subfigure}{0.32\textwidth}
        \includegraphics[width=\linewidth]{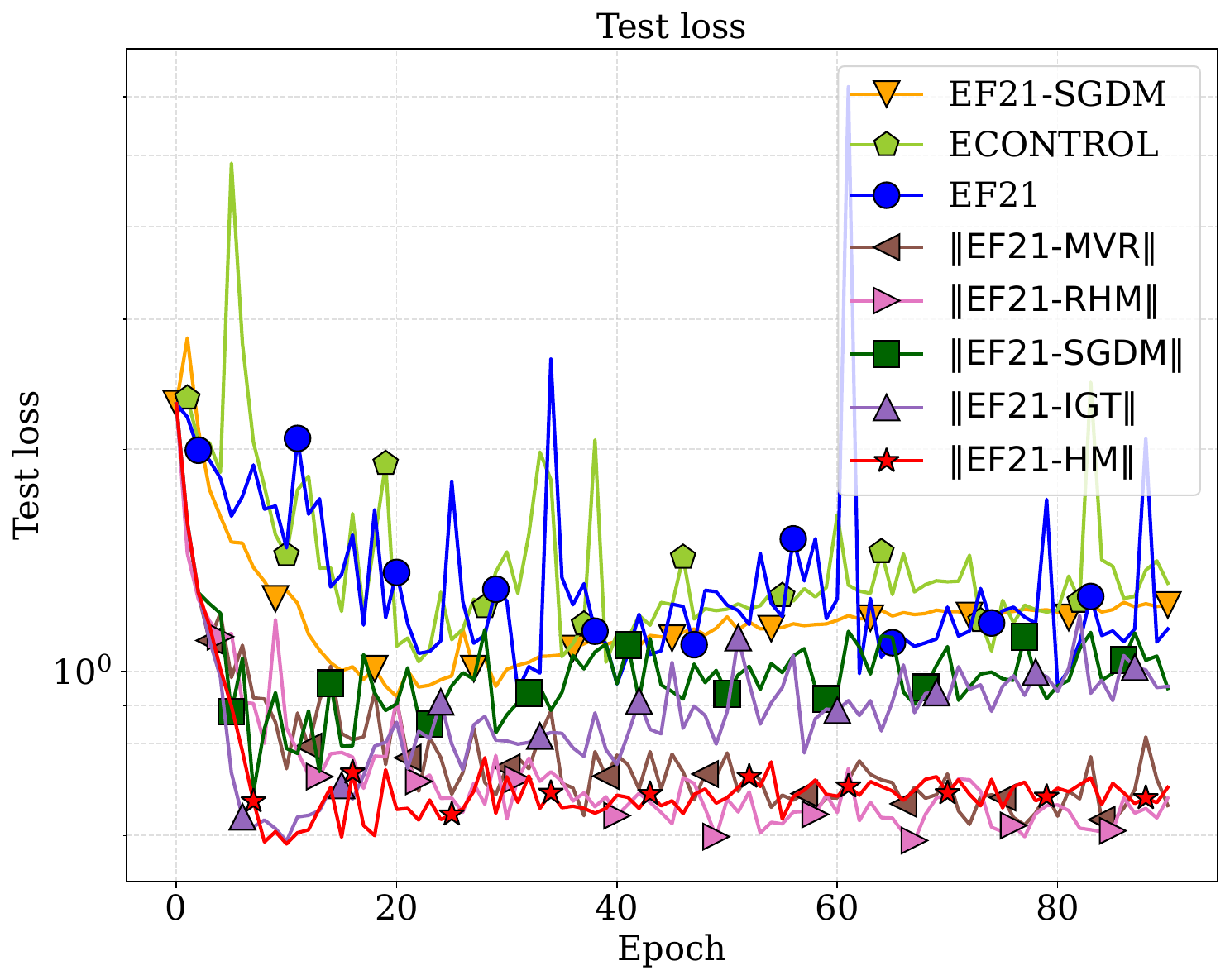}
        \caption{Test Loss}
        \label{main_fig:test_loss_epoch}
    \end{subfigure}
    \hfill 
    \begin{subfigure}{0.32\textwidth}
        \includegraphics[width=\linewidth]{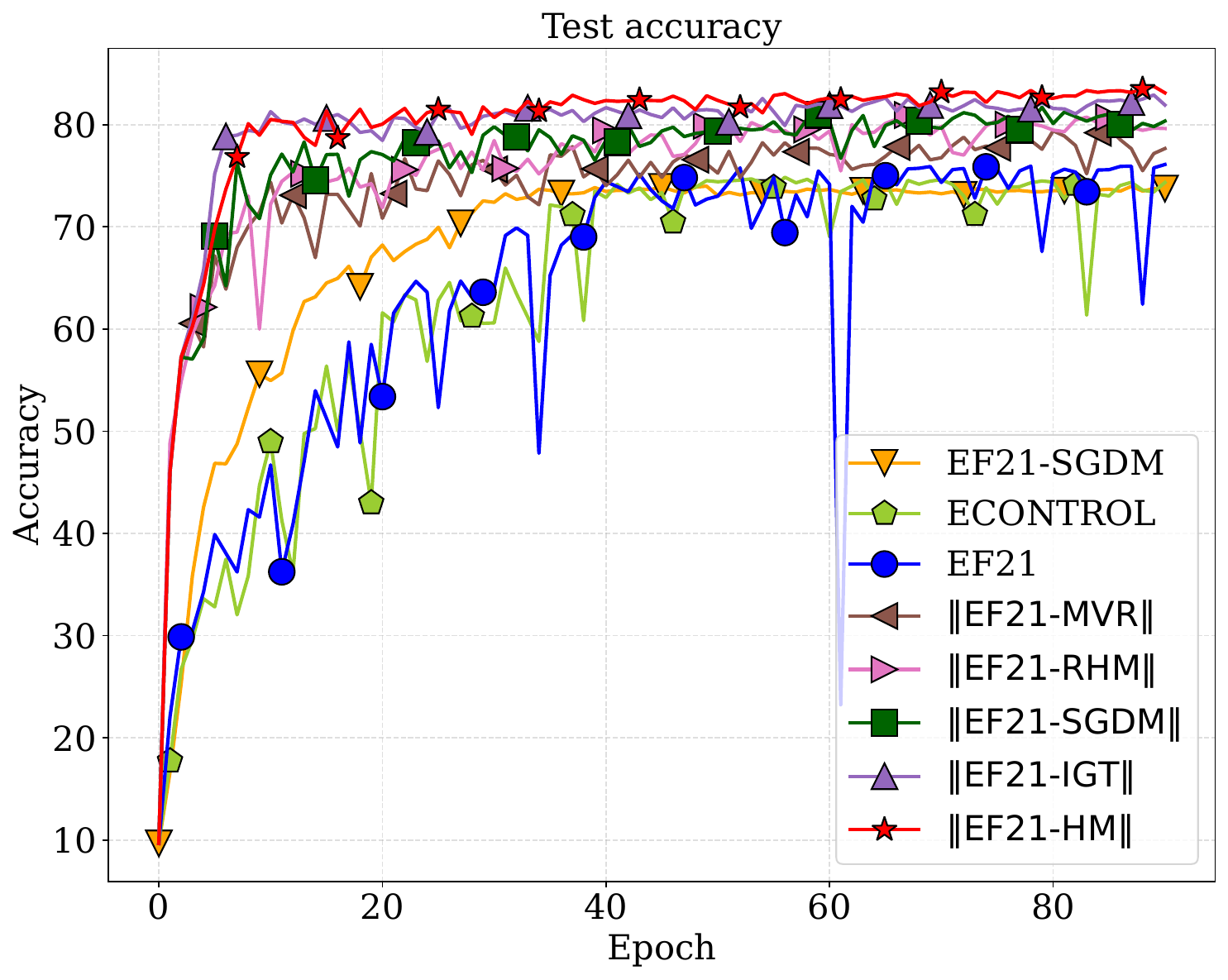}
        \caption{Test Accuracy}
        \label{main_fig:test_acc_epoch}
    \end{subfigure}
    \caption{Performance comparison of all methods on CIFAR-10 with ResNet-18, plotted as a function of epochs. The proposed momentum variants, particularly \algname{$\norm{\text{EF21-HM}}$} and \algname{$\norm{\text{EF21-IGT}}$}, show superior sample efficiency.}
    \label{main_fig:performance_plots_epoch}
\end{figure*}


\section{Numerical Experiments}\label{sec:exp}
We evaluated the performance of \algname{EF21} methods \footnote{To support reproducibility, we provide the full experimental code in our \href{https://github.com/IgorSokoloff/Parameter-Agnostic-Error-Feedback-through-Momentum}{GitHub repository}.}  with momentum variants (\Cref{alg:EF_momentum}) for training the ResNet-18 model~\citep{he2016deep} (with $d = 11,173,962$ parameters) on the CIFAR-10 dataset~\citep{krizhevsky2009learning}.
In particular, we benchmark~\Cref{alg:EF_momentum} with five momentum variants against existing distributed error feedback methods: \algname{EF21-SGD} \citep{fatkhullin2025ef21}, \algname{EF21-SGDM} \citep{fatkhullin2023momentum}, and \algname{EControl} \citep{gao2024econtrol}.

{\color{black}In these experiments,} we adopted a data distribution strategy inspired by \citet{gao2024econtrol}. Specifically, 50\% of the CIFAR-10 dataset was allocated to 10 clients based on class labels, such that data points with the $i$-th label (for $i \in \{0, \dots, 9\}$) were assigned to client $i+1$. The remaining 50\% of the dataset was distributed randomly and uniformly among the clients. Subsequently, each client’s local data was partitioned into a training set (90\%) and a test set (10\%). This partitioning scheme introduces data heterogeneity, a common characteristic of federated settings. For communication compression, we employed the Top-K sparsifier, retaining 10\% of the coordinates (i.e., $\nicefrac{K}{d} = 0.1$).
All experiments were implemented using PyTorch~\citep{paszke2019pytorch}, and were performed on a server-grade machine running Ubuntu 18.04 (Linux Kernel v5.4.0). This system was equipped with dual 16-core 3.3 GHz Intel Xeon processors (totaling 32 cores) and four NVIDIA A100 GPUs, each with 40GB of memory.
Finally, we reported the convergence with respect to the number of epochs, the GPU time, and the wall clock time. 
Details on the computation of GPU and wall-clock time, including the hyperparameter tuning procedures for each method, are provided in the appendix.



\textbf{Implementation details.}
Implementing \algname{$\|\text{EF21-SGDM}\|$}, \algname{$\|\text{EF21-IGT}\|$} and \algname{$\|\text{EF21-MVR}\|$} is straightforward. 
\algname{$\|\text{EF21-IGT}\|$} requires one additional vector to memorize the extrapolated point, while \algname{$\|\text{EF21-MVR}\|$}  requires two stochastic gradient evaluations on the same minibatch. 
For \algname{$\|\text{EF21-HM}\|$} and \algname{$\|\text{EF21-RHM}\|$}, we can compute Hessian-vector products (HVPs) efficiently.
This is achieved via the well-known identity
\begin{align*}
    \nabla^2 f(x;\xi)\,v \;=\; \nabla_x\big\langle \nabla_x f(x;\xi),\, v \big\rangle,
\end{align*}
which allows for the computation of the HVP without materializing the full Hessian matrix, which is computationally expensive.

\textbf{$\bullet$}
\algname{$\|\text{EF21-HM}\|$}:
At each iteration $t$ and worker $i$, we form the displacement vector $\Delta^{t+1} = x^{t+1} - x^t$. The stochastic gradient $\nabla f_i(x^{t+1};\xi_i^{t+1})$ and the HVP $\quad\text{and}\quad h_v^{t+1} \;\gets\; \nabla^2 f_i(x^{t+1};\xi_i^{t+1})\,\Delta^{t+1}$ are then computed at the new point $x^{t+1}$ using the same minibatch $\xi_i^{t+1}$.
The momentum buffer is subsequently updated using this Hessian correction:
\begin{align*}
    v_i^{t+1} \;=\; (1-\eta_t)\bigl(v_i^t + h_v^{t+1}\bigr) \;+\; \eta_t\, \nabla f_i(x^{t+1};\xi_i^{t+1}).
\end{align*}
This approach has a computational cost equivalent to two backpropagations per minibatch.

\textbf{$\bullet$}
\algname{$|\text{EF21-RHM}|$}: In this randomized variant, the HVP is evaluated at an interpolated point,
$$\hat{x}^{t+1} = q_t\,x^{t+1} + (1-q_t)\,x^t,
    \qquad q_t \sim \mathcal{U}(0,1).$$
At each iteration $t$, and for each worker $i$, we define the displacement vector $\Delta^{t+1} = x^{t+1} - x^t$. The stochastic gradient $\nabla f_i(\hat{x}^{t+1};\hat{\xi}_i^{t+1})$ and the HVP
$$
    h_v^{t+1} \;\gets\; \nabla^2 f_i(\hat{x}^{t+1};\hat{\xi}_i^{t+1})\,\Delta^{t+1}
$$
are then computed at the new point $x^{t+1}$ using an independent minibatch $\xi_i^{t+1}$. The overall computational overhead corresponds to three backpropagations per minibatch.

\paragraph{Epoch-based  performance.}
We observe that normalization significantly improves the per-epoch convergence performance of  distributed error feedback algorithms, in both validation and test accuracy. 
Among all methods, \algname{$\norm{\text{EF21-HM}}$}  achieves the  highest accuracy, outperforming other momentum variants. 
Meanwhile, \algname{$\norm{\text{EF21-IGT}}$}, \algname{$\norm{\text{EF21-RHM}}$}, and \algname{$\norm{\text{EF21-SGDM}}$} also perform competitively, each surpassing the $80\%$ accuracy threshold.


\paragraph{Wall-Clock Time Performance.}
Next, we evaluate the convergence performance with respect to wall-clock time.
\algname{$\norm{\text{EF21-IGT}}$} achieves the best overall performance in terms of reaching the highest accuracy within the shortest time, despite all distributed error feedback algorithms being run under the same timing conditions.
While \algname{$\norm{\text{EF21-HM}}$} achieves the strongest per-epoch accuracy, it requires roughly twice the wall-clock time to reach its peak performance.
This increased cost is due to its use of Hessian-vector products, which require the equivalent of two backpropagations per minibatch—compared to a single backpropagation used by  \algname{$\norm{\text{EF21-IGT}}$} and \algname{$\norm{\text{EF21-SGDM}}$}.




\section{Conclusion}

In this paper, we have proposed a class of \algname{EF21} algorithms that combine parameter-agnostic stepsize rules with normalization and momentum variants. We show that these methods achieve the near-optimal convergence guarantees of vanilla \algname{EF21} algorithms, which rely on problem-dependent stepsizes. Our theoretical results match existing convergence bounds for momentum-based \algname{EF21} algorithms and recover known guarantees for normalized stochastic gradient methods with momentum. Finally, our experiments on deep neural network training tasks confirm that novel momentum variants—such as IGT momentum and Hessian-corrected momentum—can further enhance the convergence of \algname{EF21} algorithms.

\paragraph{Future work.}
This paper focuses on designing distributed \algname{EF21} algorithms using momentum and normalization for nonconvex problems. 
Therefore, it is natural to consider extensions to (star-)convex optimization problems. 
Moreover, building on the recent work  by~\citet{oikonomou2024stochastic}, which integrates adaptive stepsizes such as Polyak and AdaGrad into centralized stochastic momentum methods, one promising direction is to incorporate such adaptive stepsizes into our algorithms.
This could further  improve their applicability to deep neural network training tasks. 

\subsubsection*{Acknowledgements}
The research reported in this publication was supported by funding from King Abdullah University of Science and
Technology (KAUST): i) KAUST Baseline Research Scheme, ii) CRG Grant ORFS-CRG12-2024-6460, and iii) Center
of Excellence for Generative AI, under award number 5940.

\bibliography{aistats2026_refs}
\clearpage
\appendix
\thispagestyle{empty}

\onecolumn

\aistatstitle{Supplementary Materials}

\tableofcontents

\section{Key Inequalities}

In this section, we introduce basic lemmas for facilitating our convergence analysis. 


The first lemma, similar to \citet[Lemma 8]{nguyen2018sgd}, establishes an explicit expression for the iterates $\{e_t\}$ defined by the specific recursion. 

\begin{lemma}\label{lemma:summing_the_recursion}
Let the sequence $\{e_t\}$ be governed by 
$e_{t+1} = (1-\eta_t) e_t + (1-\eta_t) A_{t+1} + \eta_t B_{t+1}$ with $A_t,B_t \geq 0$. Then, $e_{t+1} = \prod_{\tau=0}^t (1-\eta_\tau) {e}_0 + \sum_{j=0}^t \left(\prod_{\tau=j+1}^t (1-\eta_\tau)\right) (1-\eta_j) {A}_{j+1} + \sum_{j=0}^t \left(\prod_{\tau=j+1}^t (1-\eta_\tau)\right) \eta_j B_{j+1}$.    
\end{lemma}

The next lemma provides the descent inequality for the normalized gradient descent iteration $x^{t+1}=x^t-\gamma_t \nicefrac{g^t}{\norm{g^t}}$.
\begin{lemma}[Descent Lemma]
    \label{lem:descent_lemma}
    Let~\Cref{assum:L_fcn} hold. Then for  the iterates $\{x^t\}$ generated by the following gradient update
    \begin{equation}
    \label{eq:grad_update}
        x^{t+1} = x^t - \gamma_t \frac{g^t}{\|g^t\|}
    \end{equation}
    with $g^t = \frac{1}{n}\sum_{i=1}^n g_i^t$ and $v^t = \frac{1}{n}\sum_{i=1}^n v_i^t$ 
    satisfy 
    \begin{equation}
    \label{eq:descent_lemma}
        \Delta_{t+1} + \gamma_t \|\nabla f(x^t)\| \le \Delta_t  + 2 \gamma_t \norm{v^t- \nabla f(x^t)} +\frac{ 2\gamma_t}{n}\sum_{i=1}^n\norm{g_i^t-v_i^t} + \frac{\gamma_t^2L}{2} ,
    \end{equation}
    where $\Delta_t \eqdef f(x^t) - f^{\inf}$ for any $t \geq 0$. 
\end{lemma}

\begin{proof}
Applying $L$-smoothness of $f(x)$ (\Cref{assum:L_fcn}) and~\eqref{eq:grad_update}, we have 
\begin{eqnarray*}
f(x^{t+1}) &\le& f(x^t) + \langle \nabla f(x^t), x^{t+1}-x^t \rangle + \frac{L}{2} \|x^{t+1}-x^t\|^2  \\
&=& f(x^t) - \gamma_t \langle \nabla f(x^t), g^t/\|g^t\| \rangle + \frac{L}{2} \gamma_t^2  \\
&\le& f(x^t) - \gamma_t \|g^t\| + \gamma_t \langle \nabla f(x^t) - g^t, g^t/\|g^t\| \rangle + \frac{L}{2} \gamma_t^2  \\
&\overset{ (a) }{\le}& f(x^t) - \gamma_t \|g^t\| - \gamma_t \norm{ \nabla f(x^t) - g^t} + \frac{L}{2} \gamma_t^2 \\
&\overset{(b)}{\le}& f(x^t) - \gamma_t \|\nabla f(x^t)\| + 2\gamma_t \|\nabla f(x^t) - g^t\| + \frac{L}{2} \gamma_t^2,
\end{eqnarray*}
where in $(a)$ we used Cauchy-Schwartz inequality, in $(b)$ we used Triangle inequality. Finally, by Triangle inequality, we have 
\begin{eqnarray*}
    \norm{\nabla f(x^t) - g^t} 
    & \leq &  \norm{v^t- \nabla f(x^t)} + \norm{g^t-v^t} \\
    & \leq &\norm{v^t- \nabla f(x^t)} + \frac{1}{n}\sum_{i=1}^n\norm{g_i^t-v_i^t}, 
\end{eqnarray*}
 by denoting $\Delta_{t+1} := f(x^{t+1}) - f^{\inf}$ for $t \in \{0, \dots, T-1\}$, we complete the proof. 
\end{proof}

Next, the following lemma bounds $\frac{1}{n}\sum_{i=1}^n\norm{g_i^t-v_i^t}$
from the iteration $g_i^{t+1}=g_i^t  + \cC_i^{t+1}(v_i^{t+1}-g_i^t)$. 
\begin{lemma}\label{lemma:boundg_and_v}
Let $g_i^{t+1}=g_i^t + \cC_i^{t+1}(v_i^{t+1}-g_i^t)$ for $i=1,2,\ldots,n$, where $\cC_i^t(\cdot)$ satisfies~\Cref{assum:contract_comp}. Then,  
\begin{eqnarray*}
   \Exp{\cV_{t+1}}  \leq \sqrt{1-\alpha}    \Exp{\cV_{t}} + \frac{\sqrt{1-\alpha}}{n} \sum_{i=1}^n \Exp{\norm{v_i^{t+1}-v_i^t}},
\end{eqnarray*}
where $\cV_t = \frac{1}{n}\sum_{i=1}^n \norm{g_i^t-v_i^t}$.
\end{lemma}
\begin{proof}
Define $\cV_t \eqdef \frac{1}{n}\sum_{i=1}^n \norm{g_i^t-v_i^t}$. Then, 
\begin{eqnarray*}
    \cV_{t+1} 
    & \overset{g_i^{t+1}}{=} & \frac{1}{n}\sum_{i=1}^n \norm{ v_i^{t+1} - g_i^t -  \mathcal{C}_i^{t+1}(v_i^{t+1} - g_i^t)}.
\end{eqnarray*}

Next, by taking the conditional expectation with fixed $\cF_{t+1} =\{v_i^{t+1},x^{t+1},g_i^t\}$, and by using Jensen's inequality, 
\begin{eqnarray*}
    \ExpCond{\cV_{t+1}}{\cF_{t+1}} 
      & \leq & \frac{1}{n}\sum_{i=1}^n\sqrt{ \ExpCond{\sqnorm{ v_i^{t+1} - g_i^t -  \mathcal{C}_i^{t+1}(v_i^{t+1} - g_i^t)}}{\cF_{t+1}}} \\
      & \overset{(a)}{\leq} & \frac{\sqrt{1-\alpha}}{n} \sum_{i=1}^n \sqrt{\ExpCond{\sqnorm{ v_i^{t+1} - g_i^t}}{\cF_{t+1}}} \\
      & = &  \frac{\sqrt{1-\alpha}}{n} \sum_{i=1}^n \norm{v_i^{t+1}-g_i^t} \\
      & \overset{(b)}{\leq} &   \frac{\sqrt{1-\alpha}}{n} \sum_{i=1}^n \norm{v_i^{t+1}-v_i^t} + \frac{\sqrt{1-\alpha}}{n} \sum_{i=1}^n \norm{v_i^t-g_i^t},
\end{eqnarray*}
where in $(a)$ we used in \Cref{assum:contract_comp}, in $(b)$ we used Triangle inequality. Therefore, 
\begin{eqnarray*}
   \Exp{\cV_{t+1}}  & = & \Exp{\ExpCond{\cV_{t+1}}{\cF_{t+1}}} \\
   & \leq &   \frac{\sqrt{1-\alpha}}{n} \sum_{i=1}^n \Exp{\norm{v_i^{t+1}-v_i^t}} + \frac{\sqrt{1-\alpha}}{n} \sum_{i=1}^n \Exp{\norm{v_i^t-g_i^t}}. 
\end{eqnarray*}

Finally, from the definition of $\cV_t$, we obtain the final result. 
\end{proof}

Finally, the following lemma presents an explicit upper bound on the summation involving the tuning parameters $\gamma_t,\eta_t$.
\begin{lemma}[Lemma 15 of \citet{fatkhullin2023stochastic}]\label{lemma:sum_prod}
   Let $q\in [0,1), p \geq 0, \gamma_0 >0$, and let $\eta_t = \left( \frac{2}{t+2} \right)^q$ and $\gamma_t = \left( \frac{2}{t+2} \right)^p$ for every integer $t$. Then, for any integer $t$ and $T \geq 1$, 
   \begin{eqnarray*}
       \sum_{t=0}^{T-1} \gamma_t \prod_{\tau=t+1}^{T-1} (1-\eta_\tau) \leq C(p,q) \frac{\gamma_T}{\eta_T},
   \end{eqnarray*}
   where $C(p,q) = 2^{p-q} (1-q)^{-1} t_0(p,q) \exp( 2^q (1-q) t_0 ^{1-q})  + 2^{2p+1-q} (1-q)^{-2}$ and $t_0(p,q) = \max\left\{ \left(  \frac{p}{(1-q) 2^q}\right)^{\frac{1}{1-q}} , 2 \frac{p-q}{(1-q)^2}  \right\}^{\frac{1}{1-q}}$.
\end{lemma}

\newpage 
\section{EF21-SGDM}\label{sec:EF21_SGDM}

In \algname{EF21-SGDM}, we update the iterates $\{x^t\}$ according to: 
\begin{eqnarray}
    x^{t+1} = x^t - \gamma_t \frac{g^t}{\norm{g^t}}, \notag
\end{eqnarray}
where $g^t = \frac{1}{n}\sum_{i=1}^n g_i^t$ and $v^t = \frac{1}{n}\sum_{i=1}^n v_i^t$ with $g_i^t,v_i^t$ being governed by 
\begin{eqnarray}
    g_i^{t+1} & = & g_i^t + \mathcal{C}_i^{t+1}(v_i^{t+1} - g_i^t), \quad \text{and} \notag\\
    v_i^{t+1} & = & (1-\eta_t)v_i^t + \eta_t \nabla f_i \left( x^{t+1}  ;\xi_i^{t+1} \right). \notag
\end{eqnarray}

\subsection{Convergence Proof}

We prove the result in the following steps. 

\subsubsection{Deriving the descent inequality} 

From \Cref{lem:descent_lemma}, we obtain \eqref{eq:descent_lemma}:
\begin{equation*}
    \Delta_{t+1} + \gamma_t \|\nabla f(x^t)\| \le \Delta_t  + 2 \gamma_t \norm{v^t- \nabla f(x^t)} +\frac{ 2\gamma_t}{n}\sum_{i=1}^n\norm{g_i^t-v_i^t} + \frac{\gamma_t^2L}{2}.
\end{equation*}

\subsubsection{Error bound I}

Define $\cV_t \eqdef \frac{1}{n}\sum_{i=1}^n \norm{g_i^t-v_i^t}$. From~\Cref{lemma:boundg_and_v}, 
%
%
\begin{eqnarray*}
   \Exp{\cV_{t+1}}  \leq \sqrt{1-\alpha}    \Exp{\cV_{t}} + \frac{\sqrt{1-\alpha}}{n} \sum_{i=1}^n \Exp{\norm{v_i^{t+1}-v_i^t}}.
\end{eqnarray*}

To complete the upper-bound for $ \Exp{\cV_{t+1}}$, we bound $\Exp{\norm{v_i^{t+1}-v_i^t}}$ as follows:
\begin{eqnarray*}
    \Exp{\norm{v_i^{t+1}-v_i^t}}
    & \leq & \eta_t \Exp{\norm{ \nabla f_i(x^{t+1};\xi_i^{t+1}) - v_i^t }} \\
    & \leq & \eta_t \Exp{\norm{v_i^t-\nabla f_i(x^t)}} + \eta_t \Exp{\norm{\nabla f_i(x^t) - \nabla f_i(x^{t+1})}} \\
    && + \eta_t \Exp{\norm{\nabla f_i(x^{t+1};\xi_i^{t+1}) - \nabla f_i(x^{t+1})}}.
\end{eqnarray*}

Next, by the $L_i$-smoothness of $f_i(\cdot)$, by Jensen's inequality, and by the fact that $\Exp{\sqnorm{\nabla f_i(x^{t+1};\xi_i^{t+1}) - \nabla f_i(x^{t+1})}} \leq \sigma_g^2$, 
\begin{eqnarray*}
    \Exp{\norm{v_i^{t+1}-v_i^t}}
   \leq \eta_t \Exp{\norm{v_i^t-\nabla f_i(x^t)}} + \eta_t L_i \gamma_t + \eta_t \sigma_g.
\end{eqnarray*}
Therefore, 
\begin{eqnarray*}
    \Exp{\cV_{t+1}}  \leq \sqrt{1-\alpha} \Exp{cV_t} + \sqrt{1-\alpha} \eta_t \Exp{\cU_t} + \sqrt{1-\alpha} \eta_t\gamma_t \bar L + \sqrt{1-\alpha} \eta_t \sigma_g,
\end{eqnarray*}
where $\bar L =\frac{1}{n}\sum_{i=1}^n L_i$.

\subsubsection{Error bound II}

We consider $\norm{v_i^{t+1} - \nabla f_i(x^{t+1})}$. From the definition of $v_i^{t+1}$, 
\begin{eqnarray*}
    v_i^{t+1} - \nabla f_i(x^{t+1})
    & = & (1-\eta_t)(v_i^t - \nabla f_i(x^t)) + (1-\eta_t)(\nabla f_i(x^{t}) - \nabla f_i(x^{t+1})) \\
    && + \eta_t (\nabla f_i(x^{t+1};\xi_i^{t+1}) - \nabla f_i(x^{t+1})).
\end{eqnarray*}

Next, define $\cU_t = \frac{1}{n}\sum_{i=1}^n \norm{v_i^t - \nabla f_i(x^t)}$. Then, by the triangle inequality, 
\begin{eqnarray*}
    \Exp{\cU_{t+1}} 
    & \leq & \frac{1}{n}\sum_{i=1}^n \Exp{ \norm{v_i^{t+1} - \nabla f_i(x^{t+1})} } \\
    & \leq & (1-\eta_t) \Exp{\cU_t} + (1-\eta_t) \frac{1}{n}\sum_{i=1}^n \Exp{\norm{ \nabla f_i(x^{t}) - \nabla f_i(x^{t+1} }}  \\
    && + \eta_t \frac{1}{n}\sum_{i=1}^n \Exp{\norm{ \nabla f_i(x^{t+1};\xi_i^{t+1}) - \nabla f_i(x^{t+1}) }}.
\end{eqnarray*}

Next, by the $L_i$-smoothness of $f_i(\cdot)$, by Jensen's inequality, and by the fact that $\Exp{\sqnorm{\nabla f_i(x^{t+1};\xi_i^{t+1}) - \nabla f_i(x^{t+1})}} \leq \sigma_g^2$,
\begin{eqnarray*}
    \Exp{\cU_{t+1}} 
    & \leq & (1-\eta_t) \Exp{\cU_t} + (1-\eta_t) \frac{1}{n}\sum_{i=1}^n  L_i \Exp{\norm{ x^{t} - x^{t+1} }} + \eta_t \sigma_g \\
    &\leq& (1-\eta_t) \Exp{\cU_t} + (1-\eta_t) \gamma_t \bar L + \eta_t \sigma_g,
\end{eqnarray*}
where in the last inequality we used the update rule for $x^{t+1}$.

\subsubsection{Error Bound III} 
From the recursion of $v_i - \nabla f_i(x^t)$, and by the fact that $v^t = \frac{1}{n}\sum_{i=1}^n v_i^t$ and $\nabla f(x) = \frac{1}{n}\sum_{i=1}^n \nabla f_i(x)$,
\begin{eqnarray*}
    v^{t+1} - \nabla f(x^{t+1})
    & = & (1-\eta_t)(v^t - \nabla f(x^t)) + (1-\eta_t)(\nabla f(x^{t}) - \nabla f(x^{t+1})) \\
    && + \eta_t \frac{1}{n}\sum_{i=1}^n [\nabla f_i(x^{t+1};\xi_i^{t+1}) - \nabla f_i(x^{t+1})].
\end{eqnarray*}

Next, from~\Cref{lemma:summing_the_recursion}, 
\begin{eqnarray*}
 e_{t+1} = \prod_{\tau=0}^t (1-\eta_\tau) {e}_0 + \sum_{j=0}^t \left(\prod_{\tau=j+1}^t (1-\eta_\tau)\right) (1-\eta_j) {A}_{j+1} + \sum_{j=0}^t \left(\prod_{\tau=j+1}^t (1-\eta_\tau)\right) \eta_j B_{j+1},
\end{eqnarray*}
where $e_t = v^t - \nabla f(x^t)$, $A_{t+1} = \nabla f(x^t)-\nabla f(x^{t+1})$, and $B_{t+1} = \frac{1}{n}\sum_{i=1}^n [\nabla f_i(x^{t+1};\xi_i^{t+1}) - \nabla f_i(x^{t+1})]$.

Therefore, from the definition of the Euclidean norm and by the triangle inequality, 
\begin{eqnarray*}
    \Exp{\norm{e_{t+1}}} 
    & \leq & \prod_{\tau=0}^t (1-\eta_\tau) \Exp{\norm{{e}_0}} + \sum_{j=0}^t \left(\prod_{\tau=j+1}^t (1-\eta_\tau)\right) (1-\eta_j) \Exp{\norm{{A}_{j+1}}} \\
    && + \Exp{\norm{\sum_{j=0}^t \left(\prod_{\tau=j+1}^t (1-\eta_\tau)\right) \eta_j B_{j+1}}}.
\end{eqnarray*}

Next, by the $L_i$-smoothness of $f_i(\cdot)$ and by the definition of $x^{t+1}$, we can show that $\norm{A_{j+1}} \leq \bar L \gamma_j$. Therefore, 
\begin{eqnarray*}
    \Exp{\norm{e_{t+1}}} 
    & \leq & \prod_{\tau=0}^t (1-\eta_\tau) \Exp{\norm{{e}_0}} + \bar L \cdot \sum_{j=0}^t \left(\prod_{\tau=j+1}^t (1-\eta_\tau)\right) (1-\eta_j) \gamma_j  \\
    && + \Exp{\norm{\sum_{j=0}^t \left(\prod_{\tau=j+1}^t (1-\eta_\tau)\right) \eta_j B_{j+1}}}.
\end{eqnarray*}

Next, since 
\begin{eqnarray*}
    \Exp{\norm{\sum_{j=0}^t \left(\prod_{\tau=j+1}^t (1-\eta_\tau)\right) \eta_j B_{j+1}}} 
    & \overset{(a)}{\leq} & \sqrt{  \Exp{\sqnorm{\sum_{j=0}^t \left(\prod_{\tau=j+1}^t (1-\eta_\tau)\right) \eta_j B_{j+1}}} } \\
    & \overset{(b)}{\leq} &  \frac{\sigma_g}{\sqrt{n}} \sqrt{ \sum_{j=0}^t \left(\prod_{\tau=j+1}^t (1-\eta_\tau)^2\right) \eta_j^2 },
\end{eqnarray*}
where in $(a)$ we used Jensen's inequality, in $(b)$ we used \Cref{assum:stoc_g_h},
we obtain
\begin{eqnarray*}
    \Exp{\norm{e_{t+1}}} 
    & \leq & \prod_{\tau=0}^t (1-\eta_\tau) \Exp{\norm{{e}_0}} + \bar L \cdot \sum_{j=0}^t \left(\prod_{\tau=j+1}^t (1-\eta_\tau)\right) (1-\eta_j) \gamma_j  \\
    && + \frac{\sigma_g}{\sqrt{n}} \sqrt{ \sum_{j=0}^t \left(\prod_{\tau=j+1}^t (1-\eta_\tau)^2\right) \eta_j^2 }\\
     & \leq & \prod_{\tau=0}^t (1-\eta_\tau) \Exp{\norm{{e}_0}} + \bar L \cdot \sum_{j=0}^t \left(\prod_{\tau=j+1}^t (1-\eta_\tau)\right)  \gamma_j  \\
    && + \frac{\sigma_g}{\sqrt{n}} \sqrt{ \sum_{j=0}^t \left(\prod_{\tau=j+1}^t (1-\eta_\tau)\right) \eta_j^2 },
\end{eqnarray*}
where in the last inequality we used $1-\eta_j \leq 1$.

If $\eta_t = \left( \frac{2}{t+2} \right)^{1/2}$ and $\gamma_t = \gamma_0 \left( \frac{2}{t+2} \right)^{3/4}$, then we can prove  that $\eta_0=1$, and that
\begin{eqnarray*}
    \Exp{\norm{e_{t+1}}} 
    & \leq &  \bar L \cdot \sum_{j=0}^t \left(\prod_{\tau=j+1}^t (1-\eta_\tau)\right)  \gamma_j  + \frac{\sigma_g}{\sqrt{n}} \sqrt{ \sum_{j=0}^t \left(\prod_{\tau=j+1}^t (1-\eta_\tau)\right) \eta_j^2 }.
\end{eqnarray*}
In conclusion, from~\Cref{lemma:sum_prod}, 
\begin{eqnarray*}
    \Exp{\norm{v^{t+1}  - \nabla f(x^{t+1})}} 
    & \leq &  \bar L C(\nicefrac{3}{4},\nicefrac{1}{2}) \frac{\gamma_{t+1}}{\eta_{t+1}}   + \frac{\sigma_g}{\sqrt{n}} \sqrt{ C(1,\nicefrac{1}{2}) \eta_{t+1} }.
\end{eqnarray*}

\subsubsection{Bounding Lyapunov function }

Define Lyapunov function as  $V_t = \Delta_{t} + C_{1,t} \cV_t + C_{2,t} \cU_t$ with $C_{1,t} = \frac{2\gamma_t}{1-\sqrt{1-\alpha}}$ and $C_{2,t} = \frac{2\gamma_t \sqrt{1-\alpha}}{1-\sqrt{1-\alpha}}$. Then, 
\begin{eqnarray*}
    \Exp{V_{t+1}} 
    & = & \Exp{\Delta_{t+1} + C_{1,t+1} \cV_t + C_{2,t+1} \cU_{t+1}} \\
    & \leq & \Exp{\Delta_t} - \gamma_t \Exp{\norm{\nabla f(x^t)}} + 2 \gamma_t \Exp{\norm{v^t- \nabla f(x^t)}} + 2\gamma_t \Exp{\cV_t} + \frac{\gamma_t^2L}{2} \\
    && + C_{1,t+1} \Exp{\cV_{t+1}} + C_{2,t+1} \Exp{\cU_{t+1}},
\end{eqnarray*}
where in the last inequality we used \eqref{eq:descent_lemma}.
Next, by the upper-bounds for $\Exp{\cV_{t+1}}$ and  $\Exp{\cU_{t+1}}$, 
\begin{eqnarray*}
    \Exp{ V_{t+1} }
   & \leq & \Exp{\Delta_t} - \gamma_t \Exp{\norm{\nabla f(x^t)}} + 2 \gamma_t \Exp{\norm{v^t- \nabla f(x^t)}} + 2\gamma_t \Exp{\cV_t} + \frac{\gamma_t^2L}{2} \\
    && + C_{1,t+1} (\sqrt{1-\alpha} \Exp{\cV_t} + \sqrt{1-\alpha} \eta_t \Exp{\cU_t} + \sqrt{1-\alpha} \eta_t\gamma_t \bar L + \sqrt{1-\alpha} \eta_t \sigma_g) \\
    && + C_{2,t+1} ((1-\eta_t) \Exp{\cU_t} + (1-\eta_t) \gamma_t \bar L + \eta_t \sigma_g).
\end{eqnarray*}

Next,  since $\gamma_{t+1} \leq \gamma_t$, we can prove that $C_{1,t+1} \leq C_{1,t}$ and that $C_{2,t+1} \leq C_{2,t}$, and also that 
\begin{eqnarray*}
    2\gamma_t + C_{1,t+1}\sqrt{1-\alpha}  & \leq & 2\gamma_t + C_{1,t}\sqrt{1-\alpha} = C_{1,t} \\
    C_{1,t+1} \sqrt{1-\alpha}\eta_t + C_{2,t+1} (1-\eta_t) & \leq& C_{1,t} \sqrt{1-\alpha}\eta_t + C_{2,t} (1-\eta_t)= C_{2,t}
\end{eqnarray*}
Therefore, 
\begin{eqnarray*}
   \Exp{ V_{t+1} } 
   & \leq & \Exp{V_t} - \gamma_t \Exp{\norm{\nabla f(x^t)}} + 2 \gamma_t \Exp{\norm{v^t- \nabla f(x^t)}}  + \frac{\gamma_t^2L}{2} \\
    && + C_{1,t} (\sqrt{1-\alpha} \eta_t\gamma_t \bar L + \sqrt{1-\alpha} \eta_t \sigma_g) \\
    && + C_{2,t} ( (1-\eta_t) \gamma_t \bar L + \eta_t \sigma_g) \\
     & \leq & \Exp{V_t} - \gamma_t \Exp{\norm{\nabla f(x^t)}} + 2 \gamma_t \Exp{\norm{v^t- \nabla f(x^t)}}   \\
    && + \gamma_t^2 \left( \frac{2\sqrt{1-\alpha}}{1-\sqrt{1-\alpha}} \bar L + \frac{L}{2} \right) + \eta_t\gamma_t \frac{4\sqrt{1-\alpha}}{1-\sqrt{1-\alpha}} \sigma_g.
\end{eqnarray*}

Next, from the upper-bound of $\Exp{\norm{v^t- \nabla f(x^t)}}$, 
\begin{eqnarray*}
    \Exp{ V_{t+1} } 
    & \leq & \Exp{V_t} - \gamma_t \Exp{\norm{\nabla f(x^t)}} + 2 \gamma_t \left( \bar L C(\nicefrac{3}{4},\nicefrac{1}{2}) \frac{\gamma_{t}}{\eta_{t}}   + \frac{\sigma_g}{\sqrt{n}} \sqrt{ C(1,\nicefrac{1}{2}) \eta_{t} } \right)  \\
    && + \gamma_t^2 \left( \frac{2\sqrt{1-\alpha}}{1-\sqrt{1-\alpha}} \bar L + \frac{L}{2} \right) + \eta_t\gamma_t \frac{4\sqrt{1-\alpha}}{1-\sqrt{1-\alpha}} \sigma_g \\
    & = & \Exp{V_t} - \gamma_t \Exp{\norm{\nabla f(x^t)}} + \frac{\gamma_t^2}{\eta_t} \cdot 2 \bar L C_1 + \gamma_t \sqrt{\eta_t} \cdot \frac{2 \sigma_g C_2}{\sqrt{n}}  \\
    && + \gamma_t^2 \left( \frac{2\sqrt{1-\alpha}}{1-\sqrt{1-\alpha}} \bar L + \frac{L}{2} \right) +  \eta_t\gamma_t \cdot  \frac{4\sqrt{1-\alpha}}{1-\sqrt{1-\alpha}}  \sigma_g,
\end{eqnarray*}
where  we denoted $C_1 = C(\nicefrac{3}{4}, \nicefrac{1}{2})$, $C_2 = \sqrt{C(1, \nicefrac{1}{2})}$.

\subsubsection{Deriving the convergence rate}
By re-arranging the terms and by the telescopic series, 
\begin{eqnarray*}
    \frac{\sum_{t=0}^{T-1} \gamma_t\Exp{\norm{\nabla f(x^t)}}}{\sum_{t=0}^{T-1} \gamma_t}  
    & \leq & \frac{\Exp{V_0} -\Exp{V_T}}{\sum_{t=0}^{T-1}\gamma_t}  + 2 \bar L C_1 \frac{\sum_{t=0}^{T-1} \gamma_t^2 \eta_t^{-1}}{\sum_{t=0}^{T-1}\gamma_t} + \frac{2\sigma_g C_2}{\sqrt{n}} \frac{\sum_{t=0}^{T-1} \gamma_t \sqrt{\eta_t}}{ \sum_{t=0}^{T-1} \gamma_t} \\
    && +  \left( \frac{2\sqrt{1-\alpha}}{1-\sqrt{1-\alpha}} \bar L + \frac{L}{2} \right) \frac{\sum_{t=0}^{T-1} \gamma_t^2}{ \sum_{t=0}^{T-1} \gamma_t} +     \frac{4\sqrt{1-\alpha}}{1-\sqrt{1-\alpha}}  \sigma_g \frac{\sum_{t=0}^{T-1} \eta_t\gamma_t}{ \sum_{t=0}^{T-1}\gamma_t}\\
    &\leq& \frac{\Exp{V_0}}{\sum_{t=0}^{T-1}\gamma_t} + 2 \bar L C_1 \frac{\sum_{t=0}^{T-1} \gamma_t^2\eta_t^{-1}}{\sum_{t=0}^{T-1}\gamma_t} +  \left( \frac{2}{\alpha^2} \bar L + \frac{L}{2} \right) \frac{\sum_{t=0}^{T-1} \gamma_t^2}{ \sum_{t=0}^{T-1} \gamma_t}\\
    && + \frac{2\sigma_g C_2}{\sqrt{n}} \frac{\sum_{t=0}^{T-1} \gamma_t \sqrt{\eta_t}}{ \sum_{t=0}^{T-1} \gamma_t} +  \frac{4\sigma_g}{\alpha^2}  \frac{\sum_{t=0}^{T-1} \eta_t\gamma_t}{ \sum_{t=0}^{T-1}\gamma_t},
\end{eqnarray*}
where in the last inequality we used $\frac{\sqrt{1-\alpha}}{1-\sqrt{1-\alpha}} \leq \frac{1}{\alpha^2}$ for any $\alpha \in (0,1]$.
To continue the proof, we need to provide bounds on several sum-type terms. Since we set $\eta_t = \left( \frac{2}{t+2} \right)^{\nicefrac{1}{2}}$ and $\gamma_t = \gamma_0 \left( \frac{2}{t+2} \right)^{\nicefrac{3}{4}}$, we have 
\begin{eqnarray*}
    \sum_{t=0}^{T-1} \gamma_t &\geq&  T \gamma_{T-1} = \gamma_0T  \left(\frac{2}{T+1}\right)^{\nicefrac{3}{4}} \geq \gamma_0 T^{\nicefrac{1}{4}};\\
    \sum_{t=0}^{T-1} \gamma_t^2 \eta_t^{-1}  &=& \gamma_0^2\sum_{t=0}^{T-1}  \left( \frac{2}{t+2} \right)^{\nicefrac{3}{2}} \left( \frac{2}{t+2} \right)^{-\nicefrac{1}{2} } = 2\gamma_0^2 \sum_{t=0}^{T-1} \frac{1}{t+2}   = 2\gamma_0^2 \sum^{T}_{t=1}\frac{1}{1+t}\\
    &\leq& 2\gamma^2_0 \int^T_{1}\frac{1}{1+t}dt = 2\gamma^2_0 \left(\log \left(T+1\right) -\log(2)\right) \leq 2\gamma^2_0 \log \left(T+1\right); \\
    \sum_{t=0}^{T-1} \gamma_t \sqrt{\eta_t} &=& \gamma_0 \sum_{t=0}^{T-1} \left( \frac{2}{t+2} \right)^{\nicefrac{3}{4}} \left( \frac{2}{t+2} \right)^{\nicefrac{1}{4} } =   2\gamma_0 \sum_{t=0}^{T-1}\frac{1}{t+2} \leq 2\gamma_0 \log\left(T+1\right);\\
    \sum_{t=0}^{T-1} \gamma_t^2 &=& \gamma_0^2 \sum_{t=0}^{T-1} \left( \frac{2}{t+2} \right)^{\nicefrac{3}{2}}  \leq 2^{\nicefrac{3}{2}}\gamma_0^2 \sum_{t=0}^{T-1} \frac{1}{(t+2)^{\nicefrac{3}{2}}} = 2^{\nicefrac{3}{2}}\gamma_0^2 \sum_{t=1}^{T} \frac{1}{(1+t)^{\nicefrac{3}{2}}}\\
    &\leq& 2^{\nicefrac{3}{2}}\gamma_0^2 \int^T_{1}\frac{1}{(1+t)^{\nicefrac{3}{2}}}dt  = 2^{\nicefrac{5}{2}}\gamma^2_0 \left(\frac{1}{\sqrt{2}} - \frac{1}{\sqrt{T+1}}\right) \leq 4\gamma^2_0;\\
    \sum_{t=0}^{T-1} \gamma_t \eta_t &=& \gamma_0 \sum_{t=0}^{T-1}  \left( \frac{2}{t+2} \right)^{\nicefrac{3}{4}} \left( \frac{2}{t+2} \right)^{\nicefrac{1}{2} } = 2^{\nicefrac{5}{4}}\gamma_0 \sum^{T-1}_{t=0} \frac{1}{(t+2)^{\nicefrac{5}{4}}} = 2^{\nicefrac{5}{4}}\gamma_0 \sum^{T}_{t=1} \frac{1}{(t+1)^{\nicefrac{5}{4}}}\\
    &\leq& 2^{\nicefrac{5}{4}}\gamma_0\int^T_{1}\frac{1}{(1+t)^{\nicefrac{5}{4}}}dt = 4\cdot 2^{\nicefrac{5}{4}} \gamma_0\left(\frac{1}{2^{\nicefrac{1}{4}}} - \frac{1}{(T+1)^{\nicefrac{1}{4}}}\right) \leq 8\gamma_0.
\end{eqnarray*}
Therefore, denoting $\tilde{x}^T$ as a point randomly chosen from $\{x_0,x_1,\ldots,x_{T-1}\}$ with probability $\nicefrac{\gamma_t}{\sum_{t=0}^{T-1}\gamma_t}$ for $t=0,1,\ldots,T-1$, we obtain 
\begin{eqnarray*}
     \Exp{\norm{\nabla f(\tilde{x}^T)}} 
    & \leq & \frac{\Exp{V_0}}{\gamma_0 T^{\nicefrac{1}{4}}} + 2 \bar L C_1 \frac{2\gamma^2_0 \log \left(T+1\right)}{\gamma_0 T^{\nicefrac{1}{4}}} +  \left( \frac{2}{\alpha^2} \bar L + \frac{L}{2} \right) \frac{4\gamma^2_0}{ \gamma_0 T^{\nicefrac{1}{4}}}\\
    && + \frac{2\sigma_g C_2}{\sqrt{n}} \frac{2\gamma_0 \log \left(T+1\right)}{ \gamma_0 T^{\nicefrac{1}{4}}} +  \frac{4\sigma_g}{\alpha^2}  \frac{8\gamma_0}{ \gamma_0 T^{\nicefrac{1}{4}}}\\
    &=& \frac{\Exp{V_0}}{\gamma_0 T^{\nicefrac{1}{4}}} + \frac{2\gamma_0}{T^{\nicefrac{1}{4}}}\left(L + 4C_1{\bar L}\log\left(T+1\right) + \frac{4\bar L}{\alpha^2}\right)\\
    && +\frac{4}{T^{\nicefrac{1}{4}}}\left(C_2 \frac{\sigma_g}{\sqrt{n}}\log(T+1) +\frac{8\sigma_g}{\alpha^2} \right)\\
    &=& \widetilde{\cO}\left(\frac{\nicefrac{\Exp{V_0}}{\gamma_0} + \gamma_0\left(L + \nicefrac{\bar L}{\alpha^2}\right) + \sigma_g\left(\nicefrac{1}{\sqrt{n}} + \nicefrac{1}{\alpha^2}\right)}{T^{\nicefrac{1}{4}}}\right).
\end{eqnarray*}

\newpage 
\section{EF21-IGT}

In \algname{EF21-IGT}, we update the iterates $\{x^t\}$ according to: 
\begin{eqnarray}
    x^{t+1} = x^t - \gamma_t \frac{g^t}{\norm{g^t}}, \notag
\end{eqnarray}
where $g^t = \frac{1}{n}\sum_{i=1}^n g_i^t$ and $v^t = \frac{1}{n}\sum_{i=1}^n v_i^t$ with $g_i^t,v_i^t$ being governed by 
\begin{eqnarray}
    g_i^{t+1} & = & g_i^t + \mathcal{C}_i^{t+1}(v_i^{t+1} - g_i^t), \quad \text{and} \notag\\
    v_i^{t+1} & = & (1-\eta_t)v_i^t + \eta_t \nabla f_i \left( x^{t+1} + \frac{1-\eta_t}{\eta_t} (x^{t+1}-x^t) ;\xi_i^{t+1} \right). \notag
\end{eqnarray}
For simplicity, we denote $y^{t+1} = x^{t+1} +\theta_t(x^{t+1} - x^t)$, where $\theta_t = \frac{1-\eta_t}{\eta_t}$.

\subsection{Convergence Proof}

We prove the result in the following steps. 

\subsubsection{Deriving the descent inequality} From \Cref{lem:descent_lemma}, we obtain \eqref{eq:descent_lemma}:
\begin{equation*}
    \Delta_{t+1} + \gamma_t \|\nabla f(x^t)\| \le \Delta_t  + 2 \gamma_t \norm{v^t- \nabla f(x^t)} +\frac{ 2\gamma_t}{n}\sum_{i=1}^n\norm{g_i^t-v_i^t} + \frac{\gamma_t^2L}{2}.
\end{equation*}

\subsubsection{Error bound I} 

Define $\cV_t \eqdef \frac{1}{n}\sum_{i=1}^n \norm{g_i^t-v_i^t}$. From~\Cref{lemma:boundg_and_v}, 
\begin{eqnarray*}
   \Exp{\cV_{t+1}}  \leq \sqrt{1-\alpha}    \Exp{\cV_{t}} + \frac{\sqrt{1-\alpha}}{n} \sum_{i=1}^n \Exp{\norm{v_i^{t+1}-v_i^t}}.
\end{eqnarray*}

To finalize the upper-bound for $\Exp{\cV_{t+1}}$, we bound $\Exp{\norm{v_i^{t+1}-v_i^t}}$.
\begin{eqnarray*}
    \Exp{\norm{v_i^{t+1}-v_i^t}}
    & = & \eta_t \Exp{\norm{ \nabla f_i(y^{t+1};\xi_i^{t+1}) - v_i^t }} \\
    & \leq & \eta_t \Exp{\norm{v_i^t-\nabla f_i(x^t)}} + \eta_t \Exp{\norm{\nabla f_i(x^t) - \nabla f_i(y^{t+1})}} \\
     && + \eta_t \Exp{\norm{\nabla f_i(y^{t+1};\xi_i^{t+1}) - \nabla f_i(y^{t+1})}}. 
\end{eqnarray*}

Next, by the $L_i$-smoothness of $f_i$ and from the definition of $y^{t+1}$
\begin{eqnarray*}
  \Exp{\norm{v_i^{t+1}-v_i^t}}  
    & \leq & \eta_t \Exp{\norm{v_i^t-\nabla f_i(x^t)}} + \eta_t L_i \Exp{\norm{ x^t - y^{t+1}}} \\
    && + \eta_t \Exp{\norm{\nabla f_i(y^{t+1};\xi_i^{t+1}) - \nabla f_i(y^{t+1})}} \\
    & \leq & \eta_t \Exp{\norm{v_i^t-\nabla f_i(x^t)}} + \eta_t\left(1+ \theta_t \right) L_i \Exp{\norm{ x^t -  x^{t+1}}} \\
    && + \eta_t \Exp{\norm{\nabla f_i(y^{t+1};\xi_i^{t+1}) - \nabla f_i(y^{t+1})}} \\
    & \leq & \eta_t \Exp{\norm{v_i^t-\nabla f_i(x^t)}} + \gamma_t L_i \\
    && + \eta_t \Exp{\norm{\nabla f_i(y^{t+1};\xi_i^{t+1}) - \nabla f_i(y^{t+1})}},
\end{eqnarray*}
where in the last inequality we used the update rule for $x_{t+1}$ and $\theta_t = \frac{1-\eta_t}{\eta_t}$. 
Next, by~\Cref{assum:stoc_g_h}, 
\begin{eqnarray*}
  \Exp{\norm{v_i^{t+1}-v_i^t}}
  \leq \eta_t \Exp{\norm{v_i^t-\nabla f_i(x^t)}} +  L_i \gamma_t + \eta_t \sigma_g. 
\end{eqnarray*}

Therefore,
\begin{eqnarray*}
   \Exp{\cV_{t+1}}  \leq \sqrt{1-\alpha}    \Exp{\cV_{t}} + \sqrt{1-\alpha} \eta_t \Exp{\cU_t} + \sqrt{1-\alpha} \bar L \gamma_t  + \sqrt{1-\alpha} \sigma_g \eta_t.
\end{eqnarray*}

\subsubsection{Error bound II}

We consider $\norm{v_i^{t+1} - \nabla f_i(x^{t+1})}$. From the definition of $v_i^{t+1}$, 
\begin{eqnarray*}
    v_i^{t+1} - \nabla f_i(x^{t+1})
    & = & (1-\eta_t)(v_i^t - \nabla f_i(x^t)) + (1-\eta_t)(\nabla f_i(x^{t}) - \nabla f_i(x^{t+1})) \\
    && + \eta_t(\nabla f_i(y^{t+1}) -\nabla f_i(x^{t+1})) + \eta_t (\nabla f_i(y^{t+1};\xi_i^{t+1}) - \nabla f_i(y^{t+1}))\\
    & = & (1-\eta_t)(v_i^t - \nabla f_i(x^t)) + (1-\eta_t)(\nabla f_i(x^{t}) - \nabla f_i(x^{t+1})) \\
    && + \eta_t(\nabla f_i(y^{t+1}) -\nabla f_i(x^{t+1})) + \eta_t (\nabla f_i(y^{t+1};\xi_i^{t+1}) - \nabla f_i(y^{t+1}))\\
    &&+ (1-\eta_t) \cdot \nabla^2 f_i(x^{t+1})(x^t-x^{t+1}) - (1-\eta_t) \cdot \nabla^2 f_i(x^{t+1})(x^t-x^{t+1}). 
\end{eqnarray*}

Next, by the fact that $y^{t+1}-x^{t+1} = \theta_t(x^{t+1}-x^t) = \frac{1-\eta_t}{\eta_t}(x^{t+1}-x^t)$, we can prove that 
\begin{eqnarray*}
    (1-\eta_t) \cdot \nabla^2 f_i(x^{t+1})(x^t-x^{t+1}) 
    &=& \eta_t \cdot \frac{1-\eta_t}{\eta_t} \nabla^2 f_i(x^{t+1})(x^t-x^{t+1}) \\
    &=& -\eta_t \nabla^2 f_i(x^{t+1})(y^{t+1}-x^{t+1}).
\end{eqnarray*}
Therefore, 
\begin{eqnarray*}
    v_i^{t+1} - \nabla f_i(x^{t+1})
    & = & (1-\eta_t)(v_i^t - \nabla f_i(x^t)) \\
    && + (1-\eta_t) Z_{f_i}(x^t,x^{t+1}) + \eta_t Z_{f_i}(y^{t+1},x^{t+1})\\
    && + \eta_t (\nabla f_i(y^{t+1};\xi_i^{t+1}) - \nabla f_i(y^{t+1})),
\end{eqnarray*}
where $Z_{f_i}(x,y) :=\nabla f_i(x) - \nabla f_i(y) -\nabla^2 f_i(y)(x-y)$.

Next, define $\cU_t = \frac{1}{n}\sum_{i=1}^n \norm{v_i^t - \nabla f_i(x^t)}$. Then, by the triangle inequality, 
\begin{eqnarray*}
    \Exp{\cU_{t+1}} 
    & \leq & \frac{1}{n}\sum_{i=1}^n \Exp{ \norm{v_i^{t+1} - \nabla f_i(x^{t+1})} } \\
    & \leq & (1-\eta_t) \Exp{\cU_t} + \eta_t \frac{1}{n}\sum_{i=1}^n \Exp{\norm{ \nabla f_i(y^{t+1};\xi_i^{t+1}) - \nabla f_i(y^{t+1}) }} \\
    && + (1-\eta_t) \frac{1}{n}\sum_{i=1}^n \Exp{\norm{ Z_{f_i}(x^t,x^{t+1}) }}  + \eta_t  \frac{1}{n}\sum_{i=1}^n \Exp{\norm{ Z_{f_i}(y^{t+1},x^{t+1}) }}.
\end{eqnarray*}

Next, by the fact that $\Exp{\sqnorm{\nabla f_i(x^{t+1};\xi_i^{t+1}) - \nabla f_i(x^{t+1})}} \leq \sigma_g^2$,
\begin{eqnarray*}
    \Exp{\cU_{t+1}} 
    & \leq & (1-\eta_t) \Exp{\cU_t} + \eta_t \sigma_g \\
    && + (1-\eta_t) \frac{1}{n}\sum_{i=1}^n \Exp{\norm{ Z_{f_i}(x^t,x^{t+1}) }}  + \eta_t  \frac{1}{n}\sum_{i=1}^n \Exp{\norm{ Z_{f_i}(y^{t+1},x^{t+1}) }}.
\end{eqnarray*}

By the $L_{h,i}$-Lipschiz continuity of $\nabla^2 f_i(\cdot)$, i.e. $\norm{Z_{f_i}(x,y)} \leq \frac{L_{h,i}}{2} \sqnorm{x-y}$,
\begin{eqnarray*}
    \Exp{\cU_{t+1}} 
    & \leq & (1-\eta_t) \Exp{\cU_t} + \eta_t \sigma_g \\
    && + (1-\eta_t) \frac{1}{n}\sum_{i=1}^n \frac{L_{h,i}}{2}\Exp{\sqnorm{x^{t} -x^{t+1}}}  + \eta_t  \frac{1}{n}\sum_{i=1}^n \frac{L_{h,i}}{2}\Exp{\sqnorm{y^{t+1} -x^{t+1}}} \\
    & \overset{(a)}{\leq}& (1-\eta_t) \Exp{\cU_t} + \eta_t \sigma_g \\
    && + (1-\eta_t) \frac{1}{n}\sum_{i=1}^n \frac{L_{h,i}}{2}\Exp{\sqnorm{x^{t} -x^{t+1}}}  + \eta_t \theta_t^2 \frac{1}{n}\sum_{i=1}^n \frac{L_{h,i}}{2}\Exp{\sqnorm{x^{t+1} -x^{t}}} \\
    & \overset{x^{t+1}}{\leq}& (1-\eta_t) \Exp{\cU_t} + \eta_t \sigma_g + (1-\eta_t) \gamma_t^2 \frac{\bar L_{h}}{2}  + \eta_t \theta_t^2 \gamma_t^2  \frac{\bar L_{h}}{2} \\
    & \leq &  (1-\eta_t) \Exp{\cU_t} + \eta_t \sigma_g +  \frac{1-\eta_t}{\eta_t}\gamma_t^2  \frac{\bar L_{h}}{2},
\end{eqnarray*}
where in $(a)$ we used the definition of $y^{t+1}$, in $(b)$ we denoted $\bar L_h = \frac{1}{n}\sum_{i=1}^n L_{h,i}$ and used the update rule for $x^{t+1}$.

\subsubsection{Error bound III} 

From the recursion of $v_i - \nabla f_i(x^t)$, and by the fact that $v^t = \frac{1}{n}\sum_{i=1}^n v_i^t$ and $\nabla f(x) = \frac{1}{n}\sum_{i=1}^n \nabla f_i(x)$,
\begin{eqnarray*}
    v^{t+1} - \nabla f(x^{t+1})
    & = & (1-\eta_t)(v^t - \nabla f(x^t)) \\
    && + (1-\eta_t)   Z_{f}(x^t,x^{t+1}) + \eta_t Z_{f}(y^{t+1},x^{t+1})\\
    && + \eta_t \frac{1}{n}\sum_{i=1}^n (\nabla f_i(y^{t+1};\xi_i^{t+1}) - \nabla f_i(y^{t+1})). 
\end{eqnarray*}

Next, from~\Cref{lemma:summing_the_recursion}, 
\begin{eqnarray*}
 e_{t+1} 
 & = & \prod_{\tau=0}^t (1-\eta_\tau) {e}_0 + \sum_{j=0}^t \left(\prod_{\tau=j+1}^t (1-\eta_\tau)\right) (1-\eta_j) {A}_{j+1} \\
 && + \sum_{j=0}^t \left(\prod_{\tau=j+1}^t (1-\eta_\tau)\right) \eta_j B_{j+1}  + \sum_{j=0}^t \left(\prod_{\tau=j+1}^t (1-\eta_\tau)\right) \eta_j C_{j+1},
\end{eqnarray*}
where we introduced the following notation: 
\begin{equation*}
     e_t \eqdef v^t - \nabla f(x^t), \quad  A_{t+1} \eqdef  Z_{f}(x^t,x^{t+1}),\quad B_{t+1} \eqdef  Z_{f}(y^{t+1},x^{t+1}), 
\end{equation*}
and $$C_{t+1} \eqdef \frac{1}{n}\sum_{i=1}^n [\nabla f_i(y^{t+1};\xi_i^{t+1}) - \nabla f_i(y^{t+1})].$$

Therefore, 
from the definition of the Euclidean norm and by the triangle inequality, 
\begin{eqnarray*}
\Exp{\norm{e_{t+1}}} 
    & \leq & \prod_{\tau=0}^t (1-\eta_\tau) \Exp{\norm{{e}_0}} + \sum_{j=0}^t \left(\prod_{\tau=j+1}^t (1-\eta_\tau)\right) (1-\eta_j) \Exp{\norm{{A}_{j+1}}} \\
    && + \sum_{j=0}^t \left(\prod_{\tau=j+1}^t (1-\eta_\tau)\right) \eta_j \Exp{\norm{B_{j+1}}}\\
    &&+ \Exp{\norm{\sum_{j=0}^t \left(\prod_{\tau=j+1}^t (1-\eta_\tau)\right) \eta_j C_{j+1}}}.
\end{eqnarray*}

Next, by $L_{h}$-Lipschitz continuity of $\nabla^2 f(\cdot)$,
\begin{eqnarray*}
    \norm{A_{j+1}} & \leq & \frac{L_{h}}{2}\sqnorm{x^j-x^{j+1}} \leq \frac{L_{h}}{2} \gamma_j^2, \quad \text{and} \\
    \norm{B_{j+1}} & \leq & \frac{L_{h}}{2}\sqnorm{y^{j+1}-x^{j+1}} \leq \frac{L_{h}}{2} \theta_j^2\gamma_j^2 = \frac{L_{h}}{2} \frac{(1-\eta_j)^2} {\eta_j^2} \gamma_j^2. 
\end{eqnarray*}
Therefore, 
\begin{eqnarray*}
    \Exp{\norm{e_{t+1}}} 
    & \leq & \prod_{\tau=0}^t (1-\eta_\tau) \Exp{\norm{{e}_0}} + \frac{L_h}{2}\sum_{j=0}^t \left(\prod_{\tau=j+1}^t (1-\eta_\tau)\right) (1-\eta_j) \gamma_j^2 \\
    && + \frac{L_h}{2}\sum_{j=0}^t \left(\prod_{\tau=j+1}^t (1-\eta_\tau)\right) \frac{(1-\eta_j)^2}{\eta_j} \gamma_j^2 \\
    && + \Exp{\norm{\sum_{j=0}^t \left(\prod_{\tau=j+1}^t (1-\eta_\tau)\right) \eta_j C_{j+1}}}.
\end{eqnarray*}

Next, since $1-\eta_j + \frac{(1-\eta_j)^2}{\eta_j} = \frac{1-\eta_j}{\eta_j} \leq \frac{1}{\eta_j}$, we obtain 
\begin{eqnarray*}
    \Exp{\norm{e_{t+1}}} 
    & \leq & \prod_{\tau=0}^t (1-\eta_\tau) \Exp{\norm{{e}_0}} + \frac{L_h}{2}\sum_{j=0}^t \left(\prod_{\tau=j+1}^t (1-\eta_\tau)\right) \frac{1}{\eta_j} \gamma_j^2 \\
    && + \Exp{\norm{\sum_{j=0}^t \left(\prod_{\tau=j+1}^t (1-\eta_\tau)\right) \eta_j C_{j+1}}}.
\end{eqnarray*}

Next, since 
\begin{eqnarray*}
    \Exp{\norm{\sum_{j=0}^t \left(\prod_{\tau=j+1}^t (1-\eta_\tau)\right) \eta_j C_{j+1}}} 
    & \overset{(a)}{\leq} & \sqrt{  \Exp{\sqnorm{\sum_{j=0}^t \left(\prod_{\tau=j+1}^t (1-\eta_\tau)\right) \eta_j C_{j+1}}} } \\
    & \overset{(b)}{\leq} &  \frac{\sigma_g}{\sqrt{n}} \sqrt{ \sum_{j=0}^t \left(\prod_{\tau=j+1}^t (1-\eta_\tau)^2\right) \eta_j^2 },
\end{eqnarray*}
where in $(a)$ we used Jensen's inequality, in $(b)$ we used \Cref{assum:stoc_g_h}, we get
\begin{eqnarray*}
    \Exp{\norm{e_{t+1}}} 
    & \leq & \prod_{\tau=0}^t (1-\eta_\tau) \Exp{\norm{{e}_0}} + \frac{L_h}{2}\sum_{j=0}^t \left(\prod_{\tau=j+1}^t (1-\eta_\tau)\right) \frac{1}{\eta_j} \gamma_j^2 \\
    && + \frac{\sigma_g}{\sqrt{n}} \sqrt{ \sum_{j=0}^t \left(\prod_{\tau=j+1}^t (1-\eta_\tau)\right) \eta_j^2 }.
\end{eqnarray*}

If $\eta_t = \left( \frac{2}{t+2} \right)^{4/7}$ and $\gamma_t = \gamma_0 \left( \frac{2}{t+2} \right)^{5/7}$, then we can prove  that $\eta_0=1$, and that
\begin{eqnarray*}
    \Exp{\norm{e_{t+1}}} 
    \leq \frac{L_h}{2}\sum_{j=0}^t \left(\prod_{\tau=j+1}^t (1-\eta_\tau)\right) \frac{1}{\eta_j} \gamma_j^2 + \frac{\sigma_g}{\sqrt{n}} \sqrt{ \sum_{j=0}^t \left(\prod_{\tau=j+1}^t (1-\eta_\tau)\right) \eta_j^2 }.
\end{eqnarray*}

In conclusion, from~\Cref{lemma:sum_prod}, 
\begin{eqnarray*}
    \Exp{\norm{e_{t+1}}} 
    \leq \frac{L_h}{2}C(\nicefrac{6}{7}, \nicefrac{4}{7}) \left( \frac{\gamma_{t+1}}{\eta_{t+1}}\right)^2  + \frac{\sigma_g}{\sqrt{n}} \sqrt{C(\nicefrac{8}{7},\nicefrac{4}{7}) \eta_{t+1} }.
\end{eqnarray*}

\subsubsection{Bounding Lyapunov function}

We define Lyapunov function as  $V_t = \Delta_{t} + C_{1,t} \cV_t + C_{2,t} \cU_t$ with $C_{1,t} = \frac{2\gamma_t}{1-\sqrt{1-\alpha}}$ and $C_{2,t} = \frac{2\gamma_t \sqrt{1-\alpha}}{1-\sqrt{1-\alpha}}$. Then, 
\begin{eqnarray*}
    \Exp{V_{t+1}} 
    & = & \Exp{\Delta_{t+1} + C_{1,t+1} \cV_t + C_{2,t+1} \cU_{t+1}} \\
    & \leq & \Exp{\Delta_t} - \gamma_t \Exp{\norm{\nabla f(x^t)}} + 2 \gamma_t \Exp{\norm{v^t- \nabla f(x^t)}} + 2\gamma_t \Exp{\cV_t} + \frac{\gamma_t^2L}{2} \\
    && + C_{1,t+1} \Exp{\cV_{t+1}} + C_{2,t+1} \Exp{\cU_{t+1}},
\end{eqnarray*}
where in the last inequality we used \eqref{eq:descent_lemma}.
Next, by the upper-bounds for $\Exp{\cV_{t+1}}$ and  $\Exp{\cU_{t+1}}$, 
\begin{eqnarray*}
    \Exp{ V_{t+1} }
   & \leq & \Exp{\Delta_t} - \gamma_t \Exp{\norm{\nabla f(x^t)}} + 2 \gamma_t \Exp{\norm{v^t- \nabla f(x^t)}} + 2\gamma_t \Exp{\cV_t} + \frac{\gamma_t^2L}{2} \\
    && + C_{1,t+1} \left(\sqrt{1-\alpha}    \Exp{\cV_{t}} + \sqrt{1-\alpha} \eta_t \Exp{\cU_t} + \sqrt{1-\alpha} \bar L \gamma_t  + \sqrt{1-\alpha} \sigma_g \eta_t\right) \\
    && + C_{2,t+1} \left( (1-\eta_t) \Exp{\cU_t} + \eta_t \sigma_g +  \frac{1-\eta_t}{\eta_t}\gamma_t^2  \frac{\bar L_{h}}{2} \right).
\end{eqnarray*}

Next,  since $\gamma_{t+1} \leq \gamma_t$, we can prove that $C_{1,t+1} \leq C_{1,t}$ and that $C_{2,t+1} \leq C_{2,t}$, and also that 
\begin{eqnarray*}
    2\gamma_t + C_{1,t+1}\sqrt{1-\alpha}  & \leq & 2\gamma_t + C_{1,t}\sqrt{1-\alpha} = C_{1,t} \\
    C_{1,t+1} \sqrt{1-\alpha}\eta_t + C_{2,t+1} (1-\eta_t) & \leq& C_{1,t} \sqrt{1-\alpha}\eta_t + C_{2,t} (1-\eta_t)= C_{2,t}
\end{eqnarray*}
Therefore, 
\begin{eqnarray*}
   \Exp{ V_{t+1} } 
   & \leq & \Exp{V_t} - \gamma_t \Exp{\norm{\nabla f(x^t)}} + 2 \gamma_t \Exp{\norm{v^t- \nabla f(x^t)}}  + \frac{\gamma_t^2L}{2} \\
    && + C_{1,t} (\sqrt{1-\alpha} \bar L \gamma_t  + \sqrt{1-\alpha} \sigma_g \eta_t)  + C_{2,t} \left( \eta_t \sigma_g +  \frac{1-\eta_t}{\eta_t}\gamma_t^2  \frac{\bar L_{h}}{2} \right) \\
     & \leq & \Exp{V_t} - \gamma_t \Exp{\norm{\nabla f(x^t)}} + 2 \gamma_t \Exp{\norm{v^t- \nabla f(x^t)}}   \\
    && + \gamma_t^2  \left( \frac{2\sqrt{1-\alpha}}{1-\sqrt{1-\alpha}} \bar L + \frac{L}{2} \right) + \gamma_t^3 \frac{1-\eta_t}{\eta_t} \frac{\sqrt{1-\alpha}}{1-\sqrt{1-\alpha}} \bar L_h  + \eta_t\gamma_t \frac{4\sqrt{1-\alpha}}{1-\sqrt{1-\alpha}} \sigma_g.
\end{eqnarray*}

Next, from the upper-bound of $\Exp{\norm{v^t- \nabla f(x^t)}}$, 
\begin{eqnarray*}
   \Exp{ V_{t+1} } 
   & \leq & \Exp{V_t} - \gamma_t \Exp{\norm{\nabla f(x^t)}} + 2 \gamma_t \left( \frac{L_h}{2}C(\nicefrac{6}{7}, \nicefrac{4}{7}) \left( \frac{\gamma_{t}}{\eta_{t}}\right)^2  + \frac{\sigma_g}{\sqrt{n}} \sqrt{C(\nicefrac{8}{7},\nicefrac{4}{7}) \eta_{t} }  \right)  \\
    && + \gamma_t^2  \left( \frac{2\sqrt{1-\alpha}}{1-\sqrt{1-\alpha}} \bar L + \frac{L}{2} \right) + \gamma_t^3 \frac{1-\eta_t}{\eta_t} \frac{\sqrt{1-\alpha}}{1-\sqrt{1-\alpha}} \bar L_h  + \eta_t\gamma_t \frac{4\sqrt{1-\alpha}}{1-\sqrt{1-\alpha}} \sigma_g \\
    & \leq & \Exp{V_t} - \gamma_t \Exp{\norm{\nabla f(x^t)}} + \frac{\gamma_t^3}{\eta_t^2} \cdot  L_h C(\nicefrac{6}{7},\nicefrac{4}{7}) + \gamma_t \sqrt{\eta_t} \cdot \frac{2 \sigma_g \sqrt{C(\nicefrac{8}{7}, \nicefrac{4}{7})}}{\sqrt{n}}  \\
    && + \gamma_t^2  \left( \frac{2\sqrt{1-\alpha}}{1-\sqrt{1-\alpha}} \bar L + \frac{L}{2} \right) +  \frac{\gamma_t^3}{\eta_t} \frac{\sqrt{1-\alpha}}{1-\sqrt{1-\alpha}} \bar L_h  + \eta_t\gamma_t \frac{4\sqrt{1-\alpha}}{1-\sqrt{1-\alpha}} \sigma_g.
\end{eqnarray*}

\subsubsection{Deriving the convergence rate}
By re-arranging the terms and by the telescopic series, 
\begin{eqnarray*}
    \frac{\sum_{t=0}^{T-1} \gamma_t\Exp{\norm{\nabla f(x^t)}}}{\sum_{t=0}^{T-1} \gamma_t}  
    & \leq & \frac{\Exp{V_0} -\Exp{V_T}}{\sum_{t=0}^{T-1}\gamma_t} +  \left( \frac{2\sqrt{1-\alpha}}{1-\sqrt{1-\alpha}} \bar L + \frac{L}{2} \right) \frac{\sum_{t=0}^{T-1} \gamma_t^2}{ \sum_{t=0}^{T-1} \gamma_t} \\
    && + L_h C(\nicefrac{6}{7},\nicefrac{4}{7}) \frac{\sum_{t=0}^{T-1} \gamma_t^3\eta_t^{-2}}{\sum_{t=0}^{T-1}\gamma_t} + \frac{2\sigma_g \sqrt{C(\nicefrac{8}{7}, \nicefrac{4}{7})}}{\sqrt{n}} \frac{\sum_{t=0}^{T-1} \gamma_t \sqrt{\eta_t}}{ \sum_{t=0}^{T-1} \gamma_t} \\
    &&  + \frac{\sqrt{1-\alpha}}{1-\sqrt{1-\alpha}} \bar L_h \frac{\sum_{t=0}^{T-1} \gamma_t^3 \eta_t^{-2}}{ \sum_{t=0}^{T-1} \gamma_t} +     \frac{4\sqrt{1-\alpha}}{1-\sqrt{1-\alpha}}  \sigma_g \frac{\sum_{t=0}^{T-1} \eta_t\gamma_t}{ \sum_{t=0}^{T-1}\gamma_t}\\
    &\overset{(a)}{\leq}& \frac{\Exp{V_0}}{\sum_{t=0}^{T-1}\gamma_t} + \left( \frac{L}{2} + \frac{\bar L }{\alpha^2}  \right) \frac{\sum_{t=0}^{T-1} \gamma_t^2}{ \sum_{t=0}^{T-1} \gamma_t} +  \frac{2\sigma_g \sqrt{C(\nicefrac{8}{7}, \nicefrac{4}{7})}}{\sqrt{n}} \frac{\sum_{t=0}^{T-1} \gamma_t \sqrt{\eta_t}}{ \sum_{t=0}^{T-1} \gamma_t}\\
    &&+\left( L_h C(\nicefrac{6}{7},\nicefrac{4}{7}) + \frac{\bar L_h}{\alpha^2} \right)\frac{\sum_{t=0}^{T-1} \gamma_t^3\eta_t^{-2}}{\sum_{t=0}^{T-1}\gamma_t}+  \frac{4\sigma_g}{\alpha^2} \frac{\sum_{t=0}^{T-1} \eta_t\gamma_t}{ \sum_{t=0}^{T-1}\gamma_t}\\
    &\overset{(b)}{\leq}& \frac{\Exp{V_0}}{\sum_{t=0}^{T-1}\gamma_t} + \left( \frac{L}{2} + \frac{\bar L }{\alpha^2}  \right) \frac{\sum_{t=0}^{T-1} \gamma_t^2}{ \sum_{t=0}^{T-1} \gamma_t} + \left( L_h C_2 + \frac{\bar L_h}{\alpha^2} \right)\frac{\sum_{t=0}^{T-1} \gamma_t^3\eta_t^{-2}}{\sum_{t=0}^{T-1}\gamma_t}\\
    && + \frac{2C_1\sigma_g }{\sqrt{n}} \frac{\sum_{t=0}^{T-1} \gamma_t \sqrt{\eta_t}}{ \sum_{t=0}^{T-1} \gamma_t} + \frac{4\sigma_g}{\alpha^2} \frac{\sum_{t=0}^{T-1} \eta_t\gamma_t}{ \sum_{t=0}^{T-1}\gamma_t},
\end{eqnarray*}
where in $(a)$ we used $\frac{\sqrt{1-\alpha}}{1-\sqrt{1-\alpha}}$, in $(b)$ we denoted $C_1 \eqdef \sqrt{C(\nicefrac{8}{7}, \nicefrac{4}{7})}$, $C_2 \eqdef C(\nicefrac{6}{7},\nicefrac{4}{7})$.
Since $\eta_t = \left( \frac{2}{t+2} \right)^{4/7}$ and $\gamma_t = \gamma_0 \left( \frac{2}{t+2} \right)^{5/7}$, for $T \geq 1$
\begin{eqnarray*}
    \sum_{t=0}^{T-1} \gamma_t &\geq&  T \gamma_{T-1} = \gamma_0T  \left(\frac{2}{T+1}\right)^{\nicefrac{5}{7}} \geq \gamma_0 T^{\nicefrac{2}{7}};\\
    \sum_{t=0}^{T-1} \gamma_t^3 \eta_t^{-2}  &=& \gamma_0^3\sum_{t=0}^{T-1}  \left( \frac{2}{t+2} \right)^{\nicefrac{15}{7}} \left( \frac{2}{t+2} \right)^{-\nicefrac{8}{7} } = 2\gamma_0^3 \sum_{t=0}^{T-1} \frac{1}{t+2}   = 2\gamma_0^3 \sum^{T}_{t=1}\frac{1}{1+t}\\
    &\leq& 2\gamma^3_0 \int^T_{1}\frac{1}{1+t}dt = 2\gamma^3_0 \left(\log \left(T+1\right) -\log(2)\right) \leq 2\gamma^3_0 \log \left(T+1\right); \\
    \sum_{t=0}^{T-1} \gamma_t \sqrt{\eta_t} &=& \gamma_0 \sum_{t=0}^{T-1} \left( \frac{2}{t+2} \right)^{\nicefrac{5}{7}} \left( \frac{2}{t+2} \right)^{\nicefrac{2}{7} } =   2\gamma_0 \sum_{t=0}^{T-1}\frac{1}{t+2} \leq 2\gamma_0 \log\left(T+1\right);\\
    \sum_{t=0}^{T-1} \gamma_t^2 &=& \gamma_0^2 \sum_{t=0}^{T-1} \left( \frac{2}{t+2} \right)^{\nicefrac{10}{7}}  \leq 2^{\nicefrac{10}{7}}\gamma_0^2 \sum_{t=0}^{T-1} \frac{1}{(t+2)^{\nicefrac{10}{7}}} = 2^{\nicefrac{10}{7}}\gamma_0^2 \sum_{t=1}^{T} \frac{1}{(1+t)^{\nicefrac{10}{7}}}\\
    &\leq& 2^{\nicefrac{10}{7}}\gamma_0^2 \int^T_{1}\frac{1}{(1+t)^{\nicefrac{10}{7}}}dt  = 2^{\nicefrac{10}{7}} \cdot \frac{7}{3}\gamma^2_0 \left(\frac{1}{2^{\nicefrac{3}{7}}} - \frac{1}{(T+1)^{\nicefrac{3}{7}}}\right) \leq 5\gamma^2_0;\\
    \sum_{t=0}^{T-1} \gamma_t \eta_t &=& \gamma_0 \sum_{t=0}^{T-1}  \left( \frac{2}{t+2} \right)^{\nicefrac{5}{7}} \left( \frac{2}{t+2} \right)^{\nicefrac{4}{7} } = 2^{\nicefrac{9}{7}}\gamma_0 \sum^{T-1}_{t=0} \frac{1}{(t+2)^{\nicefrac{9}{7}}} = 2^{\nicefrac{9}{7}}\gamma_0 \sum^{T}_{t=1} \frac{1}{(t+1)^{\nicefrac{9}{7}}}\\
    &\leq& 2^{\nicefrac{9}{7}}\gamma_0\int^T_{1}\frac{1}{(1+t)^{\nicefrac{9}{7}}}dt =  2^{\nicefrac{9}{7}} \cdot\frac{7}{2}\gamma_0\left(\frac{1}{2^{\nicefrac{2}{7}}} - \frac{1}{(T+1)^{\nicefrac{2}{7}}}\right) \leq 7\gamma_0.
\end{eqnarray*}

Therefore, denoting $\tilde{x}^T$ as a point randomly chosen from $\{x^0,x^1,\ldots,x^{T-1}\}$ with probability $\nicefrac{\gamma_t}{\sum_{t=0}^{T-1}\gamma_t}$ for $t=0,1,\ldots,T-1$, we obtain 
\begin{eqnarray*}
    \Exp{\norm{\nabla f(\tilde{x}^T)}}  
    & \leq & \frac{\Exp{V_0}}{\gamma_0 T^{\nicefrac{2}{7}}} + \left( \frac{L}{2} + \frac{\bar L }{\alpha^2}  \right) \frac{ 5\gamma^2_0}{\gamma_0 T^{\nicefrac{2}{7}}} + \left( L_h C_2 + \frac{\bar L_h}{\alpha^2} \right)\frac{2\gamma^3_0 \log \left(T+1\right)}{\gamma_0 T^{\nicefrac{2}{7}}}\\
    && + \frac{2C_1\sigma_g }{\sqrt{n}} \frac{2\gamma_0 \log \left(T+1\right)}{ \gamma_0 T^{\nicefrac{2}{7}}} + \frac{4\sigma_g}{\alpha^2} \frac{7\gamma_0}{ \gamma_0 T^{\nicefrac{2}{7}}}\\
    &=&\widetilde{\cO}\left(\frac{\nicefrac{V_0}{\gamma_0} + \sigma_g\left(\nicefrac{1}{\sqrt{n}} + \nicefrac{1}{\alpha^2}\right) + \gamma_0\left(L+\nicefrac{\bar L}{\alpha^2}\right) + \gamma_0^2\left(L_h +\nicefrac{\bar L_h}{\alpha^2}\right)}{T^{\nicefrac{2}{7}}}\right).
\end{eqnarray*}

\newpage 
\section{EF21-RHM}
\label{sec:appendix_EF21_RHM}
In \algname{EF21-RHM}, we update the iterates $\{x^t\}$ according to: 
\begin{eqnarray}
    x^{t+1} = x^t - \gamma_t \frac{g^t}{\norm{g^t}}, \notag
\end{eqnarray}
where $g^t = \frac{1}{n}\sum_{i=1}^n g_i^t$ and $v^t = \frac{1}{n}\sum_{i=1}^n v_i^t$ with $g_i^t$ and $v_i^t$ being governed by 
\begin{eqnarray}
    g_i^{t+1} & = & g_i^t + \mathcal{C}_i^{t+1}(v_i^{t+1} - g_i^t), \quad \text{and} \notag\\
    v_i^{t+1} & = & (1-\eta_t)(v_i^t + \nabla^2 f_i(\hat x^{t+1};\xi_i^{t+1}) (x^{t+1} - x^t) ) + \eta_t \nabla f_i(x^{t+1};\xi_i^{t+1}).\notag
\end{eqnarray}
Here, $\hat x^{t+1} = q_t x^{t+1} + (1-q_t) x^t$, where $q_t \sim \cU(0,1)$.

\subsection{Convergence Proof}

We prove the result in the following steps.

\subsubsection{Deriving the descent inequality} From \Cref{lem:descent_lemma}, we obtain \eqref{eq:descent_lemma}:
\begin{equation*}
    \Delta_{t+1} + \gamma_t \|\nabla f(x^t)\| \le \Delta_t  + 2 \gamma_t \norm{v^t- \nabla f(x^t)} +\frac{ 2\gamma_t}{n}\sum_{i=1}^n\norm{g_i^t-v_i^t} + \frac{\gamma_t^2L}{2}.
\end{equation*}

\subsubsection{Error bound I}

Define $\cV_t \eqdef \frac{1}{n}\sum_{i=1}^n \norm{g_i^t-v_i^t}$. From~\Cref{lemma:boundg_and_v}, 
\begin{eqnarray*}
   \Exp{\cV_{t+1}} \leq \sqrt{1-\alpha}    \Exp{\cV_{t}} + \frac{\sqrt{1-\alpha}}{n} \sum_{i=1}^n \Exp{\norm{v_i^{t+1}-v_i^t}}.
\end{eqnarray*}

To complete the upper-bound for $\Exp{\cV_{t+1}}$, we must upper-bound $\Exp{\norm{v_i^{t+1}-v_i^t}}$. 
From the definition of $v_i^{t+1}$, 
\begin{eqnarray*}
v_i^{t+1} - v_i^t
& = & (1-\eta_t)(v_i^t + \nabla^2 f_i(\hat{x}^{t+1}, \hat{\xi}_i^{t+1})(x^{t+1}-x^t))  + \eta_t \nabla f_i(x^{t+1}, \xi_i^{t+1}) - v_i^t \\
& = &  \eta_t (\nabla f_i(x^t) - v_i^t)  + \eta_t (\nabla f_i(x^{t+1}) - \nabla f_i(x^t)) \\
&& + (1-\eta_t)(\nabla^2 f_i(\hat{x}^{t+1}, \hat{\xi}_i^{t+1})(x^{t+1}-x^t) - \nabla^2 f_i(\hat{x}^{t+1})(x^{t+1}-x^t)) \\
&& + (1-\eta_t)(\nabla^2 f_i(\hat{x}^{t+1})(x^{t+1}-x^t))  + \eta_t (\nabla f_i(x^{t+1}, \xi_i^{t+1}) - \nabla f_i(x^{t+1})).
\end{eqnarray*}

Next, from the definition of the Euclidean norm,  by the triangle inequality, and by the fact that $f_i$ is $L_i$-smooth,
\begin{eqnarray*}
\|v_i^{t+1} - v_i^t\| 
&\le& \eta_t \|v_i^t - \nabla f_i(x^t)\| + \eta_t \|\nabla f_i(x^{t+1}) - \nabla f_i(x^t)\|  \\
&& + (1-\eta_t) \|\nabla^2 f_i(\hat{x}^{t+1}, \hat{\xi}_i^{t+1}) - \nabla^2 f_i(\hat{x}^{t+1})\|_{\text{op}} \|x^{t+1}-x^t\|  \\
&& + (1-\eta_t) \|\nabla^2 f_i(\hat{x}^{t+1})\|_{\text{op}} \|x^{t+1}-x^t\|  \\
&& + \eta_t \|\nabla f_i(x^{t+1}, \xi_i^{t+1}) - \nabla f_i(x^{t+1})\| \\
&\leq& \eta_t \|v_i^t - \nabla f_i(x^t)\| + \eta_t \gamma_t L_i  \\
&& + (1-\eta_t) \gamma_t \|\nabla^2 f_i(\hat{x}^{t+1}, \hat{\xi}_i^{t+1}) - \nabla^2 f_i(\hat{x}^{t+1})\|_{\text{op}}  \\
&& + (1-\eta_t) \gamma_t \|\nabla^2 f_i(\hat{x}^{t+1})\|_{\text{op}}   \\
&& + \eta_t \|\nabla f_i(x^{t+1}, \xi_i^{t+1}) - \nabla f_i(x^{t+1})\|,
\end{eqnarray*}
where in the last inequity we used the update rule for $x^{t+1}$.
By taking the expectation, and by the fact that $f_i$ is $L_i$-smooth, i.e. $\norm{\nabla^2 f_i(x)}_{\text{op}} \leq L_i$,
\begin{eqnarray*}
\Exp{\|v_i^{t+1} - v_i^t\|} 
&\le & \eta_t \|v_i^t - \nabla f_i(x^t)\| + \eta_t L_i \gamma_t + (1-\eta_t)L_i \gamma_t + \\
&& + (1-\eta_t)\gamma_t \sigma_h + \eta_t \sigma_g.
\end{eqnarray*}

Therefore, we obtain
\begin{eqnarray*}
    \Exp{\cV_{t+1}} &\leq& \sqrt{1-\alpha} \cV_t + \sqrt{1-\alpha} \eta_t \cU_t  + \sqrt{1-\alpha} \gamma_t \bar L + \sqrt{1-\alpha} (1-\eta_t) \gamma_t \sigma_h + \sqrt{1-\alpha} \eta_t \sigma_g.
\end{eqnarray*}

\subsubsection{Error bound II} 
Define $\cU_t \eqdef \frac{1}{n}\sum_{i=1}^n \norm{v_i^t - \nabla f_i(x^t)}$. Then, from the definition of  $v_i^{t+1}$,
    \begin{eqnarray*}
         \|v_i^{t+1} - \nabla f_i(x^{t+1})\| 
        &=& \|(1-\eta_t)(v_i^t - \nabla f_i(x^t)) + (1-\eta_t)\hat S_{i,t+1} + \eta_t e_{i,t+1}\| \\
        &\le& (1-\eta_t)\|v_i^t - \nabla f_i(x^t)\| + (1-\eta_t)\|\hat{S}_{i,t+1}\| + \eta_t \|e_{i,t+1}\|,
    \end{eqnarray*}
    where 
    \begin{eqnarray*}
        \hat{S}_{i,t+1} &=& \nabla^2 f_i(\hat{x}^{t+1}, \hat{\xi}_i^{t+1})(x^{t+1}-x^t) - \nabla f_i(x^{t+1}) + \nabla f_i(x^t), \quad \text{and}\\
        e_{i,t+1} &=& \nabla f_i(x^{t+1}, \xi_i^{t+1}) - \nabla f_i(x^{t+1}).
    \end{eqnarray*}

Therefore, from the definition of $\mathcal{U}_{t+1}$, 
\begin{eqnarray*}
 \Exp{\mathcal{U}_{t+1}} 
 & \leq & (1-\eta_t)  \mathbb{E}_{t}[\mathcal{U}_{t}] + \frac{1-\eta_t}{n} \sum_{i=1}^n \Exp{\norm{\hat S_{i,t+1}}}  + \eta_t \frac{1}{n}\sum_{i=1}^n \Exp{\norm{e_{i,t+1}}} \\
 &\leq&  (1-\eta_t)  \mathbb{E}_{t}[\mathcal{U}_{t}] + \frac{1-\eta_t}{n} \sum_{i=1}^n \sqrt{\Exp{\sqnorm{\hat S_{i,t+1}}}}  + \eta_t \frac{1}{n}\sum_{i=1}^n \sqrt{\Exp{\sqnorm{e_{i,t+1}}}},
\end{eqnarray*}
where in the last inequality we applied Jensen's inequality. 
Next, from~\Cref{assum:stoc_g_h}, 
\begin{eqnarray*}
 \Exp{\mathcal{U}_{t+1}} 
\leq  (1-\eta_t)  \mathbb{E}_{t}[\mathcal{U}_{t}] + \frac{1-\eta_t}{n} \sum_{i=1}^n \sqrt{\Exp{\sqnorm{\hat S_{i,t+1}}}}  + \eta_t \sigma_g.
\end{eqnarray*}

Next, we bound $\Exp{\sqnorm{\hat S_{i,t+1}}}$. By~\Cref{assum:stoc_g_h}, 
\begin{eqnarray*}
    \Exp{\norm{\hat{S}_{i,t+1}}^2}
    & = & \Exp{\left\| (\nabla^2 f_i(\hat{x}^{t+1}, \hat{\xi}_i^{t+1}) - \nabla^2 f_i(\hat{x}^{t+1})) (x^{t+1}-x^t) \right\|^2}  \\
    && + \Exp{\left\| \nabla^2 f_i(\hat{x}^{t+1}) (x^{t+1}-x^t) - (\nabla f_i(x^{t+1}) - \nabla f(x^t)) \right\|^2} \\
    & \leq & \Exp{\left\| \nabla^2 f_i(\hat{x}^{t+1}, \hat{\xi}_i^{t+1}) - \nabla^2 f_i(\hat{x}^{t+1}) \right\|_{\rm op}^2 \left\| x^{t+1}-x^t \right\|^2}  \\
    && + 2\Exp{\left\| \nabla^2 f_i(\hat{x}^{t+1}) \right\|^2_{\rm op} \left\|x^{t+1}-x^t \right\|^2} + 2 \Exp{ \left\| \nabla f_i(x^{t+1}) - \nabla f_i(x^t)\right\|^2 } \\
    & \leq &  \sigma_g^2 \Exp{ \left\| x^{t+1}-x^t \right\|^2}  + 2\Exp{\left\| \nabla^2 f_i(\hat{x}^{t+1}) \right\|^2_{\rm op} \left\|x^{t+1}-x^t \right\|^2} \\
    && + 2 \Exp{ \left\| \nabla f_i(x^{t+1}) - \nabla f_i(x^t)\right\|^2 }.
\end{eqnarray*}
By the fact $f_i$ is $L_i$-smooth, and by the definition of $x^{t+1}$, 
    \begin{eqnarray*}
    \Exp{\norm{\hat{S}_{i,t+1}}^2} \leq  \gamma_t^2 \left( \sigma_{h}^2 + 4 L^2_i \right) .
    \end{eqnarray*}
Next, plugging the upper-bound of $\ExpSub{t}{\norm{\hat{S}_{i,t+1}}^2}$ into the upper-bound of $  \mathbb{E}_{t}[\mathcal{U}_{t+1}]$, we obtain
\begin{eqnarray*}
        \mathbb{E}_{t}[\mathcal{U}_{t+1}]
     & \leq & (1-\eta_t)  \mathcal{U}_{t} + \frac{1-\eta_t}{n} \sum_{i=1}^n \sqrt{ \gamma_t^2 \left( \sigma_{h}^2 + 4 L^2_i \right)  } + \eta_t \sigma_g \\
     & \leq & (1-\eta_t)  \mathcal{U}_{t} + (1-\eta_t)\gamma_t (\sigma_h + 2\bar L) + \eta_t \sigma_g,
    \end{eqnarray*}
where in the last inequality we used $\sqrt{a+b} \leq \sqrt{a} + \sqrt{b}$ for $a,b  \geq 0$.

\subsubsection{Error bound III} 

Next, we bound $\norm{v^{t+1} - \nabla f(x^{t+1})}$: 
\begin{eqnarray*}
\hat{e}_{t+1} 
&=& (1-\eta_t)\hat{e}_t + (1-\eta_t)\hat{S}_{t+1} + \eta_t e_{t+1} \\
&=& \prod_{\tau=0}^t (1-\eta_\tau) \hat{e}_0 + \sum_{j=0}^t \left(\prod_{\tau=j+1}^t (1-\eta_\tau)\right) \hat{S}_{j+1} + \sum_{j=0}^t \left(\prod_{\tau=j}^t (1-\eta_\tau)\right) \eta_j e_{j+1},
\end{eqnarray*}
where $\hat S_{t+1} = \frac{1}{n}\sum_{i=1}^n \left(\nabla^2 f_i(\hat{x}^{t+1}, \hat{\xi}_i^{t+1})(x^{t+1}-x^t) - (\nabla f_i(x^{t+1}) - \nabla f_i(x^t))\right)$, $e_{t+1} = \frac{1}{n}\sum_{i=1}^n e_{i,t+1}$, and $e_{i,t+1} = \nabla f_i(x^{t+1}, \xi_i^{t+1}) - \nabla f_i(x^{t+1})$. 

If $\eta_0=1$, then by taking the Euclidean norm and the expectation, 
\begin{eqnarray*}
\Exp{\norm{\hat{e}_{t+1}}} 
& \overset{(a)}{\leq} &  \Exp{\norm{ \sum_{j=0}^t \left(\prod_{\tau=j}^t (1-\eta_\tau)\right)  \hat{S}_{j+1}}} + \Exp{ \norm{ \sum_{j=0}^t \left(\prod_{\tau=j}^t (1-\eta_\tau)\right) \eta_j e_{j+1}} }\\
& \overset{(b)}{\leq}&  \left(\mathbb{E}\left\|\sum_{j=0}^t \prod_{\tau=j}^t (1-\eta_\tau) \hat{S}_{j+1}\right\|^2\right)^{1/2} + \left(\mathbb{E}\left\|\sum_{j=0}^t \prod_{\tau=j+1}^t (1-\eta_\tau) \eta_j e_{j+1}\right\|^2\right)^{1/2},
\end{eqnarray*}
where in $(a)$ we used triangle inequality, in $(b)$ we applied Jensen's inequality.
From~\Cref{assum:stoc_g_h}, we can prove that $\Exp{e_l} = 0$, $\Exp{\sqnorm{e_l}} = \sigma_g^2/n$, and $\Exp{\inp{e_l}{e_i}}=0$ for $l\neq j$. Thus, 
\begin{eqnarray*}
\Exp{\norm{\hat{e}_{t+1}}} 
\leq  \left(\mathbb{E}\left\|\sum_{j=0}^t \prod_{\tau=j}^t (1-\eta_\tau) \hat{S}_{j+1}\right\|^2\right)^{1/2} + \frac{\sigma_g}{\sqrt{n}} \left( \sum_{j=0}^{t} \prod_{\tau=j+1}^{t} (1-\eta_\tau)^2 \eta_j^2 \right)^{1/2}.
\end{eqnarray*}

Next, from~\Cref{assum:stoc_g_h}, we can show that $\Exp{\inp{\hat S_l}{\hat S_j}}=0$ for $l\neq j$, and that 
\begin{eqnarray*}
\Exp{\norm{\hat{e}_{t+1}}} 
\leq  \left( \sum_{j=0}^t \prod_{\tau=j}^t (1-\eta_\tau)^2 \mathbb{E}\left\| \hat{S}_{j+1}\right\|^2\right)^{1/2} + \frac{\sigma_g}{\sqrt{n}} \left( \sum_{j=0}^{t} \prod_{\tau=j+1}^{t} (1-\eta_\tau)^2 \eta_j^2 \right)^{1/2}.
\end{eqnarray*}
Since 
\begin{eqnarray*}
\mathbb{E}[\|\hat{S}_{j+1}\|^2] &=& \mathbb{E}\left[\left\|\frac{1}{n}\sum_{i=1}^n \left(\nabla^2 f_i(\hat{x}^{j+1}, \hat{\xi}_i^{j+1})(x^{j+1}-x^j) - (\nabla f_i(x^{j+1}) - \nabla f_i(x^j))\right)\right\|^2\right] \\
&\overset{(a)}{\leq}& \frac{1}{n^2}\sum_{i=1}^n \left( \Exp{\left\| (\nabla^2 f_i(\hat{x}^{j+1}, \hat{\xi}_i^{j+1}) - \nabla^2 f_i(\hat{x}^{j+1})) (x^{j+1}-x^j) \right\|^2}\right)  \\
&&\quad + \Exp{\left\| \nabla^2 f(\hat{x}^{j+1}) (x^{j+1}-x^j) - (\nabla f(x^{j+1}) - \nabla f(x^j)) \right\|^2}  \\
&\le& \frac{1}{n^2}\sum_{i=1}^n \left( \mathbb{E}\left\| (\nabla^2 f_i(\hat{x}^{j+1}, \hat{\xi}_i^{j+1}) - \nabla^2 f_i(x^{j+1})) (x^{j+1}-x^j) \right\|^2 \right) \\
&&\quad +  2\mathbb{E}\left\| \nabla^2 f(\hat{x}^{j+1})\right\|^2_{\text{op}} \left\|x^{j+1}-x^j \right\|^2 + 2\mathbb{E}\left\| \nabla f(x^{j+1}) - \nabla f(x^j) \right\|^2  \\
&\overset{(b)}{\leq}& \gamma_j^2 \left(\frac{1}{n^2}\sum_{i=1}^n \sigma_{h}^2 + 4L^2 \right) \\
&\le&  4\left(\frac{\sigma_h^2}{n} + L^2\right) \gamma_j^2 ,
\end{eqnarray*}
where in $(a)$ we applied variance-bias decomposition, in $(b)$ we used smoothness of $f$ and \Cref{assum:stoc_g_h}, 
we obtain:
\begin{eqnarray*}
\Exp{\norm{\hat{e}_{t+1}}}  
 \leq    4 \left(\sum_{j=1}^t \left(\prod_{\tau=j}^t (1-\eta_\tau)^2\right) \gamma_j^2\right)^{\nicefrac{1}{2}} \left(\frac{\sigma_h}{\sqrt{n}} + L\right)   + \frac{\sigma_g}{\sqrt{n}} \left(\sum_{j=0}^t \left(\prod_{\tau=j+1}^t (1-\eta_\tau)^2\right) \eta_j^2\right)^{\nicefrac{1}{2}}.
\end{eqnarray*}

If $\eta_t = \left( \frac{2}{t+2} \right)^{2/3}$ and $\gamma_t = \gamma_0 \left( \frac{2}{t+2} \right)^{2/3}$, then we can prove  that $\eta_0=1$, and from~\Cref{lemma:sum_prod} that
\begin{eqnarray*}
\Exp{\norm{\hat{e}_{t+1}}}  
& \leq &   4  \left(\frac{\sigma_h}{\sqrt{n}} + L\right)  \sqrt{C(\nicefrac{4}{3},\nicefrac{2}{3}) \frac{\gamma_{t+1}^2}{\eta_{t+1}}}  + \frac{\sigma_g}{\sqrt{n}} \sqrt{C(\nicefrac{4}{3},\nicefrac{2}{3}) \eta_{t+1}}.
\end{eqnarray*}

\subsubsection{Bounding Lyapunov function}

Define $V_t = \Delta_{t} + C_{1,t} \cV_t + C_{2,t} \cU_t$ with $C_{1,t} = \frac{2\gamma_t}{1-\sqrt{1-\alpha}}$ and $C_{2,t} = \frac{2\gamma_t \sqrt{1-\alpha}}{1-\sqrt{1-\alpha}}$. Then, 
\begin{eqnarray*}
    \Exp{V_{t+1}} 
    & = & \Exp{\Delta_{t+1} + C_{1,t+1} \cV_t + C_{2,t+1} \cU_{t+1}} \\
    &\leq& \Exp{\Delta_t} - \gamma_t \Exp{\norm{\nabla f(x^t)}} + 2 \gamma_t \Exp{\norm{v^t- \nabla f(x^t)}} + 2\gamma_t \Exp{\cV_t} + \frac{\gamma_t^2L}{2} \\
    && + C_{1,t+1} \Exp{\cV_{t+1}} + C_{2,t+1} \Exp{\cU_{t+1}},
\end{eqnarray*}
where in the last inequality we used \eqref{eq:descent_lemma}. 
Next, by the upper-bounds for $\Exp{\cV_{t+1}}$ and  $\Exp{\cU_{t+1}}$, 
\begin{eqnarray*}
    \Exp{ V_{t+1} }
   & \leq & \Exp{\Delta_t} - \gamma_t \Exp{\norm{\nabla f(x^t)}} + 2 \gamma_t \Exp{\norm{v^t- \nabla f(x^t)}} + 2\gamma_t \Exp{\cV_t} + \frac{\gamma_t^2L}{2} \\
    && + C_{1,t+1} (\sqrt{1-\alpha} \Exp{\cV_t} + \sqrt{1-\alpha} \eta_t \Exp{\cU_t} ) \\
    && + C_{1,t+1} (\sqrt{1-\alpha} \gamma_t \bar L + \sqrt{1-\alpha} (1-\eta_t) \gamma_t \sigma_h + \sqrt{1-\alpha} \eta_t \sigma_g) \\
    && + C_{2,t+1} ( (1-\eta_t)  \Exp{\mathcal{U}_{t}} + (1-\eta_t)\gamma_t (\sigma_h + 2\bar L) + \eta_t \sigma_g ).
\end{eqnarray*}

Next,  since $\gamma_{t+1} \leq \gamma_t$, we can prove that $C_{1,t+1} \leq C_{1,t}$ and that $C_{2,t+1} \leq C_{2,t}$, and also that 
\begin{eqnarray*}
    2\gamma_t + C_{1,t+1}\sqrt{1-\alpha}  & \leq & 2\gamma_t + C_{1,t}\sqrt{1-\alpha} = C_{1,t} \\
    C_{1,t+1} \sqrt{1-\alpha}\eta_t + C_{2,t+1} (1-\eta_t) & \leq& C_{1,t} \sqrt{1-\alpha}\eta_t + C_{2,t} (1-\eta_t)= C_{2,t}
\end{eqnarray*}
Therefore, 
\begin{eqnarray*}
   \Exp{ V_{t+1} } 
   & \leq & \Exp{V_t} - \gamma_t \Exp{\norm{\nabla f(x^t)}} + 2 \gamma_t \Exp{\norm{v^t- \nabla f(x^t)}}  + \frac{\gamma_t^2L}{2} \\
    && + C_{1,t} (\sqrt{1-\alpha} \gamma_t \bar L + \sqrt{1-\alpha} (1-\eta_t) \gamma_t \sigma_h + \sqrt{1-\alpha} \eta_t \sigma_g) \\
    && + C_{2,t} (  (1-\eta_t)\gamma_t (\sigma_h + 2\bar L) + \eta_t \sigma_g ) \\
     & \leq & \Exp{V_t} - \gamma_t \Exp{\norm{\nabla f(x^t)}} + 2 \gamma_t \Exp{\norm{v^t- \nabla f(x^t)}}   \\
    && + \gamma_t^2  \left( \frac{2\sqrt{1-\alpha}}{1-\sqrt{1-\alpha}} (2\sigma_h + 3 \bar L))  + \frac{L}{2}\right) + \eta_t\gamma_t \frac{4\sqrt{1-\alpha}}{1-\sqrt{1-\alpha}} \sigma_g.
\end{eqnarray*}

Next, from the upper-bound of $\Exp{\norm{v^t- \nabla f(x^t)}}$, 
\begin{eqnarray*}
   \Exp{ V_{t+1} } 
   & \leq & \Exp{V_t} - \gamma_t \Exp{\norm{\nabla f(x^t)}} + 2 \gamma_t \left(   4  \left(\frac{\sigma_h}{\sqrt{n}} + L\right)  \sqrt{C(\nicefrac{4}{3},\nicefrac{2}{3}) \frac{\gamma_{t}^2}{\eta_{t}}}  + \frac{\sigma_g}{\sqrt{n}} \sqrt{C(\nicefrac{4}{3},\nicefrac{2}{3}) \eta_{t}}  \right)  \\
    && + \gamma_t^2  \left( \frac{2\sqrt{1-\alpha}}{1-\sqrt{1-\alpha}} (2\sigma_h + 3 \bar L))  + \frac{L}{2}\right) + \eta_t\gamma_t \frac{4\sqrt{1-\alpha}}{1-\sqrt{1-\alpha}} \sigma_g \\
    & \leq & \Exp{V_t} - \gamma_t \Exp{\norm{\nabla f(x^t)}} + \frac{\gamma_t^2}{\sqrt{\eta_t}} \cdot  8 (\sigma_h/\sqrt{n} + L) \sqrt{  C(\nicefrac{4}{3},\nicefrac{2}{3}) } +  \gamma_t\sqrt{\eta_t} \cdot \frac{2\sigma_g \sqrt{C(\nicefrac{4}{3}, \nicefrac{2}{3})}}{\sqrt{n}}  \\
    && + \gamma_t^2  \left( \frac{2\sqrt{1-\alpha}}{1-\sqrt{1-\alpha}} (2\sigma_h + 3 \bar L))  + \frac{L}{2}\right) + \eta_t\gamma_t \frac{4\sqrt{1-\alpha}}{1-\sqrt{1-\alpha}} \sigma_g.
\end{eqnarray*}

\subsubsection{Deriving the convergence rate}
By re-arranging the terms and by the telescopic series, 
\begin{eqnarray*}
    \frac{\sum_{t=0}^{T-1} \gamma_t \Exp{\norm{\nabla f(x^t)}}}{\sum_{t=0}^{T-1} \gamma_t}  
    & \leq & \frac{\Exp{V_0} -\Exp{V_T}}{\sum_{t=0}^{T-1}\gamma_t}   + 8 \left(\frac{\sigma_h}{\sqrt{n}} + L\right) \sqrt{  C(\nicefrac{4}{3},\nicefrac{2}{3}) } \frac{\sum_{t=0}^{T-1} \nicefrac{\gamma_t^2}{\sqrt{\eta_t}}}{\sum_{t=0}^{T-1}\gamma_t} + \\
    && + \frac{2\sigma_g \sqrt{C(\nicefrac{4}{3}, \nicefrac{2}{3})}}{\sqrt{n}} \frac{\sum_{t=0}^{T-1} \gamma_t \sqrt{\eta_t}}{ \sum_{t=0}^{T-1} \gamma_t}  +     \frac{4\sqrt{1-\alpha}}{1-\sqrt{1-\alpha}}  \sigma_g \frac{\sum_{t=0}^{T-1} \eta_t\gamma_t}{ \sum_{t=0}^{T-1}\gamma_t}\\
    && +  \left( \frac{2\sqrt{1-\alpha}}{1-\sqrt{1-\alpha}} (2\sigma_h + 3 \bar L))  + \frac{L}{2} \right) \frac{\sum_{t=0}^{T-1} \gamma_t^2}{ \sum_{t=0}^{T-1} \gamma_t}  \\
    &\overset{(a)}{\leq}& \frac{\Exp{V_0}}{\sum_{t=0}^{T-1}\gamma_t} + \frac{2\sigma_g \sqrt{C(\nicefrac{4}{3}, \nicefrac{2}{3})}}{\sqrt{n}} \frac{\sum_{t=0}^{T-1} \gamma_t \sqrt{\eta_t}}{ \sum_{t=0}^{T-1} \gamma_t}   +   \frac{4\sigma_g}{\alpha^2} \frac{\sum_{t=0}^{T-1} \eta_t\gamma_t}{ \sum_{t=0}^{T-1}\gamma_t}\\
    && + 8 \left(\frac{\sigma_h}{\sqrt{n}} + L\right) \sqrt{  C(\nicefrac{4}{3},\nicefrac{2}{3}) } \frac{\sum_{t=0}^{T-1} \nicefrac{\gamma_t^2}{\sqrt{\eta_t}}}{\sum_{t=0}^{T-1}\gamma_t} \\
    &&+  \left( \frac{L}{2}+  \frac{2(2\sigma_h + 3 \bar L)}{\alpha^2}   \right) \frac{\sum_{t=0}^{T-1} \gamma_t^2}{ \sum_{t=0}^{T-1} \gamma_t}\\
    &\overset{(b)}{\leq}& \frac{\Exp{V_0}}{\sum_{t=0}^{T-1}\gamma_t} + \frac{2 C_1 \sigma_g}{\sqrt{n}} \frac{\sum_{t=0}^{T-1} \gamma_t \sqrt{\eta_t}}{ \sum_{t=0}^{T-1} \gamma_t}   +   \frac{4\sigma_g}{\alpha^2} \frac{\sum_{t=0}^{T-1} \eta_t\gamma_t}{ \sum_{t=0}^{T-1}\gamma_t}\\
    && + 8 C_1 \left(\frac{\sigma_h}{\sqrt{n}} + L\right) \frac{\sum_{t=0}^{T-1} \gamma_t^2{\eta_t}^{-\nicefrac{1}{2}}}{\sum_{t=0}^{T-1}\gamma_t}\\
    &&+  \left( \frac{L}{2}+  \frac{4\sigma_h + 6 \bar L}{\alpha^2}   \right) \frac{\sum_{t=0}^{T-1} \gamma_t^2}{ \sum_{t=0}^{T-1} \gamma_t},
\end{eqnarray*}
where in $(a)$ we used $\frac{\sqrt{1-\alpha}}{1-\sqrt{1-\alpha}} \leq \frac{1}{\alpha^2}$ for any $\alpha \in (0,1]$, in we denote $C_1 = \sqrt{C(\nicefrac{4}{3}, \nicefrac{2}{3})}$.
Since $\eta_t = \left( \frac{2}{t+2} \right)^{\nicefrac{2}{3}}$ and $\gamma_t = \gamma_0 \left( \frac{2}{t+2} \right)^{\nicefrac{2}{3}}$, 
we obtain

\begin{eqnarray*}
    \sum_{t=0}^{T-1} \gamma_t &\geq&  T \gamma_{T-1} = \gamma_0T  \left(\frac{2}{T+1}\right)^{\nicefrac{2}{3}} \geq \gamma_0 T^{\nicefrac{1}{3}};\\
    \sum_{t=0}^{T-1} \gamma_t^2 \eta_t^{-\nicefrac{1}{2}}  &=& \gamma_0^2\sum_{t=0}^{T-1}  \left( \frac{2}{t+2} \right)^{\nicefrac{4}{3}} \left( \frac{2}{t+2} \right)^{-\nicefrac{1}{3} } = 2\gamma_0^2 \sum_{t=0}^{T-1} \frac{1}{t+2}   = 2\gamma_0^2 \sum^{T}_{t=1}\frac{1}{1+t}\\
    &\leq& 2\gamma^2_0 \int^T_{1}\frac{1}{1+t}dt = 2\gamma^2_0 \left(\log \left(T+1\right) -\log(2)\right) \leq 2\gamma^2_0 \log \left(T+1\right); \\
    \sum_{t=0}^{T-1} \gamma_t \sqrt{\eta_t} &=& \gamma_0 \sum_{t=0}^{T-1} \left( \frac{2}{t+2} \right)^{\nicefrac{2}{3}} \left( \frac{2}{t+2} \right)^{\nicefrac{1}{3} } =   2\gamma_0 \sum_{t=0}^{T-1}\frac{1}{t+2} \leq 2\gamma_0 \log\left(T+1\right);\\
    \sum_{t=0}^{T-1} \gamma_t^2 &=& \gamma_0^2 \sum_{t=0}^{T-1} \left( \frac{2}{t+2} \right)^{\nicefrac{4}{3}}  = 2^{\nicefrac{4}{3}}\gamma_0^2 \sum_{t=0}^{T-1} \frac{1}{(t+2)^{\nicefrac{4}{3}}} = 2^{\nicefrac{3}{2}}\gamma_0^2 \sum_{t=1}^{T} \frac{1}{(1+t)^{\nicefrac{4}{3}}}\\
    &\leq& 2^{\nicefrac{4}{3}}\gamma_0^2 \int^T_{1}\frac{1}{(1+t)^{\nicefrac{4}{3}}}dt  = 2^{\nicefrac{4}{3}}\cdot 3\gamma^2_0 \left(\frac{1}{\sqrt[3]{2}} - \frac{1}{\sqrt[3]{T+1}}\right) \leq 6\gamma^2_0;\\
    \sum_{t=0}^{T-1} \gamma_t \eta_t &=& \frac{1}{\gamma_0}\sum_{t=0}^{T-1} \gamma_t^2 \leq 6\gamma_0.
\end{eqnarray*}

Therefore, denoting $\tilde{x}^T$ as a point randomly chosen from $\{x^0,x^1,\ldots,x^{T-1}\}$ with probability $\nicefrac{\gamma_t}{\sum_{t=0}^{T-1}\gamma_t}$ for $t=0,1,\ldots,T-1$, we obtain 
\begin{eqnarray*}
    \Exp{\norm{\nabla f(\tilde{x}^T)}} 
    & \leq & \frac{\Exp{V_0}}{\gamma_0 T^{\nicefrac{1}{3}}} + \frac{2 C_1 \sigma_g}{\sqrt{n}} \frac{2\gamma_0 \log \left(T+1\right)}{ \gamma_0 T^{\nicefrac{1}{3}}}   +   \frac{4\sigma_g}{\alpha^2} \frac{6\gamma_0}{ \gamma_0 T^{\nicefrac{1}{3}}}\\
    && + 8 C_1 \left(\frac{\sigma_h}{\sqrt{n}} + L\right) \frac{2\gamma^2_0 \log \left(T+1\right)}{\gamma_0 T^{\nicefrac{1}{3}}} 
    +  \left( \frac{L}{2}+  \frac{4\sigma_h + 6 \bar L}{\alpha^2}   \right) \frac{6\gamma_0^2}{ \gamma_0 T^{\nicefrac{1}{3}}}\\
    &=& \frac{\Exp{V_0}}{\gamma_0 T^{\nicefrac{1}{3}}} + \frac{4 C_1 \sigma_g}{\sqrt{n}} \frac{ \log \left(T+1\right)}{  T^{\nicefrac{1}{3}}}   +   \frac{24\sigma_g}{\alpha^2 T^{\nicefrac{1}{3}}}\\
    && + 16C_1 \gamma_0\left(\frac{\sigma_h}{\sqrt{n}} + L\right) \frac{ \log \left(T+1\right)}{ T^{\nicefrac{1}{3}}} 
    + 3\gamma_0 \left( L+  \frac{8\sigma_h}{\alpha^2} + \frac{12 \bar L }{\alpha^2}  \right) \frac{1}{ T^{\nicefrac{1}{3}}}\\
    &=& \widetilde{\cO}\left(\frac{\nicefrac{\Exp{V_0}}{\gamma_0}+ \sigma_g\left(\nicefrac{1}{\sqrt{n}}+\nicefrac{1}{\alpha^2}\right) + \gamma_0\sigma_h\left(\nicefrac{1}{\sqrt{n}}+\nicefrac{1}{\alpha^2}\right) + \gamma_0\left( L +\nicefrac{\bar L}{\alpha^2}\right)}{T^{\nicefrac{1}{3}}}\right).
\end{eqnarray*}

\newpage 
\section{EF21-HM}

In \algname{EF21-HM}, we update the iterates $\{x^t\}$ according to: 
\begin{eqnarray}
    x^{t+1} = x^t - \gamma_t \frac{g^t}{\norm{g^t}}, \notag
\end{eqnarray}
where $g^t = \frac{1}{n}\sum_{i=1}^n g_i^t$ and $v^t = \frac{1}{n}\sum_{i=1}^n v_i^t$ with $g_i^t,v_i^t$ being governed by 
\begin{eqnarray}
    g_i^{t+1} & = & g_i^t + \mathcal{C}_i^{t+1}(v_i^{t+1} - g_i^t), \quad \text{and} \notag\\
    v_i^{t+1} & = & (1-\eta_t)(v_i^t + \nabla^2 f_i(x^{t+1};\xi_i^{t+1}) (x^{t+1} - x^t) ) + \eta_t \nabla f_i(x^{t+1};\xi_i^{t+1}). \notag
\end{eqnarray}

\subsection{Convergence Proof}

We prove the result in the following steps.

\subsubsection{Deriving the descent inequality} From \Cref{lem:descent_lemma}, we obtain \eqref{eq:descent_lemma}:
\begin{equation*}
    \Delta_{t+1} + \gamma_t \|\nabla f(x^t)\| \le \Delta_t  + 2 \gamma_t \norm{v^t- \nabla f(x^t)} +\frac{ 2\gamma_t}{n}\sum_{i=1}^n\norm{g_i^t-v_i^t} + \frac{\gamma_t^2L}{2}.
\end{equation*}

\subsubsection{Error bound I}
Define $\cV_t \eqdef \frac{1}{n}\sum_{i=1}^n \norm{g_i^t-v_i^t}$. From~\Cref{lemma:boundg_and_v}, we have
\begin{eqnarray*}
   \Exp{\cV_{t+1}} &\leq& \sqrt{1-\alpha}    \Exp{\cV_{t}} + \frac{\sqrt{1-\alpha}}{n} \sum_{i=1}^n \Exp{\norm{v_i^{t+1}-v_i^t}}.
\end{eqnarray*}

To complete the upper-bound for $\Exp{\cV_{t+1}}$, we must upper-bound $\Exp{\norm{v_i^{t+1}-v_i^t}}$. 
Since
\begin{eqnarray*}
    v_i^{t+1}-v_i^t & = & - \eta_t (v_i^t - \nabla f_i(x^t)) + (1-\eta_t) ( \nabla^2 f_i(x^{t+1};\xi_i^{t+1})  - \nabla^2 f_i(x^{t+1}))(x^{t+1} - x^t) \\
    && -\eta_t\left[ \nabla f_i(x^t) - \nabla f_i(x^{t+1}) - \nabla^2 f_i(x^{t+1})(x^t -x^{t+1}) \right] \\
    && +  \nabla^2 f_i(x^{t+1})(x^{t+1}-x^t) + \eta_t \left(\nabla f_i(x^{t+1};\xi_i^{t+1}) - \nabla f_i(x^{t+1})\right)\\
    &=&- \eta_t (v_i^t - \nabla f_i(x^t)) + (1-\eta_t) S_{i,t+1} +   \nabla^2 f_i(x^{t+1})(x^{t+1}-x^t) \\
    && - \eta_t Z_{f_i}(x^{t}, x^{t+1}) + \eta_t e_{i,t+1},
\end{eqnarray*}
where  we denoted $S_{i,t+1} \eqdef ( \nabla^2 f_i(x^{t+1};\xi_i^{t+1})  - \nabla^2 f_i(x^{t+1}))(x^{t+1} - x^t)$, and  $Z_{f_i}(x^{t}, x^{t+1}) \eqdef \nabla f_i(x^t) - \nabla f_i(x^{t+1}) - \nabla^2 f_i(x^{t+1})(x^t -x^{t+1})$, and $e_{i,t+1} \eqdef \nabla f_i(x^{t+1};\xi_i^{t+1}) - \nabla f_i(x^{t+1})$, we obtain 
\begin{eqnarray*}
    \Exp{\norm{v_i^{t+1}-v_i^t}} &\overset{(a)}{\leq}& \eta_t \Exp{\norm{v_i^t - \nabla f_i(x^t)}} + (1-\eta_t) \Exp{\norm{S_{i,t+1}}} \\  
    && + \eta_t \Exp{\norm{Z_{f_i}(x^{t}, x^{t+1})}} + \eta_t\Exp{\norm{ e_{i,t+1}}}\\
    && + \Exp{\norm{\nabla^2 f_i(x^{t+1})(x^{t+1}-x^t)}}\\
    &\overset{(b)}{\leq}& \eta_t \Exp{\norm{v_i^t - \nabla f_i(x^t)}} + \Exp{\norm{\nabla^2 f_i(x^{t+1})}_{\text{op}}\norm{x^{t+1}-x^t}}\\
    && + (1-\eta_t) \Exp{\norm{\nabla^2 f_i(x^{t+1};\xi_i^{t+1})  - \nabla^2 f_i(x^{t+1})}\norm{x^{t+1}-x^t}} \\  
    && + \frac{\eta_t L_{h,i}}{2} \Exp{\norm{x^{t+1} -x^t}^2} + \eta_t\Exp{\norm{ e_{i,t+1}}}\\
    &\overset{(c)}{\leq}& \eta_t \Exp{\norm{v_i^t - \nabla f_i(x^t)}} + L_{i}\gamma_t + (1-\eta_t)\gamma_t\sigma_h + \frac{1}{2}\eta_t\gamma_t^2 L_{h,i} + \eta_t \sigma_g
\end{eqnarray*}
where in $(a)$ we used triangle inequality, in $(b)$ we used \Cref{assum:L_Hessian_comp_fcn}, 
in $(c)$ we used \Cref{assum:L_comp_fcn}, \Cref{assum:stoc_g_h} and the update rule for $x^{t+1}$. 

Therefore, we have 
\begin{eqnarray*}
   \Exp{\cV_{t+1}}  
   & \leq &  \sqrt{1-\alpha}    \Exp{\cV_{t}} + \frac{\sqrt{1-\alpha}}{n} \sum_{i=1}^n \Exp{\norm{v_i^{t+1}-v_i^t}} \\
   & \leq &  \sqrt{1-\alpha}    \Exp{\cV_{t}} + \sqrt{1-\alpha} \eta_t \Exp{\cU_t}  \\
   && + \sqrt{1-\alpha}(1-\eta_t) \gamma_t \sigma_h + \frac{\sqrt{1-\alpha}}{2} \gamma_t^2 \bar L_{h}  + \sqrt{1-\alpha} \gamma_t \bar L + \sqrt{1-\alpha} \eta_t \sigma_g,
\end{eqnarray*}
where $\bar L_h = \frac{1}{n}\sum_{i=1}^n L_{h,i}$ and $\bar L = \frac{1}{n}\sum_{i=1}^n L_{i}$.

\subsubsection{Error bound II} 

Define $\cU_t \eqdef \frac{1}{n}\sum_{i=1}^n \norm{v_i^t - \nabla f_i(x^t)}$. Then, using the update rule for $x^{t+1}$, we have 
\begin{eqnarray*}
    v_i^{t+1} - \nabla f_i(x^{t+1}) & = & (1-\eta_t)(v_i^t - \nabla f_i(x^t)) \\
    && + (1-\eta_t)(\nabla^2 f_i(x^{t+1};\xi_i^{t+1}) - \nabla^2 f_i(x^{t+1}))(x^{t+1}-x^t) \\
    && + (1-\eta_t) ( \nabla^2 f_i(x^{t+1})(x^{t+1}-x^t) - [\nabla f_i(x^{t+1}) - \nabla f_i(x^t)] ) \\
    && + \eta_t (\nabla f_i(x^{t+1};\xi_i^{t+1}) - \nabla f_i(x^{t+1}))\\
    &=& (1-\eta_t)(v_i^t - \nabla f_i(x^t)) + (1-\eta_t)S_{i,t+1}\\
    &&- (1-\eta_t)Z_{f_i}(x^t, x^{t+1}) + \eta_t e_{i,t+1}.
\end{eqnarray*} 
Therefore, using triangle inequality, we obtain
\begin{eqnarray*}
    \Exp{\cU_{t+1}} &\leq&   \frac{1-\eta_t}{n}\sum_{i=1}^n \Exp{\norm{v_i^{t} - \nabla f_i(x^{t})}}   + (1-\eta_t) \frac{1}{n}\sum_{i=1}^n \Exp{\norm{S_{i,t+1}}}  \\
    && + (1-\eta_t) \frac{1}{n}\sum_{i=1}^n \Exp{\norm{ Z_{f_i}(x^t,x^{t+1}) } }   + \eta_t \frac{1}{n}\sum_{i=1}^n \Exp{\norm{e_{i,t+1}}}\\
    &\overset{(a)}{\leq}& \frac{1-\eta_t}{n}\sum_{i=1}^n \Exp{\norm{v_i^{t} - \nabla f_i(x^{t})}}\\
    && +  \frac{1-\eta_t}{n}\sum_{i=1}^n \Exp{\norm{\nabla^2 f_i(x^{t+1};\xi^{t+1}_i) - \nabla^2 f_i(x^{t+1})}\norm{x^{t+1}-x^t}}  \\
    && + \frac{1-\eta_t}{2n}\sum_{i=1}^n L_{h,i}\Exp{\sqnorm{ x^{t+1}-x^t } } + \frac{\eta_t }{n}\sum_{i=1}^n \Exp{\norm{e_{i,t+1}}},
\end{eqnarray*}
where in $(a)$ we used \Cref{assum:L_Hessian_comp_fcn}.
Next, bounding the second term and the fourth term via \Cref{assum:stoc_g_h}, we get
\begin{eqnarray*}
    \Exp{\cU_{t+1}} & \leq &  \frac{1-\eta_t}{n}\sum_{i=1}^n \Exp{\norm{v_i^{t} - \nabla f_i(x^{t})}}  + (1-\eta_t) \sigma_h \Exp{\norm{x^{t+1}-x^t}}  \\
    && + (1-\eta_t) \frac{\bar L_{h}}{2} \Exp{\sqnorm{x^{t+1}-x^t}} + \eta_t \sigma_g\\
    & \leq & \frac{1-\eta_t}{n}\sum_{i=1}^n \Exp{\norm{v_i^{t} - \nabla f_i(x^{t})}}  + (1-\eta_t) \sigma_h \Exp{\norm{x^{t+1}-x^t}}  \\
    && + (1-\eta_t) \frac{\bar L_{h}}{2} \Exp{\sqnorm{x^{t+1}-x^t}} + \eta_t \sigma_g\\
    &\leq& (1-\eta_t)\Exp{\cU_{t}} + (1-\eta_t) \gamma_t\sigma_h  + (1-\eta_t)\gamma_t^2 \bar{L}_h + \eta_t\sigma_g,
\end{eqnarray*}
where in the last inequality we used the update rule for $x^{t+1}$.

\subsubsection{Error bound III} 
From the definition of $v^t = \frac{1}{n}\sum_{i=1}^n v_i^t$ and denoting $\hat{e}_{t+1} \eqdef v^{t+1} - \nabla f(x^{t+1}) $, we have 
\begin{eqnarray*}
    \hat{e}_{t+1}
    & = &  \frac{1}{n}\sum_{i=1}^n (v_i^{t+1} - \nabla f_i(x^{t+1}))\\
    & = & (1-\eta_t)(v^t - \nabla f(x^t)) \\
    && + (1-\eta_t) \frac{1}{n}\sum_{i=1}^n (\nabla^2 f_i(x^{t+1};\xi_i^{t+1}) - \nabla^2 f_i(x^{t+1}))(x^{t+1}-x^t) \\
    && + (1-\eta_t) \frac{1}{n}\sum_{i=1}^n ( \nabla^2 f_i(x^{t+1})(x^{t+1}-x^t) - [\nabla f_i(x^{t+1}) - \nabla f_i(x^t)] ) \\
    && + \eta_t \frac{1}{n}\sum_{i=1}^n (\nabla f_i(x^{t+1};\xi_i^{t+1}) - \nabla f_i(x^{t+1}))\\
    &=& (1-\eta_t)\hat{e}_t + (1-\eta_t)S_{t+1} - (1-\eta_t)Z_{f}(x^t, x^{t+1}) + \eta_t e_{t+1},
\end{eqnarray*}
where we used the following notation
\begin{eqnarray*}
    S_{t+1} &\eqdef& \frac{1}{n}\sum^n_{i=1} S_{i, t+1} = \frac{1}{n}\sum_{i=1}^n (\nabla^2 f_i(x^{t+1};\xi_i^{t+1}) - \nabla^2 f_i(x^{t+1}))(x^{t+1}-x^t);\\
    Z_f(x^t, x^{t+1}) &\eqdef& \nabla f(x^t) -\nabla f(x^{t+1}) - \nabla^2 f(x^{t+1})(x^{t} -x^{t+1});\\
    e_{t+1} &\eqdef& \frac{1}{n}\sum^n_{i=1} e_{i,t+1} = \frac{1}{n}\sum_{i=1}^n (\nabla f_i(x^{t+1};\xi_i^{t+1}) - \nabla f_i(x^{t+1})).
\end{eqnarray*}
Therefore, we obtain
\begin{eqnarray*}
    \hat{e}_{t+1} 
    & = & \prod_{\tau=0}^t (1-\eta_\tau) \hat{e}_{0}   + \sum_{j=0}^t \left(\prod_{\tau=j}^t (1-\eta_\tau)\right) S_{j+1}\\
    &&- \sum_{j=0}^t \left(\prod_{\tau=j}^t (1-\eta_\tau)\right) Z_f(x^j, x^{j+1})  + \sum_{j=0}^t \left(\prod_{\tau=j+1}^t (1-\eta_\tau)\right) \eta_j e_{j+1}.
\end{eqnarray*}
If $\eta_0 = 1$, then 
\begin{eqnarray*}
  \hat{e}_{t+1} 
    & = & \sum_{j=0}^t \left(\prod_{\tau=j}^t (1-\eta_\tau)\right) S_{j+1} - \sum_{j=0}^t \left(\prod_{\tau=j}^t (1-\eta_\tau)\right) Z_f(x^j, x^{j+1})  \\
    && + \sum_{j=0}^t \left(\prod_{\tau=j+1}^t (1-\eta_\tau)\right) \eta_j e_{j+1}.
\end{eqnarray*}

Therefore, from the definition of the Euclidean norm, by the triangle inequality, and by Jensen's inequality. 
\begin{eqnarray*}
 \Exp{  \norm{\hat{e}_{t+1}}  }  &\leq&   \sqrt{ \Exp{ \sqnorm{\sum_{j=0}^t \left(\prod_{\tau=j}^t (1-\eta_\tau)\right) {S}_{j+1}}  } }  + \Exp{ \norm{\sum_{j=0}^t \left(\prod_{\tau=j}^t (1-\eta_\tau)\right) Z_f(x^j, x^{j+1}) } } \\
&& + \sqrt{ \Exp{ \sqnorm{\sum_{j=0}^t \left(\prod_{\tau=j+1}^t (1-\eta_\tau)\right) \eta_j e_{j+1}} } }.
\end{eqnarray*}

Next, from \Cref{assum:stoc_g_h}, we can prove that $\Exp{S_{t+1}} = 0$, $\Exp{ \inp{S_l}{S_{i}}} = 0$ for $l \neq i$, $\Exp{e_{t+1}}=0$, $\Exp{\inp{e_l}{e_i}}=0$ for $l \neq i$.  Hence, 
\begin{eqnarray*}
    \Exp{  \norm{e_{t+1}}  } &\leq& \sqrt{ \sum_{j=0}^t \left(\prod_{\tau=j+1}^t (1-\eta_\tau)^2 \right) \Exp{ \sqnorm{ {S}_{j+1}}  } } + \sqrt{ \sum_{j=0}^t \left(\prod_{\tau=j}^t (1-\eta_\tau)^2\right) \eta_j^2 \Exp{ \sqnorm{ e_{j+1}} } }\\
    && + + \sum_{j=0}^t \left(\prod_{\tau=j+1}^t (1-\eta_\tau)\right)\Exp{ \norm{ Z_f(x^j, x^{j+1}) } }.
\end{eqnarray*}

Next, from \Cref{assum:stoc_g_h} and from the definition of $x^{t+1}$, we can prove that 
\begin{eqnarray*}
    \Exp{\sqnorm{S_{j+1}}} 
    & \leq & \frac{1}{n^2} \sum_{i=1}^n \Exp{\sqnorm{\nabla^2 f_i(x^{j+1};\xi_i^{j+1}) - \nabla^2 f_i(x^{j+1})} \cdot \sqnorm{x^{j+1} - x^j}} \leq \frac{\gamma_j^2}{n} \sigma_h^2,\\
     &\text{and}& \\
    \Exp{\sqnorm{e_{j+1}}} & \leq & \frac{1}{n^2}\sum_{i=1}^n \Exp{\sqnorm{ \nabla f_i(x^{j+1};\xi_i^{j+1}) - \nabla f_i(x^{j+1})  }} \leq \frac{\sigma_g^2}{n}.
\end{eqnarray*}
Therefore, 
\begin{eqnarray*}
    \Exp{  \norm{\hat{e}_{t+1}}  } &\leq&  \frac{\sigma_h}{\sqrt{n}}\sqrt{ \sum_{j=0}^t \left(\prod_{\tau=j+1}^t (1-\eta_\tau)^2 \right) \gamma_j^2  } + \frac{\sigma_g}{\sqrt{n}}\sqrt{ \sum_{j=0}^t \left(\prod_{\tau=j}^t (1-\eta_\tau)^2\right) \eta_j^2 }\\
    && + \sum_{j=0}^t \left(\prod_{\tau=j+1}^t (1-\eta_\tau)\right)\Exp{ \norm{ Z_f(x^j, x^{j+1}) } }.
\end{eqnarray*}

Next, by \Cref{assum:L_Hessian_fcn},  
\begin{eqnarray*}
    \Exp{  \norm{\hat{e}_{t+1}}  } &\leq& \frac{\sigma_h}{\sqrt{n}}\sqrt{ \sum_{j=0}^t \left(\prod_{\tau=j+1}^t (1-\eta_\tau)^2 \right) \gamma_j^2  } + \frac{ L_h}{2} \sum_{j=0}^t \left(\prod_{\tau=j+1}^t (1-\eta_\tau)\right)  \gamma_j^2  \\
    && + \frac{\sigma_g}{\sqrt{n}}\sqrt{ \sum_{j=0}^t \left(\prod_{\tau=j}^t (1-\eta_\tau)^2\right) \eta_j^2 } 
\end{eqnarray*}

If $\eta_t = \left( \frac{2}{t+2} \right)^{2/3}$ and $\gamma_t = \gamma_0 \left( \frac{2}{t+2} \right)^{2/3}$, then we can prove  that $\eta_0=1$, and from~\Cref{lemma:sum_prod} that
\begin{eqnarray*}
    \Exp{  \norm{v^{t+1} - \nabla f(x^{t+1})}  } 
    & \leq & \frac{\sigma_h }{\sqrt{n}} \sqrt{C(\nicefrac{4}{3},\nicefrac{2}{3}) \frac{\gamma_{t+1}^2}{\eta_{t+1}}} + \frac{C(\nicefrac{4}{3},\nicefrac{2}{3})}{2} L_h \frac{\gamma_{t+1}^2}{\eta_{t+1}} \\
    && + \frac{\sigma_g}{\sqrt{n}} \sqrt{C(\nicefrac{4}{3},\nicefrac{2}{3}) \eta_{t+1}}.
\end{eqnarray*}

\subsubsection{Bounding Lyapunov function}
Define our Lyapunov function $V_t = \Delta_{t} + C_{1,t} \cV_t + C_{2,t} \cU_t$ with $C_{1,t} = \frac{2\gamma_t}{1-\sqrt{1-\alpha}}$ and $C_{2,t} = \frac{2\gamma_t \sqrt{1-\alpha}}{1-\sqrt{1-\alpha}}$. Then, using \eqref{eq:descent_lemma}, we have 
\begin{eqnarray*}
    \Exp{V_{t+1}} 
    & = & \Exp{\Delta_{t+1} + C_{1,t+1} \cV_t + C_{2,t+1} \cU_{t+1}} \\
    & \leq & \Exp{\Delta_t} - \gamma_t \Exp{\norm{\nabla f(x^t)}} + 2 \gamma_t \Exp{\norm{v^t- \nabla f(x^t)}} + 2\gamma_t \Exp{\cV_t} + \frac{\gamma_t^2L}{2} \\
    && + C_{1,t+1} \Exp{\cV_{t+1}} + C_{2,t+1} \Exp{\cU_{t+1}}.
\end{eqnarray*}

Next, by the upper-bounds for $\Exp{\cV_{t+1}}$ and  $\Exp{\cU_{t+1}}$, 
\begin{eqnarray*}
   \Exp{ V_{t+1} }&\leq& \Exp{\Delta_t} - \gamma_t \Exp{\norm{\nabla f(x^t)}} + 2 \gamma_t \Exp{\norm{v^t- \nabla f(x^t)}} + 2\gamma_t \Exp{\cV_t} + \frac{\gamma_t^2L}{2} \\
    && + C_{1,t+1} ( \sqrt{1-\alpha}    \Exp{\cV_{t}} + \sqrt{1-\alpha} \eta_t \Exp{\cU_t})  \\
    && + C_{1,t+1}(\sqrt{1-\alpha}(1-\eta_t) \gamma_t \sigma_h + \frac{\sqrt{1-\alpha}}{2} \gamma_t^2 \bar L_{h}  + \sqrt{1-\alpha} \gamma_t \bar L + \sqrt{1-\alpha} \eta_t \sigma_g ) \\
    && + C_{2,t+1} ( (1-\eta_t) \Exp{\cU_t}  + (1-\eta_t) \gamma_t \sigma_h + (1-\eta_t) \gamma_t^2 \bar L_h   + \eta_t \sigma_g ).
\end{eqnarray*}

Next,  since $\gamma_{t+1} \leq \gamma_t$, we can prove that $C_{1,t+1} \leq C_{1,t}$ and that $C_{2,t+1} \leq C_{2,t}$, and also that 
\begin{eqnarray*}
    2\gamma_t + C_{1,t+1}\sqrt{1-\alpha}  & \leq & 2\gamma_t + C_{1,t}\sqrt{1-\alpha} = C_{1,t} \\
    C_{1,t+1} \sqrt{1-\alpha}\eta_t + C_{2,t+1} (1-\eta_t) & \leq& C_{1,t} \sqrt{1-\alpha}\eta_t + C_{2,t} (1-\eta_t)= C_{2,t}
\end{eqnarray*}
Therefore, 
\begin{eqnarray*}
    \Exp{ V_{t+1} } 
    & \leq & \Exp{V_t} - \gamma_t \Exp{\norm{\nabla f(x^t)}} + 2 \gamma_t \Exp{\norm{v^t- \nabla f(x^t)}}  + \frac{\gamma_t^2L}{2} \\
    && + C_{1,t} (\sqrt{1-\alpha}(1-\eta_t) \gamma_t \sigma_h + \frac{\sqrt{1-\alpha}}{2} \gamma_t^2 \bar L_{h}  + \sqrt{1-\alpha} \gamma_t \bar L + \sqrt{1-\alpha} \eta_t \sigma_g ) \\
    && + C_{2,t} \left( (1-\eta_t) \gamma_t \sigma_h + (1-\eta_t) \gamma_t^2 \bar L_h   + \eta_t \sigma_g  \right) \\
    & \leq & \Exp{V_t} - \gamma_t \Exp{\norm{\nabla f(x^t)}} + 2 \gamma_t \Exp{\norm{v^t- \nabla f(x^t)}}   \\
    && + \gamma_t^2  \left( \frac{4\sqrt{1-\alpha}}{1-\sqrt{1-\alpha}} \sigma_h + \frac{\sqrt{1-\alpha}}{1-\sqrt{1-\alpha}} \bar L  + \frac{L}{2}\right) \\
    && + \gamma_t^3 \left( \frac{4\sqrt{1-\alpha}}{1-\sqrt{1-\alpha}} \right) \bar L_h + \eta_t\gamma_t \frac{4\sqrt{1-\alpha}}{1-\sqrt{1-\alpha}} \sigma_g.
\end{eqnarray*}

Next, from the upper-bound of $\Exp{\norm{v^t- \nabla f(x^t)}}$, 
\begin{eqnarray*}
    \Exp{ V_{t+1} } & \leq & \Exp{V_t} - \gamma_t \Exp{\norm{\nabla f(x^t)}} \\
    && + 2 \gamma_t \left( \frac{\sigma_h }{\sqrt{n}} \sqrt{C(\nicefrac{4}{3},\nicefrac{2}{3}) \frac{\gamma_{t}^2}{\eta_{t}}} + \frac{ L_h C(\nicefrac{4}{3},\nicefrac{2}{3})}{2} \frac{\gamma_{t}^2}{\eta_{t}}  + \frac{\sigma_g}{\sqrt{n}} \sqrt{C(\nicefrac{4}{3},\nicefrac{2}{3}) \eta_{t}}  \right)  \\
    && + \gamma_t^2  \left( \frac{4\sqrt{1-\alpha}}{1-\sqrt{1-\alpha}} \sigma_h + \frac{\sqrt{1-\alpha}}{1-\sqrt{1-\alpha}} \bar L  + \frac{L}{2}\right) \\
    && + \gamma_t^3 \left( \frac{4\sqrt{1-\alpha}}{1-\sqrt{1-\alpha}} \right) \bar L_h + \eta_t\gamma_t \frac{4\sqrt{1-\alpha}}{1-\sqrt{1-\alpha}} \sigma_g \\
    & \leq & \Exp{V_t} - \gamma_t \Exp{\norm{\nabla f(x^t)}}  + \frac{2 C_1\sigma_h }{\sqrt{n}}\frac{\gamma_t^2}{\sqrt{\eta_t}}   + C^2_1 L_h \frac{\gamma_t^3}{\eta_t} + \frac{2C_1\sigma_g }{\sqrt{n}} \gamma_t \sqrt{\eta_t}  \\
    && +   \left( \frac{L}{2} +\frac{4\sigma_h}{\alpha^2}  + \frac{\bar L}{\alpha^2} \right) \gamma_t^2  + \frac{4\bar L_h}{\alpha^2}\gamma_t^3  + \frac{4\sigma_g}{\alpha^2} \eta_t\gamma_t.
\end{eqnarray*}
where we used the notation: $C_1 = \sqrt{ C(\nicefrac{4}{3},\nicefrac{2}{3}) }$ and used $\frac{\sqrt{1-\alpha}}{1-\sqrt{1-\alpha}} \leq \frac{1}{\alpha^2}$.

\subsubsection{Deriving the convergence rate}
By re-arranging the terms and by the telescopic series, 
\begin{eqnarray*}
    \frac{\sum_{t=0}^{T-1} \gamma_t \Exp{\norm{\nabla f(x^t)}}}{\sum_{t=0}^{T-1} \gamma_t}  
    & \leq & \frac{\Exp{V_0} -\Exp{V_T}}{\sum_{t=0}^{T-1}\gamma_t}   + \frac{2 C_1 \sigma_h }{\sqrt{n}} \frac{\sum_{t=0}^{T-1} \gamma_t^2\eta_t^{-\nicefrac{1}{2}}}{\sum_{t=0}^{T-1}\gamma_t} \\
    && + C^2_1 L_h  \frac{\sum_{t=0}^{T-1} \gamma_t^3 \eta_t^{-1}}{\sum_{t=0}^{T-1}\gamma_t}  + \frac{2C_1 \sigma_g}{\sqrt{n}} \frac{\sum_{t=0}^{T-1} \gamma_t \sqrt{\eta_t}}{ \sum_{t=0}^{T-1} \gamma_t} \\
    && +  \left(\frac{L}{2} + \frac{4\sigma_h}{\alpha^2}  + \frac{\bar L}{\alpha^2} \right) \frac{\sum_{t=0}^{T-1} \gamma_t^2}{ \sum_{t=0}^{T-1} \gamma_t} \\
    && +     \frac{4\bar L_h}{\alpha^2}   \frac{\sum_{t=0}^{T-1} \gamma_t^3}{ \sum_{t=0}^{T-1}\gamma_t} + \frac{4\sigma_g }{\alpha^2}  \frac{\sum_{t=0}^{T-1} \eta_t\gamma_t}{ \sum_{t=0}^{T-1}\gamma_t}.
\end{eqnarray*}

Since $\eta_t = \left( \frac{2}{t+2} \right)^{\nicefrac{2}{3}}$ and $\gamma_t = \gamma_0 \left( \frac{2}{t+2} \right)^{\nicefrac{2}{3}}$,  for $T \geq 1$ we obtain
\begin{eqnarray*}
    \sum_{t=0}^{T-1} \gamma_t &\geq&  T \gamma_{T-1} = \gamma_0T  \left(\frac{2}{T+1}\right)^{\nicefrac{2}{3}} \geq \gamma_0 T^{\nicefrac{1}{3}};\\
    \sum_{t=0}^{T-1} \gamma_t^2 \eta_t^{-\nicefrac{1}{2}}  &=& \gamma_0^2\sum_{t=0}^{T-1}  \left( \frac{2}{t+2} \right)^{\nicefrac{4}{3}} \left( \frac{2}{t+2} \right)^{-\nicefrac{1}{3} } = 2\gamma_0^2 \sum_{t=0}^{T-1} \frac{1}{t+2}   = 2\gamma_0^2 \sum^{T}_{t=1}\frac{1}{1+t}\\
    &\leq& 2\gamma^2_0 \int^T_{1}\frac{1}{1+t}dt = 2\gamma^2_0 \left(\log \left(T+1\right) -\log(2)\right) \leq 2\gamma^2_0 \log \left(T+1\right); \\
    \sum_{t=0}^{T-1} \gamma_t \sqrt{\eta_t} &=& \gamma_0 \sum_{t=0}^{T-1} \left( \frac{2}{t+2} \right)^{\nicefrac{2}{3}} \left( \frac{2}{t+2} \right)^{\nicefrac{1}{3} } =   2\gamma_0 \sum_{t=0}^{T-1}\frac{1}{t+2} \leq 2\gamma_0 \log\left(T+1\right);\\
    \sum_{t=0}^{T-1} \gamma_t^2 &=& \gamma_0^2 \sum_{t=0}^{T-1} \left( \frac{2}{t+2} \right)^{\nicefrac{4}{3}}  = 2^{\nicefrac{4}{3}}\gamma_0^2 \sum_{t=0}^{T-1} \frac{1}{(t+2)^{\nicefrac{4}{3}}} = 2^{\nicefrac{3}{2}}\gamma_0^2 \sum_{t=1}^{T} \frac{1}{(1+t)^{\nicefrac{4}{3}}}\\
    &\leq& 2^{\nicefrac{4}{3}}\gamma_0^2 \int^T_{1}\frac{1}{(1+t)^{\nicefrac{4}{3}}}dt  = 2^{\nicefrac{4}{3}}\cdot 3\gamma^2_0 \left(\frac{1}{\sqrt[3]{2}} - \frac{1}{\sqrt[3]{T+1}}\right) \leq 6\gamma^2_0;\\
    \sum_{t=0}^{T-1} \gamma_t \eta_t &=& \frac{1}{\gamma_0}\sum_{t=0}^{T-1} \gamma_t^2 \leq 6\gamma_0;\qquad\quad \sum_{t=0}^{T-1}\gamma^3_t\eta^{-1}_t ~~=~~  \gamma_0\sum_{t=0}^{T-1}\gamma^2_t \leq 6\gamma^3_0\\
    \sum_{t=0}^{T-1} \gamma_t^3  &=& \gamma^3_0\sum_{t=0}^{T-1} \left( \frac{2}{t+2} \right)^{2} = 4\gamma^3_0 \sum_{t=1}^{T} \frac{1}{(1+t)^2} \\
    &\leq& 4\gamma^3_0 \int_{1}^{T} \frac{1}{(1+t)^2}dt = 4\gamma^3_0 \left(\frac{1}{2}-\frac{1}{T+1}\right) \leq 2\gamma^3_0.
    \end{eqnarray*}
Therefore, denoting $\tilde{x}^T$ as a point randomly chosen from $\{x^0,x^1,\ldots,x^{T-1}\}$ with probability $\nicefrac{\gamma_t}{\sum_{t=0}^{T-1}\gamma_t}$ for $t=0,1,\ldots,T-1$, we obtain 
\begin{eqnarray*}
    \Exp{\norm{\nabla f(\tilde{x}^T)}}
    & \leq & \frac{\Exp{V_0}}{\gamma_0 T^{\nicefrac{1}{3}}}   + \frac{2 C_1 \sigma_h }{\sqrt{n}} \frac{2\gamma_0^2 \log\left(T+1\right)}{\gamma_0 T^{\nicefrac{1}{3}}} + C^2_1  L_h  \frac{6\gamma^3_0}{\gamma_0 T^{\nicefrac{1}{3}}} +     \frac{4\bar L_h}{\alpha^2}   \frac{2\gamma^3_0}{ \gamma_0 T^{\nicefrac{1}{3}}} \\
    && +  \left(\frac{L}{2} + \frac{4\sigma_h}{\alpha^2}  + \frac{\bar L}{\alpha^2} \right) \frac{6\gamma_0^2}{ \gamma_0 T^{\nicefrac{1}{3}}} + \frac{2C_1 \sigma_g}{\sqrt{n}} \frac{2\gamma_0 \log\left(T+1\right)}{ \gamma_0 T^{\nicefrac{1}{3}}} + \frac{4\sigma_g }{\alpha^2}  \frac{6\gamma_0}{ \gamma_0 T^{\nicefrac{1}{3}}}\\
    &=& \widetilde{\cO}\left(\frac{\nicefrac{V_0}{\gamma_0} + \left(\sigma_g + \gamma_0\sigma_h\right)\left(\nicefrac{1}{\sqrt{n}} + \nicefrac{1}{\alpha^2}\right)+ \gamma_0\left(L + \nicefrac{\bar L}{\alpha^2} \right) +  \gamma^2_0\left(L_h +\nicefrac{\bar L_h}{\alpha^2}\right) }{T^{\nicefrac{1}{3}}}\right).
\end{eqnarray*}

\newpage

\section{EF21-MVR}

In \algname{EF21-MVR}, we update the iterates $\{x^t\}$ according to: 
\begin{eqnarray}
    x^{t+1} = x^t - \gamma_t \frac{g^t}{\norm{g^t}}, \notag
\end{eqnarray}
where $g^t = \frac{1}{n}\sum_{i=1}^n g_i^t$ and $v^t = \frac{1}{n}\sum_{i=1}^n v_i^t$ with $g_i^t,v_i^t$ being governed by 
\begin{eqnarray}
    g_i^{t+1} & = & g_i^t + \mathcal{C}_i^{t+1}(v_i^{t+1} - g_i^t), \quad \text{and} \notag\\
    v_i^{t+1} & = & (1-\eta_t)(v_i^t + \nabla f_i(x^{t+1};\xi_i^{t+1}) - \nabla f_i(x^t;\xi_i^{t+1})) + \eta_t \nabla f_i(x^{t+1};\xi_i^{t+1}). \label{eq:mommetum_update_mvr}
\end{eqnarray}

\subsection{Convergence Proof}

We prove the result in the following steps. 

\subsubsection{Deriving the descent inequality} From \Cref{lem:descent_lemma}, we obtain \eqref{eq:descent_lemma}:
\begin{equation*}
    \Delta_{t+1} + \gamma_t \|\nabla f(x^t)\| \le \Delta_t  + 2 \gamma_t \norm{v^t- \nabla f(x^t)} +\frac{ 2\gamma_t}{n}\sum_{i=1}^n\norm{g_i^t-v_i^t} + \frac{\gamma_t^2L}{2}.
\end{equation*}

\subsubsection{Error bound I} 
Define $\cV_t \eqdef \frac{1}{n}\sum_{i=1}^n \norm{g_i^t-v_i^t}$. From~\Cref{lemma:boundg_and_v}, 
\begin{eqnarray*}
   \Exp{\cV_{t+1}} \leq \sqrt{1-\alpha}    \Exp{\cV_{t}} + \frac{\sqrt{1-\alpha}}{n} \sum_{i=1}^n \Exp{\norm{v_i^{t+1}-v_i^t}}.
\end{eqnarray*}

To complete the upper-bound for $\Exp{\cV_{t+1}}$, we need to upper-bound $\Exp{\norm{v_i^{t+1}-v_i^t}}$. 
Since 
\begin{eqnarray*}
    v_i^{t+1} - v_i^t = - \eta_t (v_i^t - \nabla f_i(x^t)) - \eta_t S_{i,t+1} + (1-\eta_t)S_{i,t+1} + \eta_t e_{i,t+1},
\end{eqnarray*}
where we defined $S_{i,t+1} \eqdef \nabla f_i(x^{t};\xi_i^{t+1}) - \nabla f_i(x^{t+1};\xi_i^{t+1})$  and $e_{i,t+1} \eqdef \nabla f_i(x^{t};\xi_i^{t+1}) - \nabla f_i(x^t)$, from the definition of the Euclidean norm and by the triangle inequality, we get
\begin{eqnarray*}
    \Exp{\norm{v_i^{t+1}-v_i^t}} & \leq & \eta_t \Exp{\norm{v_i^t -\nabla f_i(x^t)}} + \Exp{ \norm{S_{i,t+1}} } + \Exp{\norm{e_{i,t+1}}} \\
    & \leq & \eta_t \Exp{\norm{v_i^t -\nabla f_i(x^t)}} + \sqrt{ \Exp{ \sqnorm{S_{i,t+1}} }  } + \sqrt{ \Exp{\sqnorm{e_{i,t+1}}} }, 
\end{eqnarray*}
where in the last inequality we used Jensen's inequality. 
Next, by \Cref{assum:L_stoc_comp_fcn}, \Cref{assum:stoc_g_h}, and by the update rule for $x^{t+1}$,  
\begin{eqnarray*}
    \Exp{\norm{v_i^{t+1}-v_i^t}} 
    & \leq & \eta_t \Exp{\norm{v_i^t - \nabla f_i(x^t)}} +  \gamma_t L_{\text{ms},i} + \eta_t \sigma_g.
\end{eqnarray*}
Plugging the above result into the upper-bound for $\Exp{\cV_{t+1}}$, and then defining $\cU_t \eqdef \frac{1}{n}\sum\limits_{i=1}^n \norm{v_i^t-\nabla f_i(x^t)}$,  we obtain
\begin{eqnarray}\label{eqn:MVR_V_t}
   \Exp{\cV_{t+1}}
   & \leq &  \sqrt{1-\alpha}    \Exp{\cV_{t}} + \sqrt{1-\alpha} \eta_t \Exp{\cU_t} + \sqrt{1-\alpha} \gamma_t \bar L_{\text{ms}} + \sqrt{1-\alpha} \eta_t \sigma_g,
\end{eqnarray}
where $\bar L_{\text{ms}} = \frac{1}{n}\sum_{i=1}^n L_{\text{ms},i}$.

\subsubsection{Error bound II} 
Define $\cU_t \eqdef \frac{1}{n}\sum\limits_{i=1}^n \norm{v_i^t - \nabla f_i(x^t)}$. Then, we have 
\begin{eqnarray*}
    v_i^{t+1} - \nabla f_i(x^{t+1}) 
    & \overset{(a)}{=}& (1-\eta_t)(v_i^t + \nabla f_i(x^{t+1};\xi_i^{t+1}) - \nabla f_i(x^t;\xi_i^{t+1})) \\
    && + \eta_t \nabla f_i(x^{t+1};\xi_i^{t+1}) - \nabla f_i(x^{t+1}) \\
    & \overset{(b)}{=} & (1-\eta_t)(v_i^t - \nabla f_i(x^t)) - (1-\eta_t) D_{i,t+1} + \eta_t e_{i,t+1},
\end{eqnarray*}
where  in $(a)$ we used the update rule for $v^{t+1}_i$, i.e. \eqref{eq:mommetum_update_mvr}, in $b$ we defined $D_{i,t+1} \eqdef \nabla f_i(x^{t};\xi_i^{t+1}) - \nabla f_i(x^{t+1};\xi_i^{t+1}) - [\nabla f_i(x^t) - \nabla f_i(x^{t+1})]$] and $e_{i,t+1} \eqdef \nabla f_i(x^{t+1};\xi_i^{t+1}) - \nabla f_i(x^{t+1})$.
Therefore, 
\begin{eqnarray*}
    \Exp{\cU_{t+1}}
    & \overset{(a)}{\leq}&  \frac{1-\eta_t}{n}\sum_{i=1}^n \Exp{\norm{v_i^{t} - \nabla f_i(x^{t})}} +\frac{1-\eta_t}{n}\sum_{i=1}^n \Exp{\norm{D_{i,t+1}}}  +  \frac{\eta_t}{n}\sum_{i=1}^n \Exp{\norm{e_{i,t+1}}} \\
    & \overset{(b)}{\leq} &  \frac{1-\eta_t}{n}\sum_{i=1}^n \Exp{\norm{v_i^{t} - \nabla f_i(x^{t})}} + \frac{1-\eta_t}{n}\sum_{i=1}^n \sqrt{\Exp{\sqnorm{D_{i,t+1}}}}\\
    && +  \frac{\eta_t}{n}\sum_{i=1}^n \sqrt{\Exp{\sqnorm{e_{i,t+1}}}} \\
    &\overset{(c)}{\leq} &  \frac{1-\eta_t}{n}\sum_{i=1}^n \Exp{\norm{v_i^{t} - \nabla f_i(x^{t})}} + \frac{1-\eta_t}{n}\sum_{i=1}^n \sqrt{\Exp{\sqnorm{S_{i,t+1}}}}\\
    && +  \frac{\eta_t}{n}\sum_{i=1}^n \sqrt{ \Exp{\sqnorm{e_{i,t+1}}}},
\end{eqnarray*}
where in $(a)$ we used the update rule for $v^{t+1}_i$, in $(b)$ we used Jensen's Inequality, in $(c)$ we used bias-variance decomposition, more precisely 
$\Exp{\sqnorm{D_{i,t+1}}}\leq \Exp{\sqnorm{S_{i,t+1}}} $.
Next, by \Cref{assum:L_stoc_comp_fcn} and \Cref{assum:stoc_g_h}, 
\begin{eqnarray}
    \Exp{\cU_{t+1}}
    & \leq &  (1-\eta_t) \Exp{\cU_t} + (1-\eta_t) \frac{1}{n}\sum_{i=1}^n L_{\text{ms},i}\sqrt{  \Exp{\sqnorm{ x^t - x^{t+1} } } }   + \eta_t \sigma_g \notag \\
    & \leq & (1-\eta_t) \Exp{\cU_t} + (1-\eta_t)\gamma_t  \bar L_{\text{ms}}  + \eta_t \sigma_g,  \label{eqn:MVR_U_t}
\end{eqnarray}
where  in the last inequality we used the update rule for $x^{t+1}$ and denoted $\bar L_{\text{ms}} = \frac{1}{n}\sum_{i=1}^n L_{\text{ms},i}$.

\subsubsection{Error bound III} 
From the definition of $v^t = \frac{1}{n}\sum_{i=1}^n v_i^t$,  we denote $\hat{e}_{t+1}$ as follows 
\begin{eqnarray*}
    \hat{e}_{t+1} & \eqdef & v^{t+1} - \nabla f(x^{t+1}) \\
    & = &  \frac{1}{n}\sum_{i=1}^n (v_i^{t+1} - \nabla f_i(x^{t+1}))\\
    & = & (1-\eta_t)(v^t - \nabla f(x^t)) - (1-\eta_t) \frac{1}{n}\sum_{i=1}^n D_{i,t+1}  + \eta_t \frac{1}{n}\sum_{i=1}^n e_{i,t+1},
\end{eqnarray*}
where $D_{i,t+1} \eqdef \nabla f_i(x^{t};\xi_i^{t+1}) - \nabla f_i(x^{t+1};\xi_i^{t+1}) - [\nabla f_i(x^t) - \nabla f_i(x^{t+1})]$ and $e_{i,t+1} = \nabla f_i(x^{t+1};\xi_i^{t+1}) - \nabla f_i(x^{t+1}) $.

Next, by applying Lemma~\ref{lemma:summing_the_recursion}, 
\begin{eqnarray*}
    \hat{e}_{t+1} & = & \prod_{\tau=0}^t (1-\eta_\tau) \hat{e}_{0} + \sum_{j=0}^t \left(\prod_{\tau=j+1}^t (1-\eta_\tau)\right) {D}_{j+1}  + \sum_{j=0}^t \left(\prod_{\tau=j}^t (1-\eta_\tau)\right) \eta_j e_{j+1},
\end{eqnarray*}
where $D_{t+1} = \frac{1}{n}\sum\limits_{i=1}^n D_{i,t+1}$ and $e_{t+1} = \frac{1}{n}\sum_{i=1}\limits^n e_{i,t+1}$. If $\eta_0 = 1$, then we have 
\begin{eqnarray*}
    \hat{e}_{t+1} & = &  \sum_{j=0}^t \left(\prod_{\tau=j+1}^t (1-\eta_\tau)\right) {D}_{j+1} + \sum_{j=0}^t \left(\prod_{\tau=j}^t (1-\eta_\tau)\right) \eta_j e_{j+1}. 
\end{eqnarray*}
Therefore, 
\begin{eqnarray*}
    \Exp{\norm{\hat{e}_{t+1}}} 
   &  \leq & \Exp{ \norm{ \sum_{j=0}^t \left(\prod_{\tau=j+1}^t (1-\eta_\tau)\right) {D}_{j+1}  } }  + \Exp{ \norm{ \sum_{j=0}^t \left(\prod_{\tau=j}^t (1-\eta_\tau)\right) \eta_j e_{j+1} } } \\
    &\leq & \sqrt{ \Exp{ \sqnorm{ \sum_{j=0}^t \left(\prod_{\tau=j+1}^t (1-\eta_\tau)\right) {D}_{j+1}  } }  } \\
    &&+ \sqrt{ \Exp{ \sqnorm{ \sum_{j=0}^t \left(\prod_{\tau=j}^t (1-\eta_\tau)\right) \eta_j e_{j+1} } } },
\end{eqnarray*} 
where in the last inequality we used Jensen's inequality.

Next, from \Cref{assum:stoc_g_h}, we can prove that $\Exp{D_{t+1}} = 0$, $\Exp{ \inp{D_l}{D_{i}}} = 0$ for $l \neq i$, $\Exp{e_{t+1}}=0$, $\Exp{\inp{e_l}{e_i}}=0$ for $l \neq i$.  Hence, 
\begin{eqnarray*}
    \Exp{\norm{\hat{e}_{t+1}}} 
    & \leq & \sqrt{ \sum_{j=0}^t \left(\prod_{\tau=j+1}^t (1-\eta_\tau)^2\right)  \Exp{ \sqnorm{  {D}_{j+1}  } }  } \\
    && + \sqrt{ \sum_{j=0}^t \left(\prod_{\tau=j}^t (1-\eta_\tau)^2\right) \eta_j^2 \Exp{ \sqnorm{  e_{j+1} } } } \\
    & \leq & \sqrt{ \sum_{j=0}^t \left(\prod_{\tau=j+1}^t (1-\eta_\tau)^2\right)  \frac{1}{n}\sum_{i=1}^n\Exp{ \sqnorm{  \nabla f_i(x^{j};\xi_i^{j+1}) - \nabla f_i(x^{j+1};\xi_i^{j+1})  } }  } \\
    && + \sqrt{ \sum_{j=0}^t \left(\prod_{\tau=j}^t (1-\eta_\tau)^2\right) \eta_j^2 \Exp{ \sqnorm{  e_{j+1} } } }.
\end{eqnarray*}

Next, by \Cref{assum:L_stoc_comp_fcn} and \Cref{assum:stoc_g_h}, and  by the fact that $\Exp{\sqnorm{e_{j+1}}} \leq \frac{\sigma_g^2}{n}$, 
\begin{eqnarray*}
    \Exp{\norm{\hat{e}_{t+1} }} 
  & \leq&  \sqrt{\bar L_{\text{ms}}^2}\sqrt{ \sum_{j=0}^t \left(\prod_{\tau=j+1}^t (1-\eta_\tau)^2\right) \cdot  \gamma_j^2  } + \frac{\sigma_g}{\sqrt{n}}\sqrt{ \sum_{j=0}^t \left(\prod_{\tau=j}^t (1-\eta_\tau)^2\right) \eta_j^2 }\\
  & \leq&  \sqrt{\bar L_{\text{ms}}^2}\sqrt{ \sum_{j=0}^t \left(\prod_{\tau=j+1}^t (1-\eta_\tau)\right) \cdot  \gamma_j^2  } + \frac{\sigma_g}{\sqrt{n}}\sqrt{ \sum_{j=0}^t \left(\prod_{\tau=j}^t (1-\eta_\tau)\right) \eta_j^2 },
\end{eqnarray*} 
where $\bar L_{\text{ms}}^2 = \frac{1}{n}\sum_{i=1}^n L_{\text{ms},i}^2$.

If $\eta_t = \left( \frac{2}{t+2} \right)^{2/3}$ and $\gamma_t = \gamma_0 \left( \frac{2}{t+2} \right)^{2/3}$, then we can prove  that $\eta_0=1$, and from~\Cref{lemma:sum_prod} that 
\begin{eqnarray*}\label{eqn:MVR_whole_v_err}
    \Exp{\norm{v^{t+1} - \nabla f(x^{t+1}) }} 
  \leq \sqrt{\bar L_{\text{ms}}^2}\sqrt{ C(\nicefrac{4}{3},\nicefrac{2}{3}) \frac{\gamma_{t+1}^2}{\eta_{t+1}} } + \frac{\sigma_g}{\sqrt{n}} \sqrt{ C(\nicefrac{4}{3},\nicefrac{2}{3}) \frac{\eta_{t+1}^2}{\eta_{t+1}} }.
\end{eqnarray*}

\subsubsection{Bounding Lyapunov function}
Define our Lyapunov function $V_t = \Delta_{t} + C_{1,t} \cV_t + C_{2,t} \cU_t$ with $C_{1,t} = \frac{2\gamma_t}{1-\sqrt{1-\alpha}}$ and $C_{2,t} = \frac{2\gamma_t \sqrt{1-\alpha}}{1-\sqrt{1-\alpha}}$. Then, 
\begin{eqnarray*}
    \Exp{V_{t+1}} 
    & = & \Exp{\Delta_{t+1} + C_{1,t+1} \cV_t + C_{2,t+1} \cU_{t+1}} \\
    & \leq & \Exp{\Delta_t} - \gamma_t \Exp{\norm{\nabla f(x^t)}} + 2 \gamma_t \Exp{\norm{v^t- \nabla f(x^t)}} + 2\gamma_t \Exp{\cV_t} + \frac{\gamma_t^2L}{2} \\
    && + C_{1,t+1} \Exp{\cV_{t+1}} + C_{2,t+1} \Exp{\cU_{t+1}},
\end{eqnarray*}
where we use \eqref{eq:descent_lemma}.

Next, by the upper-bounds for $\Exp{\cV_{t+1}}$ and  $\Exp{\cU_{t+1}}$, 
\begin{eqnarray*}
    \Exp{ V_{t+1} }
   & \leq & \Exp{\Delta_t} - \gamma_t \Exp{\norm{\nabla f(x^t)}} + 2 \gamma_t \Exp{\norm{v^t- \nabla f(x^t)}} + 2\gamma_t \Exp{\cV_t} + \frac{\gamma_t^2L}{2} \\
    && + C_{1,t+1} (\sqrt{1-\alpha}    \Exp{\cV_{t}} + \sqrt{1-\alpha} \eta_t \Exp{\cU_t} + \sqrt{1-\alpha} \gamma_t \bar L_{\text{ms}} + \sqrt{1-\alpha} \eta_t \sigma_g) \\
    && + C_{2,t+1} ((1-\eta_t) \Exp{\cU_t} + (1-\eta_t)\gamma_t  \bar L_{\text{ms}}  + \eta_t \sigma_g ).
\end{eqnarray*}

Next,  since $\gamma_{t+1} \leq \gamma_t$, we can prove that $C_{1,t+1} \leq C_{1,t}$ and that $C_{2,t+1} \leq C_{2,t}$, and also that 
\begin{eqnarray*}
    2\gamma_t + C_{1,t+1}\sqrt{1-\alpha}  & \leq & 2\gamma_t + C_{1,t}\sqrt{1-\alpha} = C_{1,t} \\
    C_{1,t+1} \sqrt{1-\alpha}\eta_t + C_{2,t+1} (1-\eta_t) & \leq& C_{1,t} \sqrt{1-\alpha}\eta_t + C_{2,t} (1-\eta_t)= C_{2,t}
\end{eqnarray*}
Therefore, 
\begin{eqnarray*}
   \Exp{ V_{t+1} } 
   & \leq & \Exp{V_t} - \gamma_t \Exp{\norm{\nabla f(x^t)}} + 2 \gamma_t \Exp{\norm{v^t- \nabla f(x^t)}}  + \frac{\gamma_t^2L}{2} \\
    && + C_{1,t} (\sqrt{1-\alpha} \gamma_t \bar L_{\text{ms}} + \sqrt{1-\alpha} \eta_t \sigma_g) \\
    && + C_{2,t} \left( (1-\eta_t)\gamma_t  \bar L_{\text{ms}}  + \eta_t \sigma_g  \right) \\
     & \leq & \Exp{V_t} - \gamma_t \Exp{\norm{\nabla f(x^t)}} + 2 \gamma_t \Exp{\norm{v^t- \nabla f(x^t)}}   \\
    && + \gamma_t^2  \left( \frac{4\sqrt{1-\alpha}}{1-\sqrt{1-\alpha}} \bar L_{\text{ms}}  + \frac{L}{2}\right) + \eta_t\gamma_t \frac{4\sqrt{1-\alpha}}{1-\sqrt{1-\alpha}} \sigma_g.
\end{eqnarray*}

Next, from the upper-bound of $\Exp{\norm{v^t- \nabla f(x^t)}}$, 
\begin{eqnarray*}
   \Exp{ V_{t+1} } 
   & \leq & \Exp{V_t} - \gamma_t \Exp{\norm{\nabla f(x^t)}} + 2 \gamma_t \left( \sqrt{\bar L_{\text{ms}}^2}\sqrt{ C(\nicefrac{4}{3},\nicefrac{2}{3}) \frac{\gamma_{t}^2}{\eta_{t}} } + \frac{\sigma_g}{\sqrt{n}} \sqrt{ C(\nicefrac{4}{3},\nicefrac{2}{3}) \eta_t }  \right)  \\
    && + \gamma_t^2  \left( \frac{4\sqrt{1-\alpha}}{1-\sqrt{1-\alpha}} \bar L_{\text{ms}}  + \frac{L}{2}\right) + \eta_t\gamma_t \frac{4\sqrt{1-\alpha}}{1-\sqrt{1-\alpha}} \sigma_g \\
    & \leq & \Exp{V_t} - \gamma_t \Exp{\norm{\nabla f(x^t)}} + \frac{\gamma_t^2}{\sqrt{\eta_t}} \cdot  2 \sqrt{ \bar L^2_{ms} C(\nicefrac{4}{3},\nicefrac{2}{3}) } +  \gamma_t\sqrt{\eta_t} \cdot \frac{2\sigma_g \sqrt{C(\nicefrac{4}{3}, \nicefrac{2}{3})}}{\sqrt{n}}  \\
    && + \gamma_t^2  \left( \frac{4\sqrt{1-\alpha}}{1-\sqrt{1-\alpha}} \bar L_{\text{ms}}  + \frac{L}{2}\right) + \eta_t\gamma_t \frac{4\sqrt{1-\alpha}}{1-\sqrt{1-\alpha}} \sigma_g\\
    &\leq& \Exp{V_t} - \gamma_t \Exp{\norm{\nabla f(x^t)}} +  2 C_1  \bar L^2_{ms}\frac{\gamma_t^2}{\sqrt{\eta_t}}    +  \frac{2 C_1 \sigma_g }{\sqrt{n}}\gamma_t\sqrt{\eta_t} \\
    && + \gamma_t^2  \left( \frac{L}{2} + \frac{4 {\bar L}_{\text{ms}} }{\alpha^2}   \right) + \eta_t\gamma_t \frac{4\sigma_g}{\alpha^2},
\end{eqnarray*} 
where in the last inequality we used $ \frac{\sqrt{1-\alpha}}{1-\sqrt{1-\alpha}} \leq \frac{1}{\alpha^2}$ and denoted $C_1 = \sqrt{ C(\nicefrac{4}{3},\nicefrac{2}{3}) }$.

\subsubsection{Deriving the convergence rate}
By re-arranging the terms and by the telescopic series, 
\begin{eqnarray*}
    \frac{\sum_{t=0}^{T-1} \gamma_t\Exp{\norm{\nabla f(x^t)}}}{\sum_{t=0}^{T-1} \gamma_t}  
    & \leq & \frac{\Exp{V_0} -\Exp{V_T}}{\sum_{t=0}^{T-1}\gamma_t} + \frac{2C_1\sigma_g }{\sqrt{n}} \frac{\sum_{t=0}^{T-1} \gamma_t \sqrt{\eta_t}}{ \sum_{t=0}^{T-1} \gamma_t} +     \frac{4 \sigma_g}{\alpha^2}  \frac{\sum_{t=0}^{T-1} \gamma_t \eta_t}{ \sum_{t=0}^{T-1}\gamma_t}\\
    && + 2 C_1 {\bar L_{\text{ms}}} \frac{\sum_{t=0}^{T-1} \gamma_t^2 \eta_t^{-\nicefrac{1}{2}}}{\sum_{t=0}^{T-1}\gamma_t}  +  \left(\frac{L}{2} + \frac{4 {\bar L}_{\text{ms}}}{\alpha^2} \right) \frac{\sum_{t=0}^{T-1} \gamma_t^2}{ \sum_{t=0}^{T-1} \gamma_t}  .
\end{eqnarray*}

Since $\eta_t = \left( \frac{2}{t+2} \right)^{2/3}$ and $\gamma_t = \gamma_0 \left( \frac{2}{t+2} \right)^{2/3}$, and  for $T \geq 1$ we have 
\begin{eqnarray*}
    \sum_{t=0}^{T-1} \gamma_t &\geq&  T \gamma_{T-1} = \gamma_0T  \left(\frac{2}{T+1}\right)^{\nicefrac{2}{3}} \geq \gamma_0 T^{\nicefrac{1}{3}};\\
    \sum_{t=0}^{T-1} \gamma_t^2 \eta_t^{-\nicefrac{1}{2}}  &=& \gamma_0^2\sum_{t=0}^{T-1}  \left( \frac{2}{t+2} \right)^{\nicefrac{4}{3}} \left( \frac{2}{t+2} \right)^{-\nicefrac{1}{3} } = 2\gamma_0^2 \sum_{t=0}^{T-1} \frac{1}{t+2}   = 2\gamma_0^2 \sum^{T}_{t=1}\frac{1}{1+t}\\
    &\leq& 2\gamma^2_0 \int^T_{1}\frac{1}{1+t}dt = 2\gamma^2_0 \left(\log \left(T+1\right) -\log(2)\right) \leq 2\gamma^2_0 \log \left(T+1\right); \\
    \sum_{t=0}^{T-1} \gamma_t \sqrt{\eta_t} &=& \gamma_0 \sum_{t=0}^{T-1} \left( \frac{2}{t+2} \right)^{\nicefrac{2}{3}} \left( \frac{2}{t+2} \right)^{\nicefrac{1}{3} } =   2\gamma_0 \sum_{t=0}^{T-1}\frac{1}{t+2} \leq 2\gamma_0 \log\left(T+1\right);\\
    \sum_{t=0}^{T-1} \gamma_t^2 &=& \gamma_0^2 \sum_{t=0}^{T-1} \left( \frac{2}{t+2} \right)^{\nicefrac{4}{3}}  = 2^{\nicefrac{4}{3}}\gamma_0^2 \sum_{t=0}^{T-1} \frac{1}{(t+2)^{\nicefrac{4}{3}}} = 2^{\nicefrac{3}{2}}\gamma_0^2 \sum_{t=1}^{T} \frac{1}{(1+t)^{\nicefrac{4}{3}}}\\
    &\leq& 2^{\nicefrac{4}{3}}\gamma_0^2 \int^T_{1}\frac{1}{(1+t)^{\nicefrac{4}{3}}}dt  = 2^{\nicefrac{4}{3}}\cdot 3\gamma^2_0 \left(\frac{1}{\sqrt[3]{2}} - \frac{1}{\sqrt[3]{T+1}}\right) \leq 6\gamma^2_0;\\
    \sum_{t=0}^{T-1} \gamma_t \eta_t &=& \frac{1}{\gamma_0}\sum_{t=0}^{T-1} \gamma_t^2 \leq 6\gamma_0.
\end{eqnarray*}

Therefore,  denoting $\tilde{x}^T$ as a point randomly chosen from $\{x^0,x^1,\ldots,x^{T-1}\}$ with probability $\nicefrac{\gamma_t}{\sum_{t=0}^{T-1}\gamma_t}$ for $t=0,1,\ldots,T-1$, we obtain
\begin{eqnarray*}
    \Exp{\norm{\nabla f(\tilde{x}^T)}}
    & \leq & \frac{\Exp{V_0}}{\gamma_0 T^{\nicefrac{1}{3}}} + \frac{2C_1\sigma_g }{\sqrt{n}} \frac{2\gamma_0 \log\left(T+1\right)}{ \gamma_0 T^{\nicefrac{1}{3}}} +     \frac{4 \sigma_g}{\alpha^2}  \frac{6\gamma_0}{\gamma_0 T^{\nicefrac{1}{3}}}\\
    && + 2 C_1 {\bar L_{\text{ms}}} \frac{2\gamma_0^2 \log\left(T+1\right)}{\gamma_0 T^{\nicefrac{1}{3}}}  +  \left(\frac{L}{2} + \frac{4 {\bar L}_{\text{ms}}}{\alpha^2} \right) \frac{6\gamma_0^2}{\gamma_0 T^{\nicefrac{1}{3}}} \\
    &=& \widetilde{\cO}\left(\frac{\nicefrac{V_0}{\gamma_0} + \sigma_g\left(\nicefrac{1}{\sqrt{n}} +\nicefrac{1}{\alpha^2}\right) +\gamma_0\left(L + {\bar L}_{\text{ms}} + \nicefrac{{\bar L}_{\text{ms}}}{\alpha^2}\right)}{T^{\nicefrac{1}{3}}}\right).
\end{eqnarray*}

\newpage
\section{Deep Learning experiments}
\subsection{Hardware and Datasets.}
To evaluate the performance of the proposed methods in training Deep Neural Networks (DNNs), we utilized the ResNet-18 architecture \citep{he2016deep}. ResNet-18 is a prominent model for image classification, and its architecture is also frequently adapted for tasks such as feature extraction in image segmentation, object detection, image embedding, and image captioning. Our experiments involved training all layers of the ResNet-18 model, corresponding to an optimization problem with $d = 11,173,962$ parameters.

All implementations were developed in PyTorch \citep{paszke2019pytorch}, and experiments were conducted on the CIFAR-10 dataset \citep{krizhevsky2009learning}. Numerical evaluations were performed on a server-grade machine running Ubuntu 18.04 (Linux Kernel v5.4.0). This system was equipped with dual 16-core 3.3 GHz Intel Xeon processors (totaling 32 cores) and four NVIDIA A100 GPUs, each with 40GB of memory.

To simulate a federated learning environment, we adopted a data distribution strategy inspired by \citet{gao2024econtrol}. Specifically, 50\% of the CIFAR-10 dataset was allocated to 10 clients based on class labels, such that data points with the $i$-th label (for $i \in \{0, \dots, 9\}$) were assigned to client $i+1$. The remaining 50\% of the dataset was distributed randomly and uniformly among the clients. Subsequently, each client’s local data was partitioned into a training set (90\%) and a test set (10\%). This partitioning scheme introduces data heterogeneity, a common characteristic of federated settings. For communication compression, we employed the Top-K sparsifier, retaining 10\% of the coordinates (i.e., $\nicefrac{K}{d} = 0.1$).

\subsection{Implementation details.}
We implement both \algname{$\|\text{EF21-HM}\|$} and \algname{$\|\text{EF21-RHM}\|$} in PyTorch, leveraging automatic differentiation to compute Hessian-vector products (HVPs) efficiently. This is achieved via the well-known identity
\begin{align*}
    \nabla^2 f(x;\xi)\,v \;=\; \nabla_x\big\langle \nabla_x f(x;\xi),\, v \big\rangle,
\end{align*}
which allows for the computation of the HVP without materializing the full Hessian matrix. Our implementation first obtains the gradient while constructing the computational graph, and then calls \texttt{torch.autograd.grad} a second time with the displacement vector as input to produce the HVP.

\paragraph{\algname{$\|\text{EF21-HM}\|$}.}
At each iteration $t$ and worker $i$, we form the displacement vector $\Delta^{t+1} = x^{t+1} - x^t$. The stochastic gradient $\nabla f_i(x^{t+1};\xi_i^{t+1})$ and the HVP $\quad\text{and}\quad h_v^{t+1} \;\gets\; \nabla^2 f_i(x^{t+1};\xi_i^{t+1})\,\Delta^{t+1}$ are then computed at the new point $x^{t+1}$ using the same minibatch $\xi_i^{t+1}$.
This is implemented in two stages: (i) a single backward pass with $\texttt{create\_graph=True}$ to obtain the stochastic gradient and retain the graph, followed by (ii) a call to \texttt{autograd.grad} that uses the list of parameter gradients as `outputs` and $\Delta^{t+1}$ as `grad\_outputs` to obtain the HVP. The momentum buffer is subsequently updated using this Hessian correction:
\begin{align*}
    v_i^{t+1} \;=\; (1-\eta_t)\bigl(v_i^t + h_v^{t+1}\bigr) \;+\; \eta_t\, \nabla f_i(x^{t+1};\xi_i^{t+1}).
\end{align*}
This approach has a computational cost equivalent to two backpropagations per minibatch.
\paragraph{\algname{$\|\text{EF21-RHM}\|$}.}

The randomized variant evaluates the HVP at an interpolated point,
\begin{align*}
    \hat{x}^{t+1} \;=\; q_t\,x^{t+1} + (1-q_t)\,x^t, \qquad q_t \sim \mathcal{U}(0,1),
\end{align*}

while the stochastic gradient is still computed at $x^{t+1}$. The implementation proceeds as follows: (i) a first backward pass at $x^{t+1}$ with $\texttt{create\_graph=True}$ is performed to obtain $\nabla f_i(x^{t+1}; \xi_i^{t+1})$; (ii) the model's parameters are temporarily and in-place swapped to $\hat{x}^{t+1}$ within a $\texttt{torch.no\_grad()}$ context; (iii) a forward and backward pass is executed to build the first-order graph at $\hat{x}^{t+1}$; (iv) an \texttt{autograd.grad} HVP call is made at $\hat{x}^{t+1}$ with `grad\_outputs` set to $\Delta^{t+1}$; and finally, (v) the original parameters at $x^{t+1}$ are restored, and the optimizer step is applied with the randomized Hessian-corrected momentum:
\begin{align*}
    v_i^{t+1} \;=\; (1-\eta_t)\bigl(v_i^t + \nabla^2 f_i(\hat{x}^{t+1};\xi_i^{t+1})\,\Delta^{t+1}\bigr) \;+\; \eta_t\, \nabla f_i(x^{t+1};\xi_i^{t+1}).
\end{align*}
This process incurs a cost of approximately three backpropagation-equivalents per minibatch.

\paragraph{\algname{$\|\text{EF21-MVR}\|$}.}
The momentum-with-variance-reduction (MVR) variant, also known as STORM, replaces the Hessian correction with a gradient difference computed on the same minibatch. At iteration $t$ on worker $i$, we compute two gradients using the same data sample $\xi_i^{t+1}$:
\begin{align*}
    \nabla f_i(x^{t+1};\xi_i^{t+1}) \quad\text{and}\quad \nabla f_i(x^t;\xi_i^{t+1}).
\end{align*}
The momentum state is then updated as:
\begin{align*}
    v_i^{t+1} \;=\; (1-\eta_t)\Bigl(v_i^t + \nabla f_i(x^{t+1};\xi_i^{t+1}) - \nabla f_i(x^t;\xi_i^{t+1})\Bigr) \;+\; \eta_t\, \nabla f_i(x^{t+1};\xi_i^{t+1}).
\end{align*}
Our implementation achieves this by: (i) caching the parameters $x^t$ in the optimizer's state; (ii) temporarily swapping the live parameters to $x^t$ under \texttt{torch.no\_grad()}; (iii) calling \texttt{autograd.grad} to obtain $\nabla f_i(x^t;\xi_i^{t+1})$ without modifying the model's `.grad` attributes; (iv) restoring the parameters to $x^{t+1}$; (v) performing a standard \texttt{loss.backward()} call to compute $\nabla f_i(x^{t+1};\xi_i^{t+1})$; and (vi) passing both gradients to the optimizer step. This path costs approximately two backpropagations per minibatch and does not require second-order graph construction.

\paragraph{\algname{$\|\text{EF21-IGT}\|$}.}
Iterative Gradient Transport (IGT) maintains a first-order computational cost by evaluating the stochastic gradient at an extrapolated point. On worker $i$ at iteration $t$, let $\Delta^{t+1} = x^{t+1} - x^t$ be the recent parameter displacement. IGT forms the extrapolated point
\[
    x_{\mathrm{ex}}^{t+1} \;=\; x^{t+1} \;+\; \frac{1-\eta_t}{\eta_t}\,\Delta^{t+1},
\]
and uses $\nabla f_i(x_{\mathrm{ex}}^{t+1};\xi_i^{t+1})$ in the momentum update. Our implementation caches the previous weights $x^t$ in the optimizer state. At each step, it calculates the displacement, temporarily moves the model parameters to the extrapolated point $x_{\mathrm{ex}}^{t+1}$ within a \texttt{no\_grad} context, performs a single backward pass to obtain the extrapolated gradient, restores the original parameters, and then invokes the optimizer step. This entire procedure requires only a single backpropagation per minibatch.

\paragraph{Other Methods.}
The baseline algorithms, including \algname{EF21-SGD}, \algname{EF21-SGDM}, \algname{$\norm{\text{EF21-SGDM}}$}, and \algname{EControl}, all follow a standard training procedure requiring a single first-order backpropagation per minibatch. After computing the loss, a call to \texttt{loss.backward()} is made, followed by the optimizer step. The differences between these methods lie entirely within the optimizer's internal logic for state updates and do not incur additional backpropagation costs.

\subsection{Hyperparameter Tuning.}
We benchmark the proposed algorithms---\algname{$\|\text{EF21-HM}\|$}, \algname{$\|\text{EF21-RHM}\|$}, \algname{$\norm{\text{EF21-MVR}}$}, and \algname{$\norm{\text{EF21-IGT}}$}---against several state-of-the-art error feedback methods: \algname{EF21-SGD} \citep{fatkhullin2025ef21}, \algname{EF21-SGDM} \citep{fatkhullin2023momentum}, \algname{$\norm{\text{EF21-SGDM}}$} \citep{khirirat2024errorfeedbackl0l1smoothnessnormalization}, and \algname{EControl} \citep{gao2024econtrol}. A comprehensive summary of the hyperparameter search space for each method is provided in Table~\ref{table:hyperparameters}.

\paragraph{Momentum Parameter Tuning.}
For the baseline methods \algname{EF21-SGDM} and \algname{EControl}, the momentum parameter $\eta$ was set to a constant value of $0.1$, following the recommendations in their respective original publications. For our proposed \algname{$\|\text{EF21-HM}\|$} and the \algname{$\norm{\text{EF21-SGDM}}$} baseline, we explored both constant $\eta$ values from the set $\{0.01, 0.1, 0.2\}$ and theoretically motivated decreasing schedules. Based on the superior performance of decreasing schedules observed for these variants, our tuning for \algname{$\|\text{EF21-RHM}\|$}, \algname{$\norm{\text{EF21-MVR}}$}, and \algname{$\norm{\text{EF21-IGT}}$} focused exclusively on their theoretically derived schedules. As a practical adaptation to prevent the momentum parameter from diminishing too rapidly in the early stages of training, we update the epoch-dependent momentum schedule $\eta_e$ on a per-epoch basis rather than per-iteration.

\paragraph{Stepsize Tuning.}
For the non-normalized baselines (\algname{EF21-SGD}, \algname{EF21-SGDM}, and \algname{EControl}), we tuned a constant stepsize $\gamma$ from the set $\{1.0, 0.1, 0.05, 0.01, 0.005\}$. For all normalized and second-order methods, we evaluated both a constant stepsize $\gamma \in \{1.0, 0.1, 0.05, 0.01\}$ and an epoch-dependent decreasing schedule of the form $\gamma_e = \gamma_0(e+1)^{-p}$, where the initial learning rate $\gamma_0$ was tuned from the set $\{2.0, 1.0, 0.5, 0.1\}$ and the exponent $p$ was chosen based on our theoretical analysis for each algorithm.

\begin{table}[h!]
\centering
\caption{Summary of hyperparameter tuning search space for all methods.}
\label{table:hyperparameters}
\begin{threeparttable}
\makebox[\linewidth]{%
\small
\begin{tabular}{@{}lll@{}}
\toprule
\textbf{Method} & \textbf{Learning Rate ($\gamma$)} & \textbf{Momentum ($\eta$)} \\
\midrule
\algname{EF21-SGDM}   & Constant: $\gamma \in \{1.0, 0.1, 0.05, 0.01, 0.005\}$ & Constant: $\eta = 0.1$ \\
\algname{EControl}   & Constant: $\gamma \in  \{1.0, 0.1, 0.05, 0.01, 0.005\}$ & Constant: $\eta = 0.1$ \\
\algname{EF21-SGD}    & Constant: $\gamma \in \{1.0, 0.1, 0.05, 0.01, 0.005\}$ & Not Applicable \\
\midrule
\algname{$\norm{\text{EF21-MVR}}$}    & \makecell[l]{Constant: $\gamma \in \{1.0, 0.1, 0.05, 0.01\}$ \\ or Decreasing: $\gamma_e = \gamma_0(e+1)^{-2/3}$} & Decreasing: $\eta_e = (2(e+2)^{-1})^{2/3}$ \\
\algname{$\norm{\text{EF21-RHM}}$} & \makecell[l]{Constant: $\gamma \in \{1.0, 0.1, 0.05, 0.01\}$ \\ or Decreasing: $\gamma_e = \gamma_0(e+1)^{-2/3}$} & Decreasing: $\eta_e = (2(e+2)^{-1})^{2/3}$ \\
\algname{$\norm{\text{EF21-SGDM}}$} & \makecell[l]{Constant: $\gamma \in \{1.0, 0.1, 0.05, 0.01\}$ \\ or Decreasing: $\gamma_e = \gamma_0(e+1)^{-0.75}$} & \makecell[l]{Constant: $\eta \in \{0.01, 0.1, 0.2\}$ \\ or Decreasing: $\eta_e = (2(e+2)^{-1})^{0.5}$} \\
\algname{$\norm{\text{EF21-IGT}}$}   & \makecell[l]{Constant: $\gamma \in \{1.0, 0.1, 0.05, 0.01\}$ \\ or Decreasing: $\gamma_e = \gamma_0(e+1)^{-5/7}$} & Decreasing: $\eta_e = (2(e+2)^{-1})^{4/7}$ \\
\algname{$\norm{\text{EF21-HM}}$} & \makecell[l]{Constant: $\gamma \in \{1.0, 0.1, 0.05, 0.01\}$ \\ or Decreasing: $\gamma_e = \gamma_0(e+1)^{-2/3}$} & \makecell[l]{Constant: $\eta \in \{0.01, 0.1, 0.2\}$ \\ or Decreasing: $\eta_e = (2(e+2)^{-1})^{2/3}$} \\
\bottomrule
\end{tabular}}
\begin{tablenotes}[para,flushleft]
\footnotesize
\item 1. For all decreasing schedules, $\gamma_0$ was tuned from the set $\{2.0, 1.0, 0.5, 0.1\}$;

\item 2. $e$ denotes the epoch index. The base momentum value $\eta_0$ was 1.0 for all decreasing schedules.

\end{tablenotes}
\end{threeparttable}
\end{table}

All methods were trained for a fixed budget of $90$ epochs. Since the per-iteration communication cost is identical for all compared algorithms, the total number of epochs serves as a direct proxy for the total bits communicated. Upon completion of all experimental runs, the optimal hyperparameters (stepsize $\gamma_e$ and momentum parameter $\eta_e$) for each method were selected based on the best validation accuracy achieved and observed stable convergence behavior.
A summary of the selected tuned hyperparameters is provided in \Cref{table:tuned_hyperparameters}, and the best-achieved accuracy metrics for each method are detailed in \Cref{table:accuracy_results}.

\begin{table}[htbp]
\centering
\caption{Summary of the optimally tuned hyperparameters for each method.}
\label{table:tuned_hyperparameters}
\begin{threeparttable}
\makebox[\linewidth]{%
\small
\begin{tabular}{@{}lll@{}}
\toprule
\textbf{Method} & \textbf{Learning Rate ($\gamma$)} & \textbf{Momentum ($\eta$)} \\ \midrule
\algname{EF21-SGDM} & Constant $\gamma = 0.1$ & Constant $\eta = 0.1$ \\
\algname{EControl} & Constant $\gamma = 1.0$ & Constant $\eta = 0.1$ \\
\algname{EF21-SGD} & Constant $\gamma = 1.0$ & --- \\
\midrule
\algname{$\norm{\text{EF21-MVR}}$} & Constant $\gamma = 0.1$ & Decreasing $\eta_e = (\nicefrac{2}{e+2})^{2/3}$ \\
\algname{$\|\text{EF21-RHM}\|$} & Constant $\gamma = 0.1$ & Decreasing $\eta_e = (\nicefrac{2}{e+2})^{2/3}$ \\
\algname{$\norm{\text{EF21-SGDM}}$} & Constant $\gamma = 0.1$ & Decreasing $\eta_e = (\nicefrac{2}{e+2})^{0.5}$ \\
\algname{$\norm{\text{EF21-IGT}}$} & Constant $\gamma = 0.1$ & Decreasing $\eta_e = (\nicefrac{2}{e+2})^{0.57}$ \\
\algname{$\|\text{EF21-HM}\|$} & Constant $\gamma = 0.1$ & Decreasing $\eta_e = (\nicefrac{2}{e+2})^{2/3}$ \\ \bottomrule
\end{tabular}}
\begin{tablenotes}[para,flushleft]
\footnotesize
\item 1. $e$ denotes the epoch index. The base momentum value $\eta_0$ was 1.0 for all decreasing schedules.
\end{tablenotes}
\end{threeparttable}
\end{table}

\begin{table}[htbp]
\centering
\caption{Best performance metrics achieved by each method when training ResNet-18 on CIFAR-10, sorted by validation accuracy. The top two results are highlighted in \textbf{bold}.}
\label{table:accuracy_results}
\begin{threeparttable}
\resizebox{\textwidth}{!}{%
\begin{tabular}{@{}lccccc@{}}
\toprule
\textbf{Method} & \textbf{\makecell{Best Val.\\Accuracy (\%)}} & \textbf{\makecell{Corr. Test\\Accuracy (\%)}} & \textbf{\makecell{Epoch of\\Best}} & \textbf{\makecell{GPU Time\\to Best}} & \textbf{\makecell{Wall Time\\to Best}} \\ \midrule
\algname{EF21-SGDM}          & $74.66$          & $73.53$          & $76$ & 0h 35m 59s & 0h 35m 59s \\
\algname{EControl}           & $76.38$          & $74.49$          & $79$ & 0h 38m 06s & 0h 38m 06s \\
\algname{EF21-SGD}           & $77.50$          & $75.90$          & $69$ & 0h 30m 35s & 0h 30m 36s \\
\midrule
\algname{$\norm{\text{EF21-MVR}}$}  & $79.30$          & $78.77$          & $72$ & 0h 44m 33s & 0h 44m 33s \\
\algname{$\|\text{EF21-RHM}\|$} & $81.48$          & $80.54$          & $66$ & 1h 35m 10s & 1h 35m 10s \\
\algname{$\norm{\text{EF21-SGDM}}$} & $82.66$          & $81.70$          & $79$ & 0h 37m 54s & 0h 37m 55s \\
\textbf{\algname{$\norm{\text{EF21-IGT}}$}}  & $\mathbf{83.48}$          & $\mathbf{81.77}$          & $\mathbf{70}$ & \textbf{0h 39m 03s} & \textbf{0h 39m 03s} \\
\textbf{\algname{$\|\text{EF21-HM}\|$}}  & $\mathbf{84.32}$ & $\mathbf{83.22}$ & $\mathbf{75}$ & \textbf{1h 37m 55s} & \textbf{1h 37m 56s} \\ \bottomrule
\end{tabular}}

\begin{tablenotes}[para,flushleft]
\footnotesize
    \item[*] \textbf{GPU Time to Best} has the following precise definition. Let $e^\star$ be the (first) epoch index that attains the maximal validation accuracy. For each epoch $e\in\{1,\dots,e^\star\}$ we create CUDA events $\mathrm{start}_e$ and $\mathrm{end}_e$ via \texttt{torch.cuda.Event(enable\_timing=True)}. We record $\mathrm{start}_e$ immediately before the first training minibatch of epoch $e$, and we record $\mathrm{end}_e$ \emph{after} all GPU work for that epoch has finished (training \emph{and} validation \emph{and} test). We then call \texttt{torch.cuda.synchronize()} and compute the per epoch device time
    $\tau_e \;=\; \nicefrac{\texttt{start\_e.elapsed\_time(end\_e)}}{1000}\quad\text{seconds}$.
    The reported quantity is the partial sum
    $T^{\mathrm{GPU}}_{\le e^\star} \;=\; \sum_{e=1}^{e^\star} \tau_e,$
    which accumulates all measured device time up to and including the best epoch $e^\star$. \textbf{Wall Time to Best} is the corresponding real elapsed time measured with \texttt{time.time()}, reset immediately after the epoch-0 snapshot; at epoch $e^\star$ it equals the cumulative \texttt{wall\_seconds} logged by the code. In our setup, runs execute one at a time on a single GPU, and each epoch’s CUDA events bracket \emph{all} GPU work with an explicit \texttt{torch.cuda.synchronize()}, so wall time closely matches GPU time (CPU-only overheads are comparatively small).
\end{tablenotes}

\end{threeparttable}
\end{table}


\subsection{Performance Comparison}

As the experimental results in \Cref{table:accuracy_results} and the convergence plots indicate, our proposed momentum variants demonstrate a significant performance improvement over existing baselines. The Hessian-corrected method, \algname{$\norm{\text{EF21-HM}}$}, consistently achieves the highest final test accuracy and the lowest test loss, underscoring the benefits of incorporating second-order information. Its advantage is clearly illustrated in \Cref{fig:performance_plots_epoch} and \Cref{fig:performance_plots_time_full}, where it not only reaches a superior final state but also maintains a stable convergence trajectory.
While \algname{$\norm{\text{EF21-HM}}$} sets the performance ceiling, our other proposed methods—\algname{$\norm{\text{EF21-RHM}}$}, \algname{$\norm{\text{EF21-MVR}}$}, and \algname{$\norm{\text{EF21-IGT}}$}—also offer substantial gains. 

To provide a comprehensive analysis, we present the results from two perspectives: convergence per epoch (\Cref{fig:performance_plots_epoch}), which measures sample efficiency, and convergence in wall-clock time (\Cref{fig:performance_plots_time_full} and \Cref{fig:performance_plots_time_truncated}), which measures computational efficiency.

\paragraph{Key Observations}
\begin{enumerate}
    \item \textbf{Epoch-Based Performance (\Cref{fig:performance_plots_epoch}):} When measured per epoch, the proposed second-order and advanced first-order methods are highly effective. \algname{$\norm{\text{EF21-HM}}$} exhibits the best overall performance, achieving the highest test accuracy and lowest loss. However, \algname{$\norm{\text{EF21-RHM}}$}, \algname{$\norm{\text{EF21-IGT}}$}, and \algname{$\norm{\text{EF21-MVR}}$} also demonstrate superior sample efficiency, clearly outperforming all existing baselines and reaching a higher-quality solution within the 90-epoch budget.

    \item \textbf{Wall-Clock Time Performance (\Cref{fig:performance_plots_time_full}):} This view highlights the trade-off between per-iteration cost and convergence speed. As the training was run for a fixed number of epochs, methods with higher computational costs, such as \algname{$\norm{\text{EF21-HM}}$} and \algname{$\norm{\text{EF21-RHM}}$}, naturally take longer to complete the full training schedule. This aligns with the per-epoch costs reported in the section \ref{sec:gpu}. While they achieve the best final results, their time-to-solution may be longer.

    \item \textbf{Time-to-Solution Analysis (\Cref{fig:performance_plots_time_truncated}):} By truncating the timeline to the point where the fastest methods complete, we can directly compare their time-to-solution efficiency. A key finding here is that \algname{$\norm{\text{EF21-IGT}}$}, despite being a first-order method, achieves a test accuracy and loss trajectory that is remarkably competitive with the much more expensive \algname{$\norm{\text{EF21-HM}}$}. It rapidly converges to a high-accuracy region, making it a highly efficient choice. Similarly, \algname{$\norm{\text{EF21-MVR}}$} also demonstrates strong performance in this view, positioning it as another excellent option that balances computational cost and convergence speed.
\end{enumerate}

\begin{figure}[htbp]
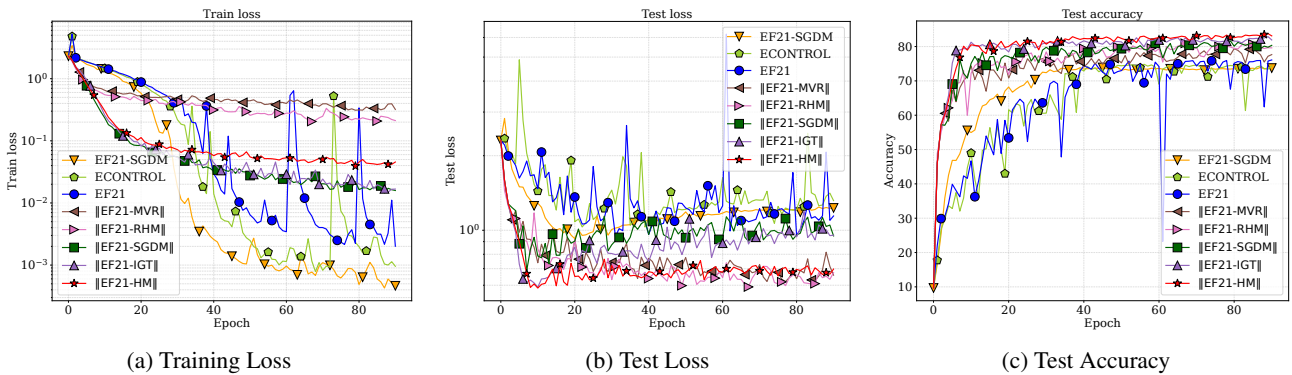

    \centering
    \begin{subfigure}{0.32\textwidth}
        \includegraphics[width=\linewidth]{plots/resnet18_cifar10/train_loss_epoch.pdf}
        \caption{Training Loss}
        \label{fig:train_loss_epoch}
    \end{subfigure}
    \hfill 
    \begin{subfigure}{0.32\textwidth}
        \includegraphics[width=\linewidth]{plots/resnet18_cifar10/test_loss_epoch.pdf}
        \caption{Test Loss}
        \label{fig:test_loss_epoch}
    \end{subfigure}
    \hfill 
    \begin{subfigure}{0.32\textwidth}
        \includegraphics[width=\linewidth]{plots/resnet18_cifar10/test_acc_epoch.pdf}
        \caption{Test Accuracy}
        \label{fig:test_acc_epoch}
    \end{subfigure}
    \caption{Performance comparison of all methods on CIFAR-10 with ResNet-18, plotted as a function of epochs. The proposed momentum variants, particularly \algname{$\norm{\text{EF21-HM}}$} and \algname{$\norm{\text{EF21-IGT}}$}, show superior sample efficiency.}
    \label{fig:performance_plots_epoch}
\end{figure}

\begin{figure}[htbp]
    \centering
    \begin{subfigure}{0.32\textwidth}
        \includegraphics[width=\linewidth]{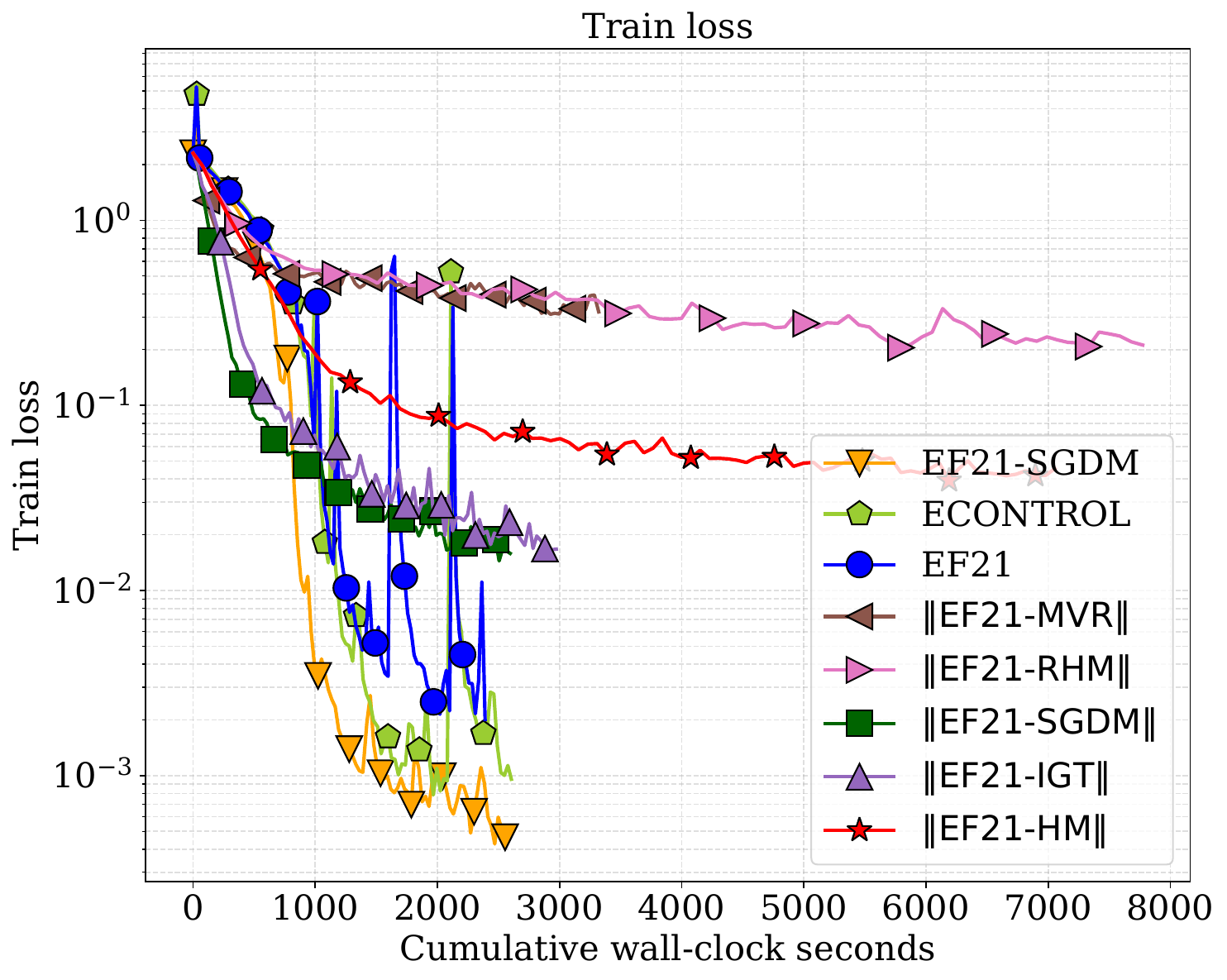}
        \caption{Training Loss}
        \label{fig:train_loss_time_full}
    \end{subfigure}
    \hfill 
    \begin{subfigure}{0.32\textwidth}
        \includegraphics[width=\linewidth]{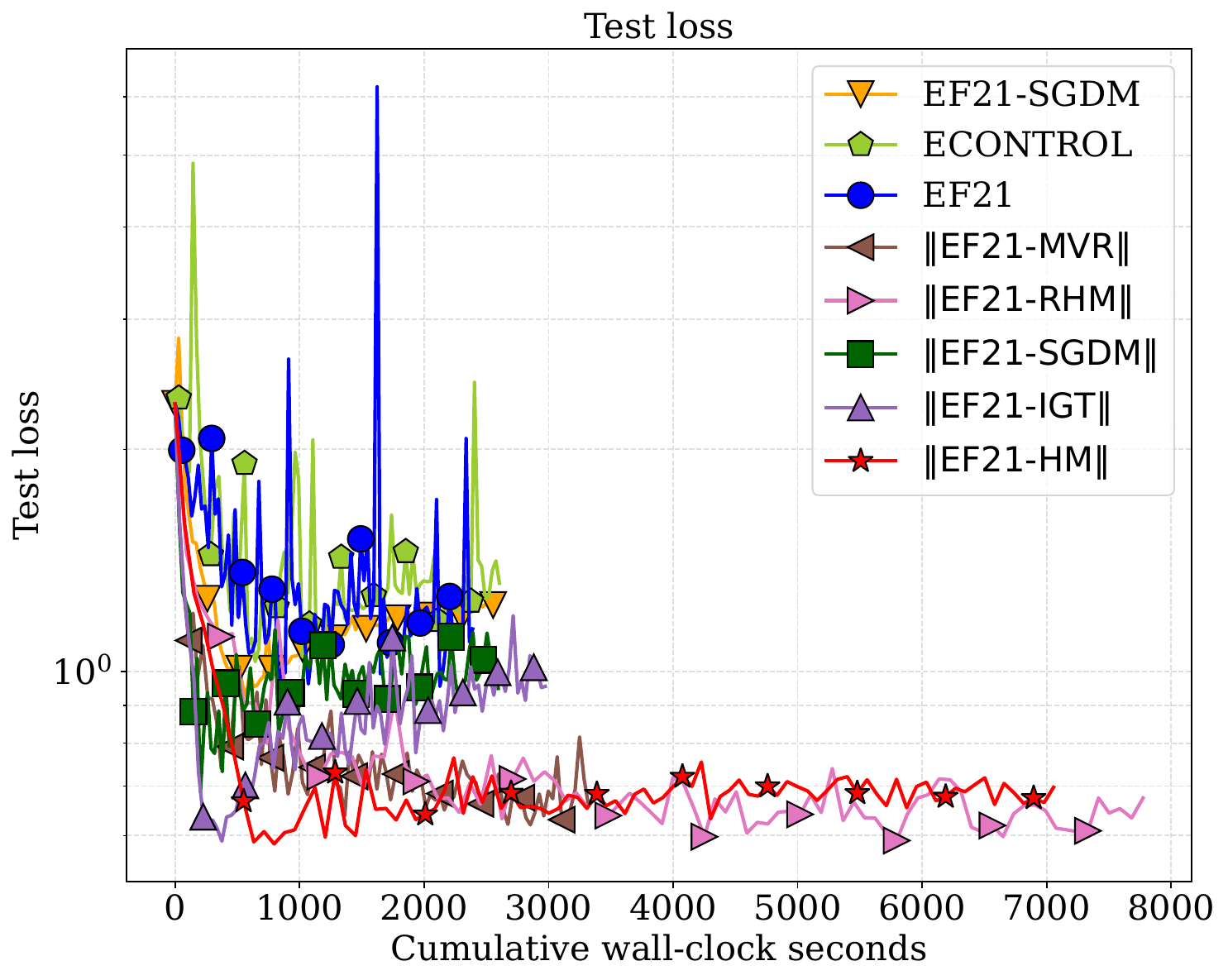}
        \caption{Test Loss}
        \label{fig:test_loss_time_full}
    \end{subfigure}
    \hfill 
    \begin{subfigure}{0.32\textwidth}
        \includegraphics[width=\linewidth]{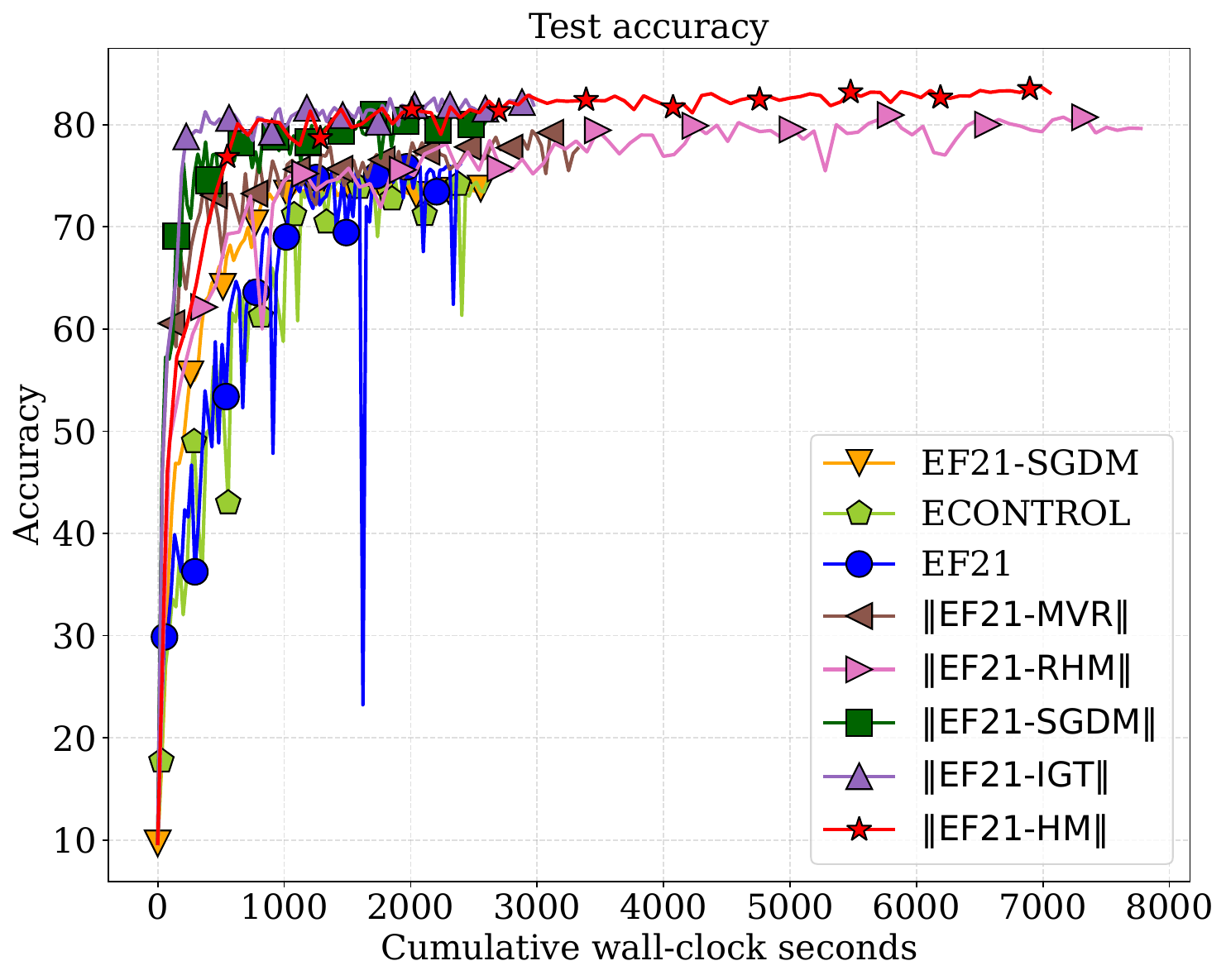}
        \caption{Test Accuracy}
        \label{fig:test_acc_time_full}
    \end{subfigure}
    \caption{Performance comparison as a function of cumulative wall-clock seconds over the full training duration. Methods with higher per-epoch costs take longer to complete the 90-epoch training schedule.}
    \label{fig:performance_plots_time_full}
\end{figure}

\begin{figure}[htbp]
    \centering
    \begin{subfigure}{0.32\textwidth}
        \includegraphics[width=\linewidth]{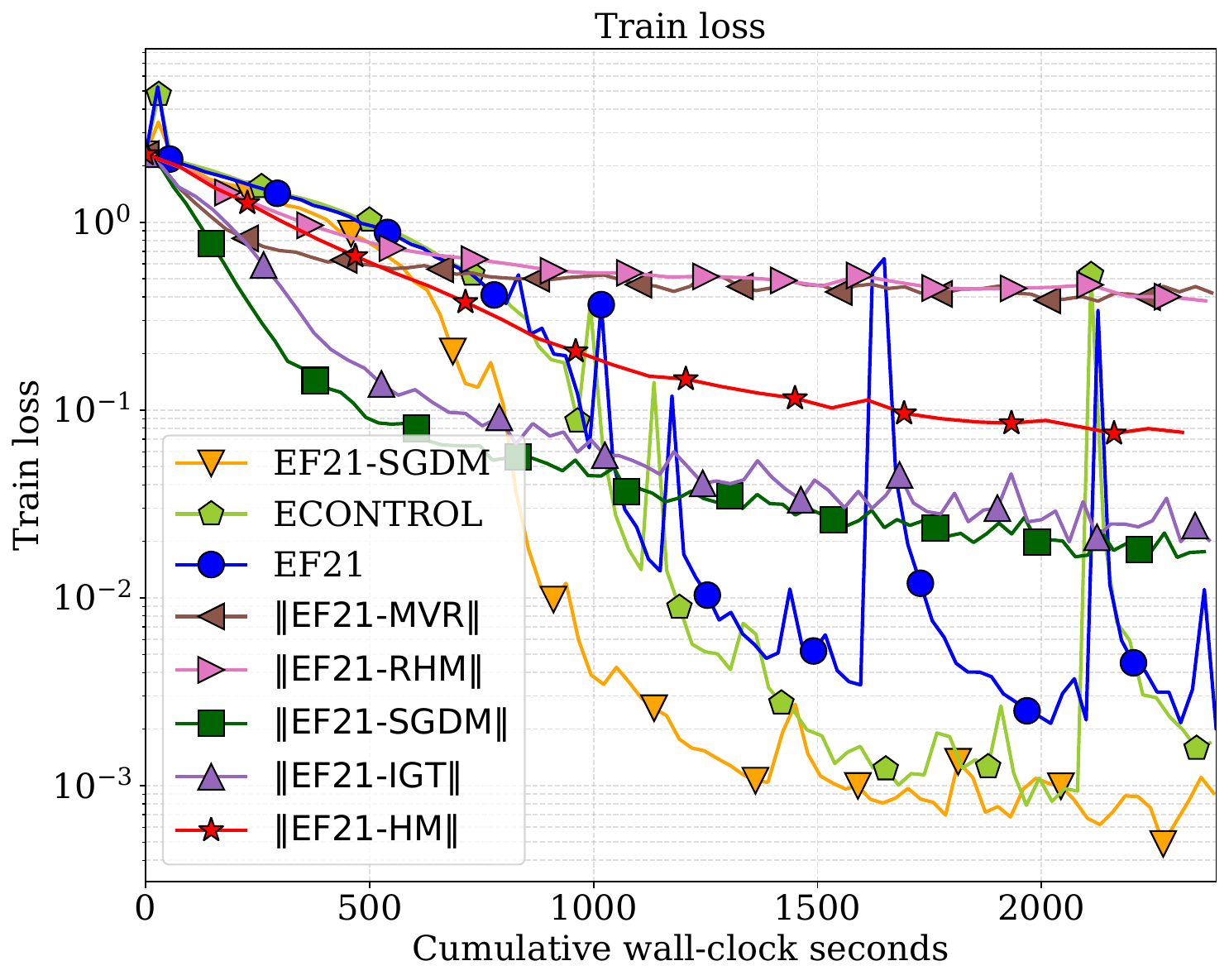}
        \caption{Training Loss}
        \label{fig:train_loss_time_trunc}
    \end{subfigure}
    \hfill 
    \begin{subfigure}{0.32\textwidth}
        \includegraphics[width=\linewidth]{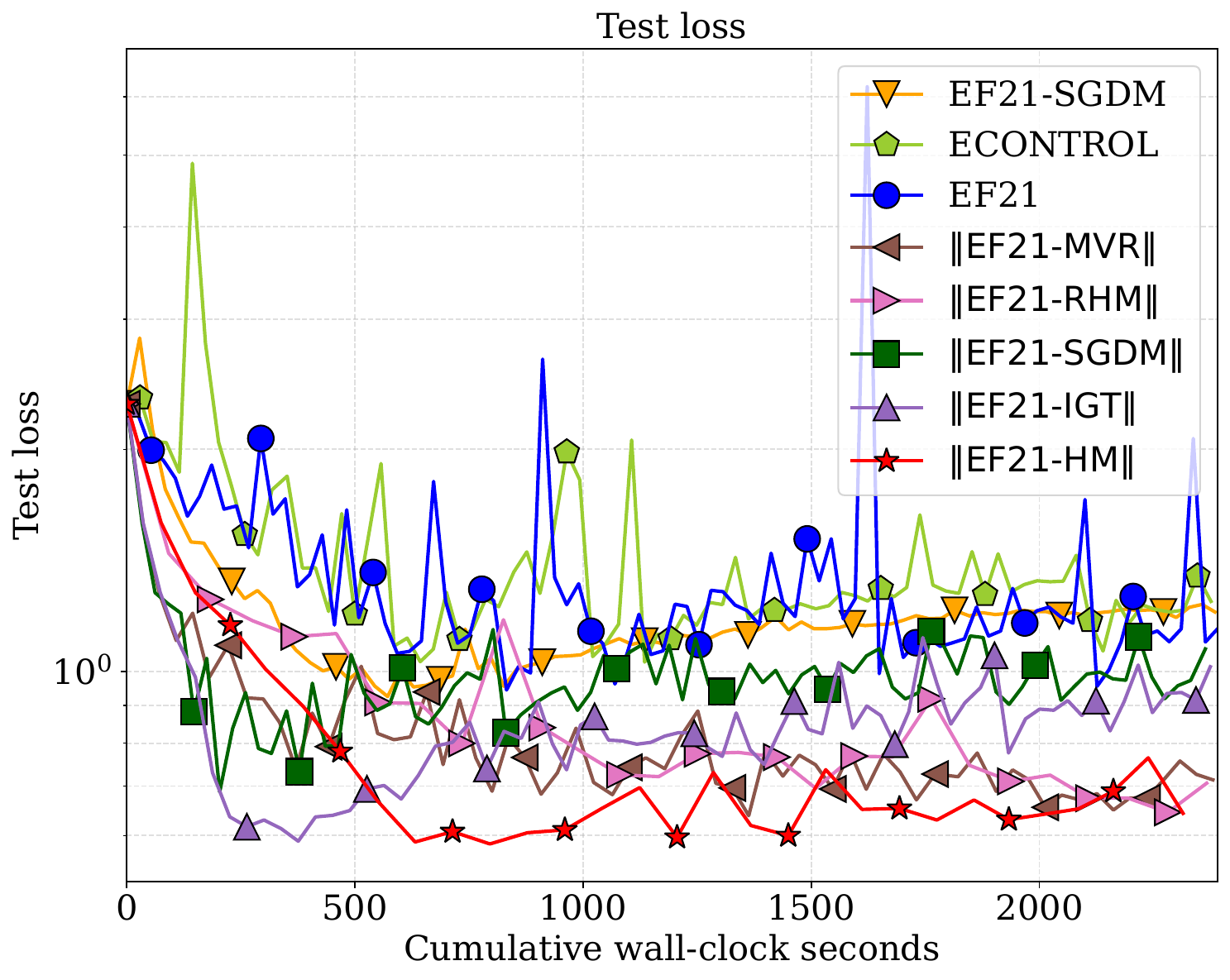}
        \caption{Test Loss}
        \label{fig:test_loss_time_trunc}
    \end{subfigure}
    \hfill 
    \begin{subfigure}{0.32\textwidth}
        \includegraphics[width=\linewidth]{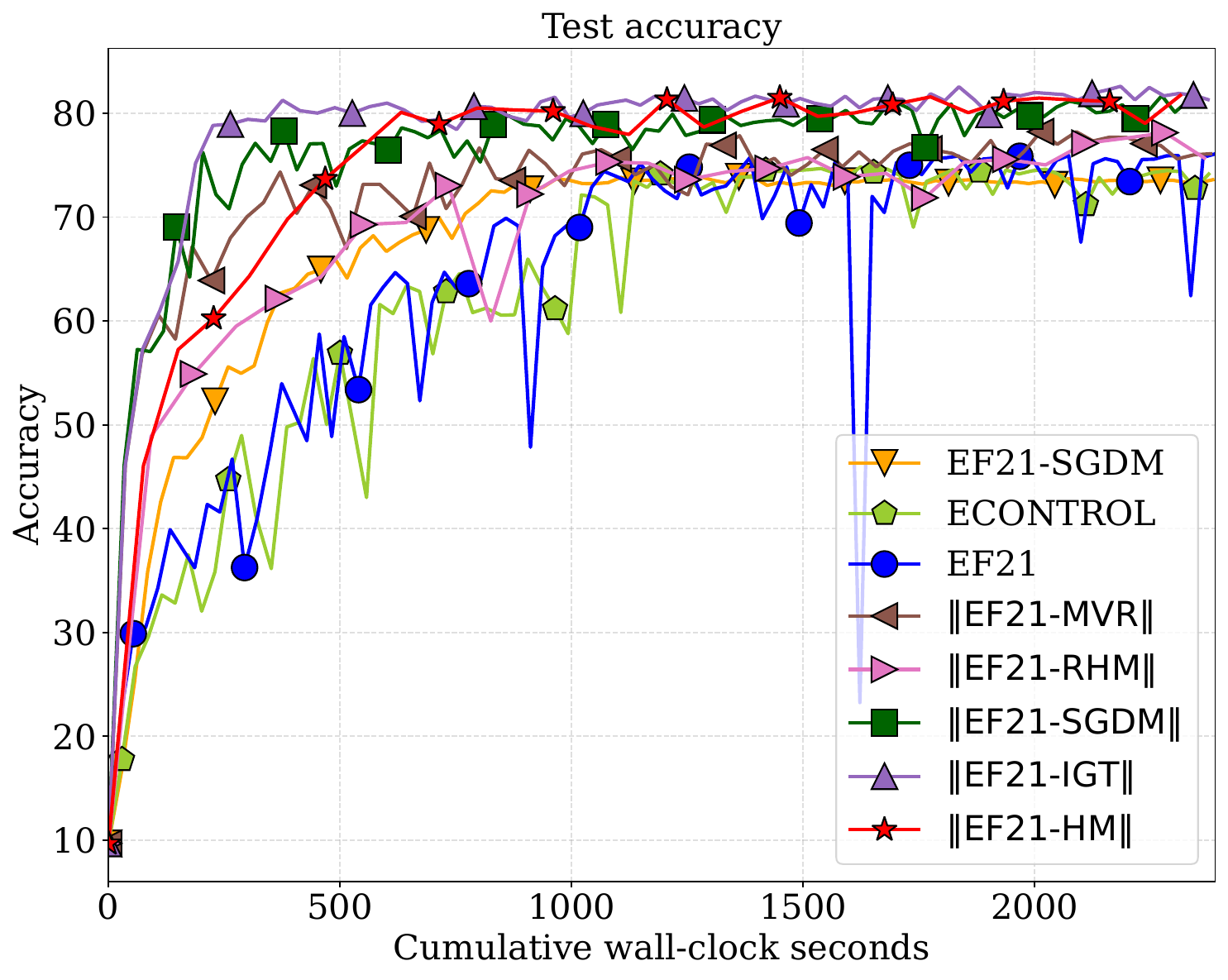}
        \caption{Test Accuracy}
        \label{fig:test_acc_time_trunc}
    \end{subfigure}
    \caption{Time-to-solution performance comparison, with the timeline truncated to the completion time of the fastest methods. This view highlights that \algname{$\norm{\text{EF21-IGT}}$} achieves a convergence speed and accuracy comparable to the much more costly \algname{$\norm{\text{EF21-HM}}$}.}
    \label{fig:performance_plots_time_truncated}
\end{figure}

\subsection{Comparison of Per-Epoch GPU Cost}\label{sec:gpu}
To assess the practical overhead of each method, we report the average GPU time per epoch and normalize these values relative to \algname{$\norm{\text{EF21-SGDM}}$}, which serves as our first-order baseline.

\begin{table}[h!]
\centering
\small
\caption{Average per-epoch GPU runtime and relative cost compared to the baseline, sorted by final validation accuracy.}
\begin{tabular}{@{}lcc@{}}
\toprule
\textbf{Method} & \textbf{Mean sec/epoch} & \textbf{Relative to \algname{$\norm{\text{EF21-SGDM}}$}} \\
\midrule
\algname{EF21-SGDM}          & $28.373$  & $0.984\times$ \\
\algname{EControl}           & $28.943$  & $1.004\times$ \\
\algname{EF21-SGD}           & $26.553$  & $0.921\times$ \\
\midrule
\algname{$\norm{\text{EF21-MVR}}$}  & $36.892$  & $1.279\times$ \\
\algname{$\|\text{EF21-RHM}\|$} & $86.380$  & $2.996\times$ \\
\algname{$\norm{\text{EF21-SGDM}}$} & $28.835$  & $1.000\times$ \\
\algname{$\norm{\text{EF21-IGT}}$}  & $33.026$  & $1.145\times$ \\
\algname{$\|\text{EF21-HM}\|$}  & $78.390$  & $2.719\times$ \\
\bottomrule
\end{tabular}
\end{table}

We note three key practical observations regarding the empirical runtime costs:
\begin{enumerate}
    \item \textbf{Why \algname{$\norm{\text{EF21-HM}}$}'s cost exceeds the idealized $2\times$ baseline.}
    Although \algname{$\norm{\text{EF21-HM}}$} involves two automatic differentiation passes, its empirical cost of $2.72\times$ is notably higher than the theoretical $2\times$ baseline. This discrepancy arises because the two passes are not computationally equivalent:
    \begin{enumerate}[label=(\roman*), itemsep=0pt, topsep=3pt]
        \item \textbf{The Cost of Graph Creation:} The primary source of the additional overhead is the first pass, which calls \texttt{loss.backward(create\_graph=True)}. This operation is significantly more expensive than a standard backward pass. It not only computes the gradients but also constructs and retains a detailed computational graph that includes the gradient operations themselves. This process requires keeping intermediate activations in memory, leading to higher memory consumption and greater computational work, making this single step substantially more costly than a standard $1\times$ backpropagation.
        \item \textbf{The HVP Cost:} The second pass, the \texttt{autograd.grad} call for the HVP, traverses this newly created, more complex graph to compute the second-order information. While this VJP/JVP chain is efficient, it still represents a full computational pass that adds to the total time.
    \end{enumerate}
    The sum of these two steps—one expensive graph-creating backward pass and one HVP pass—results in a total cost greater than two standard backpropagations. This combined cost is then partially mitigated by factors like fixed epoch overheads (data loading, validation) and GPU optimizations (caching, kernel fusion), leading to the final observed slowdown of approximately $2.72\times$.

    \item \textbf{Why \algname{$\norm{\text{EF21-RHM}}$}'s cost aligns with its $3\times$ backprop count.}
    In contrast, the cost of \algname{$\norm{\text{EF21-RHM}}$} ($\approx 2.996\times$) aligns more closely with its theoretical $3\times$ cost. Its three backprop-equivalents consist of one expensive graph-creating pass and two additional passes for the HVP at a different point. The costs in this case are more additive, and the same mitigating factors (fixed overheads, GPU optimizations) explain why the final result is slightly below, but very close to, a perfect $3\times$ slowdown.

    \item \textbf{Why \algname{$\norm{\text{EF21-MVR}}$}'s cost is significantly lower than the idealized $2\times$ baseline.}
    \algname{$\norm{\text{EF21-MVR}}$} is faster than its theoretical $2\times$ cost (empirically $1.32\times$) precisely because it avoids the most expensive operations used in the Hessian-based methods. Its efficiency stems from several factors:
    \begin{enumerate}[label=(\roman*), itemsep=0pt, topsep=3pt]
        \item \textbf{First-Order Operations Only:} Crucially, neither of its two backpropagation-equivalents uses the costly \texttt{create\_graph=True} flag. It relies on two standard, efficient first-order automatic differentiation calls.
        \item \textbf{Efficient Implementation:} The \texttt{autograd.grad} call used for the previous-point gradient is slightly cheaper than a full \texttt{loss.backward()} as it does not need to populate the `.grad` attributes of the model parameters.
        \item \textbf{Minibatch Reuse:} Reusing the same data batch for both gradient computations eliminates the overhead associated with additional data loading and host-device transfers.
    \end{enumerate}
    These efficiencies, combined with the standard fixed per-epoch costs, result in the observed sub-$2\times$ slowdown of approximately $1.32\times$.
\end{enumerate}













\end{document}

%% file: math_commands.tex
\usepackage{amsmath,amsfonts,bm}

\def\eqref#1{equation~\ref{#1}}

\def\1{\bm{1}}

\def\rb{{\textnormal{b}}}

\def\rt{{\textnormal{t}}}

\DeclareMathAlphabet{\mathsfit}{\encodingdefault}{\sfdefault}{m}{sl}
\SetMathAlphabet{\mathsfit}{bold}{\encodingdefault}{\sfdefault}{bx}{n}

\newcommand{\R}{\mathbb{R}}

